\documentclass{amsart}
\usepackage{amsmath}
  \usepackage{paralist}
  \usepackage{graphics} %% add this and next lines if pictures should be in esp format
  \usepackage{epsfig} %For pictures: screened artwork should be set up with an 85 or 100 line screen
\usepackage{graphicx}  %\usepackage{epstopdefinition}%This is to transfer .eps figure to .pdefinition figure; please compile your paper using PdefinitionLeTex or PdefinitionTeXify.
 \usepackage[colorlinks=true]{hyperref}

\usepackage{amssymb}
\usepackage{latexsym}
\usepackage{mathabx}
\usepackage{cancel}
\usepackage{tikz,tikz-cd}
\usepackage{quiver}
\newcommand{\lie}[1]{\mathfrak{#1}}

\usepackage[all]{xy}
\newdir{ >}{{}*!/-5pt/@{>}}

\usepackage{xcolor}

\usepackage{xspace}

\usepackage{mathrsfs}
\newcommand{\inv}{^{-1}}

\newcommand{\tH}{\text{H}}

\newcommand{\ov}{\overline}		
\usepackage{amsmath,amssymb,amsthm}	
\newcommand{\add}[1]{\mathbin{\lower 5pt%
    \hbox{${\stackrel{\textstyle +}{\scriptscriptstyle #1}}$}}}
\renewcommand{\phi}{\varphi}
\newcommand{\chigh}{{\raise1.5pt\hbox{$\chi$}}}

\usepackage{stmaryrd}
\usepackage{wasysym}

\DeclareMathOperator{\pr}{pr}
\DeclareMathOperator{\Ppr}{Pr}

\newcommand{\duer}{%
\mathbin{\raisebox{3pt}{\varhexstar}\kern-3.70pt{\rule{0.15pt}{4pt}}}\,}

\newcommand{\id}{\text{\upshape id}}

\newcommand{\R}{\mathbb{R}}

\newcommand{\N}{\mathbb{N}}

\newcommand{\dr}{\mathbf{d}}

\newcommand{\an}[1]{\arrowvert_{#1}}
\newcommand{\mx}{\mathfrak{X}}

\newcommand{\thbr}[2]%
{\rule[-1pt]{1pt}{10pt}\hspace{2pt} #1,\, #2\hspace{1pt}\rule[-1pt]{1pt}{10pt}}

\newcommand{\newpair}[2]%
{\talloblong\hspace{1pt} #1,\, #2\hspace{0.5pt}\talloblong}

\allowdisplaybreaks[1]

\usepackage{array}

\usepackage{subfig}
\usepackage{caption}
\newcommand{\ldr}[1]{{{\pounds}}_{#1}}

\newcommand{\act}{\mathbin{\hbox{$<\kern-.4em\mapstochar\kern.4em$}}}
\newcommand{\ract}{\mathbin{\hbox{$\mapstochar\kern-.3em>$}}}

\usepackage{url}
\newcommand{\epf}{\hspace*{\fill}\lower0pt\hbox{$\Box$}\medskip} 

\hypersetup{urlcolor=blue, citecolor=red}

\newtheorem{theorem}{Theorem}[section]
\newtheorem{corollary}[theorem]{Corollary}
\newtheorem*{main}{Main Theorem}
\newtheorem*{theorem*}{Theorem}
\newtheorem*{lemma*}{Lemma}
\newtheorem*{proposition*}{Proposition}
\newtheorem*{corollary*}{Corollary}
\newtheorem{lemma}[theorem]{Lemma}
\newtheorem{proposition}[theorem]{Proposition}

\theoremstyle{definition}
\newtheorem{definition}[theorem]{Definition}
\newtheorem{remark}[theorem]{Remark}
\newtheorem{example}[theorem]{Example}

\newtheorem{ex}[theorem]{Examples}

\title[] %Use the shortened version of the full title
      {On the homotopy invariance of the twisted Lie algebroid cohomology.}

\author[M.~Jotz and R.~Marchesini$^*$]{}

\subjclass{Primary: 
58A12,  %	de Rham theory in global analysis 
57R99, %  	None of the above, but in this section (Differential topology)
	17B56, %Cohomology of Lie (super)algebras
	53C05; %  	Connections (general theory)
Secondary: 
57R55, %  	Differentiable structures in differential topology
53D17, % 	Poisson manifolds; Poisson groupoids and algebroids
	53C12, %  	Foliations (differential geometric aspects)
	53C15. %  	General geometric structures on manifolds (almost complex, almost product structures, etc.)
}

 \email{madeleine.jotz@uni-wuerzburg.de}
\email{rosa.marchesini@mathematik.uni-goettingen.de}
\thanks{This project was partially funded by the RTG 2491 in G\"ottingen.}

\thanks{$^*$Corresponding author}

\begin{document}
\begin{abstract}
Twisted Lie algebroid cohomologies, i.e.~with values in representations, are shown to be Lie algebroid homotopy-invariant. Several important classes of examples are discussed. As an application, a generalized version of the Poincar\'e lemma is given in the transitive case. Together with the Mayer-Vietoris theorem, which holds in this more general context as well, this leads to a K\"unneth formula for the cohomologies of Lie algebroids with values in representations. In particular, this comprehensive paper gives a systematic way to compute explicitly the twisted Lie algebroid cohomologies of some transitive and almost transitive Lie algebroids.
\end{abstract}

\maketitle

% Enter the first author's name and address:
\centerline{\scshape M.~Jotz}
\medskip
{\footnotesize
% please put the address of the first author
 \centerline{Institut für Mathematik}
   \centerline{Julius-Maximilians-Universit\"at W\" urzburg}
   \centerline{Germany}
} % Do not forget to end the {\footnotesize by the sign }

\medskip

\centerline{\scshape R.~Marchesini$^*$}
\medskip
{\footnotesize
 % please put the address of the second  and third author

\centerline{Mathematisches Institut}
    \centerline{Georg-August-Universit\"at G\"ottingen}
   \centerline{Germany}
}

\bigskip

\tableofcontents

\section{Introduction}

 The extensively studied de Rham, Chevalley-Eilenberg, BRST, Poisson, Foliated and Equivariant de Rham cohomologies can be treated at once as special cases of  Lie algebroid cohomology, which was introduced by Pradines in 1967 \cite{Pradines67}. This generality of the Lie algebroid cohomology brings pathologies that make it difficult to control and compute.
 Lie algebroid cohomology has been studied for a few decades now -- the latest references on the subject include \cite{Kub02},  \cite{Cr03}, \cite{KuMish04}, \cite{Mackenzie05}, \cite{Mish11}, \cite{Bal12} , \cite{Kub13}, \cite{Br17}, \cite{Wal23} and~\cite{MaSc24}.
 Lie algebroid cohomology contains obstructions for the existence of geometric objects such as flat connection and infinitesimal ideal systems (see e.g.~\cite{Jotz22}). The more general \emph{twisted} Lie algebroid cohomologies provide further geometric information. For instance, the second cohomology vector space $\tH^2_{\nabla}(A,E)$ is in natural bijection with the set of equivalence classes of extensions of the Lie algebroid $A$ by a vector bundle $E$ whose induced representation is $\nabla$ (see \cite[Proposition 7.1.13]{Mackenzie05} and also \cite[XIV.5]{CaEi56} for the original case of Lie algebras).

  The main goal of this paper is a systematic study of %an appropriate notion of 
  homotopy of Lie algebroids
   and of its relations with further classical tools of differential topology adapted to twisted Lie algebroid cohomology.
  \bigskip

 A Lie algebroid is a smooth vector bundle $A \to M$ with a Lie bracket on its space of sections, that is \emph{anchored} by the geometry of the basis manifold $M$, in the sense that a vector bundle homomorphism $\rho\colon A\to TM$ over the identity on $M$ allows for differentiation of smooth functions on $M$ \emph{along} sections of $A$; an operation that is compatible with the Lie bracket on $\Gamma(A)$ (see Definition \ref{def:LA}).
 
The tangent bundles of smooth manifolds on the one hand, and Lie algebras on the other hand, are prototypes of Lie algebroids.
 The de Rham cohomology of a smooth manifold is defined by the geometry of its tangent bundle. Precisely, given a smooth manifold $M$, its de Rham differential is defined on the exterior algebra of sections of $T^*M$. Similarly, the Chevalley-Eilenberg differential associated to a Lie algebra $\mathfrak g$ lives on $\bigwedge^\bullet \mathfrak g^*$.
 In general, the Lie algebroid differential of a Lie algebroid $A\to M$ is a differential operator on $\Gamma(\bigwedge^\bullet A^*)$, that is defined 
 in the same manner as the de Rham differential, but using the anchor $\rho\colon A\to TM$ in order to differentiate smooth functions on $M$ along sections of $A$.
 The resulting cohomology is called the \emph{Lie algebroid cohomology} and encompasses all types of cohomologies listed above. 
  More generally, a Lie algebroid representation of $A\to M$ on a vector bundle $E\to M$ defines the \emph{Lie algebroid cohomology of $A$ with values in the representation $E$}. It is also called 
  the \emph{twisted} Lie algebroid cohomology of $A$ with coefficients in $E$.

\medskip

 This paper studies LA-homotopies, i.e.~homotopies in the category of Lie algebroids. 
 Two Lie algebroid morphisms $(\Phi_0,\phi_0)$, $(\Phi_1,\phi_1)$ from a Lie algebroid $A\to M$ to a Lie algebroid $B\to N$ are \textbf{LA-homotopic} if there exists an LA-morphism $(\Phi,\phi)$ of the form
    	\begin{equation*}
    		\begin{xy}
    			\xymatrix{ TI \times A \ar[r]^{\quad\Phi}\ar[d]& B\ar[d]\\
    				I \times M \ar[r]_{\quad\phi} & N}
    		\end{xy}
    	\end{equation*}
    		such that $\Phi_j= \Phi(0_j, \cdot)$ for $j=0,1$. The pair $(\Phi,\phi)$ is then called an \textbf{LA-homotopy} between $(\Phi_0,\phi_0)$ and  $(\Phi_1,\phi_1)$.
    	    	   
 This notion of LA-homotopy seems natural and sporadically appears in the literature (e.g.~\cite{Bal12} and \cite{Str05}).
  It agrees with the notion of homotopy of $A$-paths used by Crainic and Fernandes \cite{CrFer01}. The invariance of the Lie algebroid cohomology under LA-homotopy is given in \cite{Bal12}, as a corollary of a Stokes' theorem for Lie algebroids.
  Brahic and Pasievitch give an alternative proof in \cite{BrPa20} and show that LA-homotopies integrate to smooth families of natural transformations between Lie groupoids.
  
The following extension to twisted Lie algebroid cohomology is proved here, and a study of examples and further applications for computing the cohomology itself is made.

 \begin{main}[see \ref{th:hi}]
 	The twisted Lie algebroid cohomology is LA-homotopy-invariant.
More precisely, consider two LA-homotopic LA-morphisms \[(\Phi_0,\phi_0),(\Phi_1,\phi_1)\colon A \to B\] via an LA-homotopy
     \[\Phi\colon TI\times A\to B\]
     over $\phi\colon I\times M\to N$, as well as a vector bundle $E\to N$ with a flat $B$-connection $\nabla\colon \Gamma(B) \times \Gamma(E) \to \Gamma(E)$.
    	Then the maps in twisted cohomology induced by the Lie algebroid inclusions $\mathcal{I}_t\colon A \to TI \times A$, $a\mapsto (0_t,a)$,     	
	\begin{equation*}
	\overline{\mathcal{I}_{t}^*} \colon \tH_{\Phi^*\nabla}^{\bullet}(TI \times A, \phi^!E) \to \tH_{\Phi_t^*\nabla}^{\bullet}(A, \phi_t^!E) 
	\end{equation*}
    		are isomorphisms for all $t\in I$, and the following diagram commutes
	% https://q.uiver.app/#q=WzAsMyxbMCwwLCJIX1xcbmFibGEoQixFKSJdLFsyLDAsIkhfe1xcUGhpXzFeKlxcbmFibGF9KEEsXFxwaGlfMV4hRSkiXSxbMiwyLCJIX3tcXFBoaV8wXipcXG5hYmxhfShBLFxccGhpXzBeIUUpIl0sWzAsMSwiXFxvdmVybGluZXtcXFBoaV8xXip9Il0sWzEsMiwiXFx0aGV0YV8qIl0sWzAsMiwiXFxvdmVybGluZXtcXFBoaV8wXip9IiwyXV0=
\[\begin{tikzcd}
	{H_\nabla(B,E)} && {H_{\Phi_0^*\nabla}(A,\phi_0^!E)} \\
	\\
	&& {H_{\Phi_1^*\nabla}(A,\phi_1^!E).}
	\arrow["{\overline{\Phi_0^*}}", from=1-1, to=1-3]
	\arrow["{\overline{\Phi_1^*}}"', from=1-1, to=3-3]
	\arrow["{\overline{\mathcal{I}_{1}^*} \circ (\overline{\mathcal{I}_{0}^*})}^{-1}", from=1-3, to=3-3]
\end{tikzcd}\]
Moreover, there exists a gauge equivalence $\theta\colon \phi_0^!E\to \phi_1^!E$ from $\Phi_0^*\nabla$ to $\Phi_1^*\nabla$ such that \[\overline{\mathcal{I}_{1}^*} \circ (\overline{\mathcal{I}_{0}^*})^{-1}= \theta_*.\]
    In other words, $(\Phi_0,\phi_0),(\Phi_1,\phi_1)\colon A \to B$ induce, up to a gauge isomorphism, the same morphism in LA-cohomology.
    	
%   
%
%	
%	More precisely,  consider two LA-homotopic LA-morphisms $(\Phi_0,\phi_0),(\Phi_1,\phi_1)\colon A \to B$ via an LA-homotopy
%     $\Phi\colon TI\times A\to B$
%     over $\phi\colon I\times M\to N$, as well as a vector bundle $E\to N$ with a representation $\nabla\colon \Gamma(B) \times \Gamma(E) \to \Gamma(E)$.
%    	Then   there exists an isomorphim  $\theta\colon \phi_0^!E\to \phi_1^!E$ intertwining $\Phi_0^*\nabla$ and $\Phi_1^*\nabla$ such that 
%	% https://q.uiver.app/#q=WzAsMyxbMCwwLCJIX1xcbmFibGEoQixFKSJdLFsyLDAsIkhfe1xcUGhpXzFeKlxcbmFibGF9KEEsXFxwaGlfMV4hRSkiXSxbMiwyLCJIX3tcXFBoaV8wXipcXG5hYmxhfShBLFxccGhpXzBeIUUpIl0sWzAsMSwiXFxvdmVybGluZXtcXFBoaV8xXip9Il0sWzEsMiwiXFx0aGV0YV8qIl0sWzAsMiwiXFxvdmVybGluZXtcXFBoaV8wXip9IiwyXV0=
%\[\begin{tikzcd}
%	{H_\nabla(B,E)} && {H_{\Phi_0^*\nabla}(A,\phi_0^!E)} \\
%	\\
%	&& {H_{\Phi_1^*\nabla}(A,\phi_1^!E)}
%	\arrow["{\overline{\Phi_0^*}}", from=1-1, to=1-3]
%	\arrow["{\overline{\Phi_1^*}}"', from=1-1, to=3-3]
%	\arrow["{\theta_*}", from=1-3, to=3-3]
%\end{tikzcd}\]
%    	commutes.
%	In other words, $(\Phi_0,\phi_0),(\Phi_1,\phi_1)\colon A \to B$ induce, up to an isomorphism, the same morphism in LA-cohomology.

 \end{main}

After the first version of this work was completed, it has been brought to the authors' attention that an alternative proof of Theorem \ref{th:hi} can be provided with techniques and results studied in \cite{Fr19}, despite the statement not appearing there in this generality (see Remark \ref{relationswithPedroswork} for more details).

\medskip
The key for proving the main theorem above is the study of the behavior of a Lie algebroid representation when pulled back by an LA-homotopy. This is studied in detail in Section \ref{ap:A}.
 Note that  the notion of  \emph{pullback} of a Lie algebroid is crucial itself for constructing examples of LA-homotopies\footnote{It turns out that it was already suggested in \cite{Sev05}.} (see Corollary \ref{cor:00} and Lemma \ref{cor:000}).
 The behavior of pullback Lie algebroids under homotopic maps is therefore investigated (see Proposition \ref{pr:mein}). 
 This leads as well to new insights about the twisted cohomology of pullback Lie algebroids.

 	\begin{theorem*}[see \ref{cor:0}]
 		The twisted cohomology of a transitive Lie algebroid is isomorphic to the twisted cohomology of its pullback Lie algebroid under a homotopy equivalence of the base manifolds. Under suitable hypotheses, the same result holds also for non-transitive Lie algebroids.
 	\end{theorem*}
 	
 	A direct consequence of this is a generalized Poincar\'e lemma for transitive Lie algebroids.

  \begin{lemma*}[see \ref{Th:PL}]
  	The twisted Lie algebroid cohomology of a transitive Lie algebroid over a contractible manifold %with values on a representation 
	is isomorphic to the Chevalley-Eilenberg cohomology of its isotropy Lie algebra, with values in the restricted representation.
  \end{lemma*}
  
  Theorem \ref{cor:0} and Lemma \ref{Th:PL} unravel the full potential of the Mayer-Vietoris principle and allow explicit computations for the twisted Lie algebroid cohomology. Some of these computations will be done in the second author's PhD thesis in preparation.\\
  
  The explicit computation of the boundary map of the Mayer-Vietoris sequence is done in Theorem \ref{MV} and then used to prove the following K\"unneth formula, which generalizes results of Mishchenko and Kubarski in \cite{Kub02}, and of Waldron in \cite{Wal23}.
  \begin{theorem*}[see \ref{th:k1}]
  	The K\"unneth isomorphism holds for the twisted Lie algebroid cohomology of the direct product Lie algebroid $A \times B$, where $A$ is any Lie algebroid and $B$ is a transitive Lie algebroid over a manifold admitting a finite good cover.
  \end{theorem*} 
Concerning computational tools for Lie algebroid cohomology, it is relevant to mention that a new spectral sequence for Lie algebroid cohomology is studied in the recent work \cite{MaSc24}. Submersions by Lie algebroid, studied in \cite{Fr19} give other computational insights. The invariance of twisted Lie algebroid cohomology under Morita equivalence is studied in \cite{Cr03} -- some relations with this work are explained in Section \ref{subs:5.3}. Further relations with the considerable amount of existing literature on Lie algebroids are indicated along the paper. For instance, relations between LA-homotopy and deformation theory are investigated in Proposition \ref{homotopy_lie_algebras}. More precisely,
 LA-homotopies between Lie algebra morphisms are shown to be smooth curves of Lie algebra morphisms, the variations of which are controlled by the adjoint representation of the target Lie algebras.
  This implies that the chosen notion of LA-homotopy is compatible with the deformation theory of Lie algebras (see Proposition \ref{LA_Hom_def}).
  
   Equivariant de Rham cohomology and foliated de Rham cohomology are not homotopy-invariant in general. However Theorem \ref{th:hi} implies that they are LA-homotopy-invariant. This remark is consistent with the literature, since LA-homotopy specialises to the notion of equivariant homotopy and of integral homotopy, in the respective categories. Those notions of homotopies are known to preserve the respective cohomologies (see Examples \ref{ex:Gbundles} and \ref{ex:foliations} and references therein).
	
	\medskip

		Note finally that the results in this paper should be extendable to general differential graded manifolds. This is current work in progress, using a notion of homotopy compatible with the definition of this paper and the ones of \cite{Sev05}, \cite{Cam20} and \cite{Royt10}.

\bigskip

  \subsection*{Outline of the paper}
Section \ref{sec:prelim} introduces necessary concepts and results for the rest of the paper. Particular attention is dedicated here to the notion of pullback under Lie algebroid morphisms of forms with values in representations. This is a natural definition and likely to be folklore knowledge, although hard to find in the literature. A detailed treatment is provided in Subsection \ref{subs:2.4} for future reference. The other parts of Section \ref{sec:prelim} are intended for readers who are not familiar with Lie algebroids and their cohomology, so very detailed as well.\\

Section \ref{sec:H} defines the notion of homotopy of Lie algebroids, called LA-ho\-mo\-to\-py, and carefully justifies it.  
The main theorem on the LA-homotopy invariance of the twisted LA-cohomology is given in Section \ref{sub:proof}. Its proof is technical and therefore splits between this section and Section \ref{ap:A}. It uses a study in Section \ref{sec:6} of flows of special linear vector fields on Lie algebroids.
Section \ref{subs:3.60} studies the homotopy-invariance of pullback Lie algebroids. As a consequence, a large class of examples of LA-homotopies is constructed in \ref{subs:3.6}.\\

The special case of Lie algebras is investigated in Section \ref{subs:3.1} and \ref{subs:3.2}. In particular Section \ref{subs:3.2} explains how LA-homotopy turns out to be compatible with the deformation theory of morphisms of Lie algebras. \\

Results on the twisted cohomology of pullback Lie algebroids via homotopy equivalences are studied in Section \ref{subs:5.1}.
A relevant application is a generalized Poincaré-Lemma for transitive Lie algebroids, which is explained in Section \ref{subs:5.2}. Relations with the existing literature on Lie algebroid cohomology of Morita equivalent Lie algebroids are investigated in Section \ref{subs:5.3}.\\

Section \ref{sec:MV} observes that the Mayer-Vietoris principle holds for twisted Lie algebroid cohomology. This result provides in particular an alternative proof of the finite dimensionality of the twisted LA-cohomology of transitive Lie algebroids over a compact base manifold. By combining Mayer-Vietoris with the generalized Poincaré-Lemma, a K\"unneth formula is then proved in Section \ref{sec:K}.\\

Section \ref{sec:6} contains a detailed treatment of the invariance under LA-homotopies of representations of Lie algebroids. This specializes to the invariance of flat connections under classical homotopies.

  \subsection*{Notation and conventions}

  All manifolds and vector bundles in this paper are smooth and real.
  Vector bundle projections are written $q_E\colon E\to M$, except for the projections of tangent bundles, which are written
  $p_M\colon TM\to M$.  Given a section
  $\varepsilon$ of $E^*$, the map $\ell_\varepsilon\colon E\to \R$ is
  the linear function associated to it, i.e.~the function defined by
  $e_m\mapsto \langle \varepsilon(m), e_m\rangle$ for all $e_m\in E$.
  The set of global smooth sections of a vector bundle $E\to M$ is denoted by
  $\Gamma(E)$, $\mx(M)=\Gamma(TM)$ is the space of smooth vector fields on a
  smooth manifold $M$, and
  $\Omega_{\text{dR}}^\bullet(M)=\Gamma(\bigwedge^\bullet T^* M)$ is the space of smooth
  forms on $M$.
  
  Several notions of pullbacks are defined and used in this paper. The (classical) pullback of a vector bundle $B \to N$ and of a section $b \in \Gamma(B)$ through a map $f\colon M\to N$ are denoted  by $f^!B\to M$ and $f^!b\in\Gamma(f^!B)$, respectively. $f^{!!}B$ denotes the pullback of a \emph{Lie algebroid} $B \to N$ under a smooth map $f\colon M\to N$  (\cite{Man05}, see Section \ref{sec:prelim} below for the details). The canonical Lie algebroid morphism over $f$ resulting from the Lie algebroid pullback of $B\to N$ via $f$ is written  $p_{B,f}\colon f^{!!}B \to B$.  $\Phi^*\omega\in\Omega^\bullet(f^{!!}B, f^!E)$ stands for the pullback of a form $\omega \in \Omega^{\bullet}(B, E)$ and similarly, $\Phi^*\nabla\colon \Gamma(f^{!!}B)\times \Gamma(f^!E)\to\Gamma(f^!E)$ is the pullback of a connection $\nabla\colon\Gamma(B)\times\Gamma(E)\to\Gamma(E)$.

  \subsection*{Acknowledgement}
 
 This paper contains some of the results of the PhD thesis in preparation of the second author.
 
She would like to thank T.~Schick and C.~Zhu for their great support. The constant interaction with them, the Topology and the Higher Structures groups  as well as the Research Training Group 2491 in G\"ottingen, was a precious source of inspiration for her work on this project. She thanks R.~Nest for 
long and invaluable discussions and precious encouragement, during her research stays in Copenhagen, supported by both the Centre for Geometry and Topology in Copenhagen and the RTG 2491 in G\"ottingen. She thanks the Topology group in Copenhagen for stimulating questions during a talk on this paper.
She also thanks M.~Crainic, C.~Laurent-Gengoux, I.~Marcut, E.~Meinrenken, J.~Schnitzer and P.~Ševera for interesting discussions. 

  The authors thank C.~Blohmann and his research group for useful observations after a talk on this paper at a joint Bonn-G\"ottingen-W\"urzburg retreat. They thank K.-H.~Neeb and I.~Marcut for suggesting the solution studied in Subsection \ref{subs:3.2} (and K.-H.~Neeb for the first part of Remark \ref{rem_LA_hom_LAlgebras}), K.~Singh for suggesting  as well the connections with deformation theory in the Lie algebra case, T.~Schick for proofreading Appendix \ref{ap:A} and a comment leading to a simplification of the exposition,  as well as O.~Brahic and P.~Freijlich for pointing out the relation of this paper to their work.
%  , which led to Subsection \ref{subs:3.2}. 
  R.~Fernandes, D.~Roytenberg and T.~Strobl suggested some references after a talk on this paper at the ``Singular foliations and dg-manifolds" workshop in Paris, Institut Henri-Poincaré. Finally, the authors are grateful to V.~Ginzburg for his help on verifying some references.

\section{Preliminaries}\label{sec:prelim}
This section collects necessary background on Lie algebroids, their representations and their cohomology, as well as on vector-valued forms and their pullbacks under Lie algebroid morphisms. 
	\subsection{The category of Lie algebroids}
     The definitions of Lie algebroid and of Lie algebroid morphisms are recalled here for the convenience of the reader. See e.g.~\cite{Mackenzie05} for more detail.
   
     \begin{definition}
     	\label{def:LA}
     	A \textbf{Lie algebroid} is a triple $(A\to M, [\cdot, \cdot], \rho_A)$, where $A \to M$ is a vector bundle, $[\cdot, \cdot]$ is a $\mathbb R$-Lie algebra bracket on $\Gamma(A)$, that is \emph{anchored} by a vector bundle morphism $\rho_A \colon A \to TM$, called the \textbf{anchor}:
     	\[[a_1, f\cdot a_2]=f\cdot [a_1, a_2]+ \ldr{\rho_A(a_1)}(f)
     	\cdot a_2 \]
     	for all $a_1, a_2 \in  \Gamma(A)$ and $f \in C^{\infty}(M)$. $\ldr{\rho_A(a_1)}(f)$ denotes the Lie derivative of $f \in C^{\infty}(M)$ with respect to $\rho_A(a_1) \in \mathfrak{X}(M)$, it is often written $\rho_A(a_1)(f)$.
     \end{definition}
   
     It follows that the anchor is compatible with the bracket in the following sense:
     \[[\rho_A(a_1), \rho_A(a_2)]= \rho_A[a_1,a_2] \]
     for all $a_1, a_2 \in  \Gamma(A)$.
     
     \medskip

     A Lie algebroid $A\to M$ is called \textbf{regular} if $\operatorname{Im}(\rho_A)$ is a subbundle of $TM$. \textbf{Transitive Lie algebroids}, i.e.~Lie algebroids with surjective anchor, are special cases of regular Lie algebroids.
      
     \begin{ex} 
     	\begin{itemize}
     		
     	\item For $M=\{*\}$, the anchor map vanishes and $A$ is a real vector space with a Lie bracket, i.e.~a real Lie algebra.
        \item The tangent bundle $TM\to M$ is a Lie algebroid with the identity as anchor and the Lie bracket of vector fields. Involutive subbundles are exactly the infinitesimal description of foliations (by Frobenius' Theorem, \cite{Deahna1840,Clebsch1866}). 
        Therefore, involutive subbundles are often simply called \emph{foliations}. An involutive subbundle  $F \subseteq TM$ of the tangent bundle of a smooth manifold $M$ is a Lie algebroid with the inclusion $\iota\colon F\to TM$ as anchor.
        \item Let $G$ be a Lie group and consider a principal $G$-bundle $\pi\colon P\to M$. That is, $P$ is equipped with a free and proper smooth right action of the Lie group $G$, $M\simeq P/G$ and $\pi\colon P \to P/G=M$ is the orbit projection, a surjective submersion. The associated algebroid, called \textbf{Atiyah Lie algebroid}, has
     	\[A(P):= TP/G \to M\]
     	as underlying vector bundle.
     	The space of sections is isomorphic to the space of $G$-equivariant vector fields on $P$: $\Gamma(TP/G) \cong\mathfrak{X}^G(P)$ and the Lie bracket of $G$-equivariant vector fields on $P$ is again $G$-equivariant.
	Hence the $C^{\infty}(M)$-module of $G$-equivariant vector fields in $P$ is equipped with the Lie bracket
     	\[ \left[\overline{X}, \overline{Y}\right]_{TP/G}(m):= \overline{[X, Y]}(m) \]
     	for all $X,Y\in\mx^G(P)$ and all $p \in \pi^{-1}(m)$. 
	Here, $\overline{X}\in \Gamma(TP/G)$ is the section defined by $X$:
	\[ \overline{X}(m)=[X(p)]\in TP/G
	\]
	for all $m\in M$ and $p\in \pi^{-1}(m)$.
Since the projection $\pi$ is $G$-invariant, its tangent map factors to 
\[ \rho:=\overline{T\pi}\colon TP/G\to TM, \qquad \rho[v_p]=T_p\pi v_p
\]	
for all $v_p\in T_pP$. This induced vector bundle morphism $\rho$ over the identity on $M$ anchors the Lie bracket on $A(P)=TP/G$ and 
$(A(P)\to M, [\cdot\,,\cdot]_{TP/G}, \rho)$ is a Lie algebroid.
        \end{itemize}
     	\end{ex}
 
     Morphisms of Lie algebroids are vector bundle morphisms with an addition compatibility with the anchors and the Lie brackets. 
    If $A,B\to M$ are Lie algebroids over a same base, then the definition is straightforward: A vector bundle  morphism $\Phi\colon A\to B$ over the identity on $M$ is a Lie algebroid morphism if $\rho_B\circ \Phi =\rho_A$ and 
    \[ \Phi[ a_1,a_2]_A=[\Phi(a_1), \Phi(a_2)]_B
    \]
    for all $a_1,a_2\in\Gamma(A)$.
     In general, the bracket condition of a morphism of Lie algebroids over different bases is more difficult to formulate, since the composition of sections of a vector bundle $A\to M$ with a vector bundle morphism $\Phi\colon A\to B$  are generally not sections of the codomain vector bundle $B\to N$. The bracket condition is written using an appropriate pullback construction of the codomain vector bundle via the base map $\phi\colon M \to N$, as in the following definition.     \begin{definition}
     	A vector bundle map $(\Phi,\phi)$ between Lie algebroids $A \to M$ and $B \to N$ 
     	\begin{equation*}
     	\begin{xy}
     		\xymatrix{A \ar[r]^{\Phi}\ar[d]& B \ar[d]\\
     			M \ar[r]_{\phi} & N}
     	\end{xy}
       \end{equation*}
       is a \textbf{Lie algebroid morphism} if the anchor maps satisfy
       $\rho_B \circ \Phi = T\phi \circ \rho_A $
       and if the Lie brackets satisfy the condition described in the following.
       Let \[\Phi^! \colon A \to \phi^!B,\qquad a\mapsto (q_A(a), \Phi(a))\] be vector bundle morphism (over the identity on $M$) induced by $\Phi$ to the pullback of $B\to N$ through $\phi\colon M \to N$. If $a_1,a_2 \in \Gamma(A)$ are such that $\Phi^!(a_1)= \sum_i f^1_{i} \varphi^!b^1_{i}$ and $\Phi^!(a_2)= \sum_j f^2_{j}  \varphi^!b^2_{j}$, for $b^1_{i}, b^2_{j} \in \Gamma(B)$ and $f^1_{i}, f^2_{j} \in C^{\infty}(M)$, then the bracket condition is
       \small{
       \[\Phi^!([a_1,a_2]_A)= \sum_{ij} f^1_{i}\, f^2_{j} \varphi^![b^1_{i},b^2_{j}]_B 
       + \sum_j \rho_A(a_1)(f^2_{j})\varphi^!b^2_{j} - \sum_i \rho_A(a_2)(f^1_{i})\varphi^!b^1_{i}.\]}
     \end{definition}
 The space of Lie algebroid morphisms from $A$ to $B$ is denoted by $\text{Hom}_{LA}(A,B)$.\\
 The following are examples of Lie algebroid morphisms.
 \begin{ex}
 	\label{ex:LAmorph}
 	\begin{itemize}
 \item	If $A$ and $B$ are Lie algebras, i.e.~if $M$ and $N$ are just a point, then a Lie algebroid morphism $A\to B$ is a  Lie algebra morphism.
 \item Let $M$ and $N$ be smooth manifolds and let $f\colon M\to N$ be a smooth map.
 	Then $Tf\colon TM\to TN$ is a Lie algebroid morphism over $f\colon M\to N$. If $F_M\subseteq TM$ and $F_N\subseteq TN$ are involutive subbundles and $Tf(F_M)\subseteq F_N$, then $Tf$ restricts to a Lie algebroid morphism $Tf\colon F_M\to F_N$.
\item Let $\pi_P\colon P\to M$ and $\pi_Q\colon Q\to N$ be principal $G$-bundles and let $f\colon P\to Q$ be a smooth $G$-equivariant map. Then $f$ factors to a smooth map $\bar f\colon M\to N$ such that $\bar f\circ \pi_P=\pi_Q\circ f$ and  $Tf$ factors to a Lie algebroid morphism $\overline{Tf}\colon A(P)\to A(Q)$ over $\bar f\colon M\to N$.
   \end{itemize}
 \end{ex}
 
 Later Proposition \ref{prop:0} shows that a Lie algebroid morphism $A\to B$ over $M\to N$ induces a map in cohomology, and can alternatively be defined as a vector bundle morphism that defines a cochain map $\Gamma(\wedge^\bullet B^*)\to \Gamma(\wedge^\bullet A^*)$ with respect to the Lie algebroid differentials.
      
      \bigskip
   
     The remainder of this section collects some further useful notions and constructions around Lie algebroids.

	 \begin{definition}
	 Given two Lie algebroids $A \to M$ and $B \to N$, the product  $A \times B \to M \times N$ of vector bundles comes naturally equipped with the  \textbf{direct product Lie algebroid} structure. Note that if $\pr_M \colon M \times N \to M$ and $\pr_N \colon M \times N \to N$ are the projections, then $A\times B$ is isomorphic to the Whitney sum $\pr_M^!A\oplus \pr_N^!B$, as vector bundles over $M\times N$.
	 The anchor $\rho_{A \times B} \colon A \times B \to TM \times TN$ is the product $\rho_A\times \rho_B\colon (a_m,b_n)\mapsto (\rho_A(a_m),\rho_B(b_n))\in T_mM\times T_nN$.
	 Since the $C^{\infty}(M \times N)$-module of sections of $A \times B$ is generated by pullback sections of the form  $\pr_M^!a\colon M\times N\to A\times B$, $(m,n)\mapsto (a(m),0^B_n)$ and $\pr_N^!b$ with $a \in \Gamma(A)$ and $b \in \Gamma(B)$, the bracket is uniquely (well-)defined by the Leibniz rule and 
	 \[\left[\pr_M^!a_1, \pr_M^!a_2 \right]_{A \times B}= \pr_M^!\left[a_1, a_2\right]_A, \qquad 
	\left[\pr_N^!b_1, \pr_N^!b_2 \right]_{A \times B}= \pr_N^!\left[b_1, b_2\right]_B\]
	 \[\left[\pr_M^!a, \pr_N^!b \right]_{A \times B}=0=\left[\pr_N^!b, \pr_M^!a\right]_{A \times B}\]
	 for $a, a_1, a_2 \in \Gamma(A)$ and $b, b_1, b_2 \in \Gamma(B)$.
	 \end{definition}

	\bigskip 
      As already mentioned in the introduction, the key idea of the research on Lie algebroid cohomology and of this paper is to understand standard de Rham cohomology
      as the Lie algebroid cohomology of the tangent bundles of manifolds, and then attempt to adapt known results in this special case to the general setting of Lie algebroids. It is easy to see that the pullback of a tangent bundle under a smooth map is not a tangent bundle anymore.
      The correct notion of pullback Lie algebroid remedies to this -- in this setting the pullback of a tangent bundle is again a tangent bundle.
The following construction is standard, see \cite{Mackenzie05}. It is given in detail here because the notion of pullback of Lie algebroids is central in this paper.

      Let $B \to N$ be a Lie algebroid and consider a smooth map $f \colon M \to N$.
      Set 
      \[f^{!!}B:= TM \times_{TN} B:=\left\{(v,b)\in  TM \times B \mid Tf(v)=\rho_B(b)\right\} \subseteq TM \oplus f^!B.\]
      If $f^{!!}B$ has constant rank over $M$, then it is a smooth subbundle of $TM\oplus f^! B$. It is then equipped with a Lie algebroid structure defined below, which is called 
      \textbf{pullback of the Lie algebroid $B\to N$ under the smooth map $f \colon M \to N$}.
      
      Note that the vector bundle $f^{!!}B\to M$ is then a sub-vector bundle of $TM \times B \to M \times N$, over the graph $\operatorname{graph}(f)\subseteq M\times N$ of $f$.
	     % https://q.uiver.app/#q=WzAsOCxbMCwwLCJUTVxcb3BsdXMgZl4hQiJdLFsxLDAsIk0iXSxbMCwxLCJUTVxcdGltZXMgQiJdLFsxLDEsIk1cXHRpbWVzIE4iXSxbMywwLCIodl9tLGJfbikiXSxbNCwwLCJtIl0sWzMsMSwiKHZfbSxiX24pIl0sWzQsMSwiKG0sbikiXSxbMCwyXSxbMCwxXSxbMSwzXSxbMiwzXSxbNCw1LCIiLDAseyJzdHlsZSI6eyJ0YWlsIjp7Im5hbWUiOiJtYXBzIHRvIn19fV0sWzQsNiwiIiwyLHsic3R5bGUiOnsidGFpbCI6eyJuYW1lIjoibWFwcyB0byJ9fX1dLFs2LDcsIiIsMix7InN0eWxlIjp7InRhaWwiOnsibmFtZSI6Im1hcHMgdG8ifX19XSxbNSw3LCIiLDAseyJzdHlsZSI6eyJ0YWlsIjp7Im5hbWUiOiJtYXBzIHRvIn19fV1d
\[\begin{tikzcd}
	{f^{!!}B} & {TM\times B} && {(v_m,b_{f(m)})}& {(v_m,b_{f(m)})}\\
	 M& {M\times N} && m & {(m,f(m))}
	\arrow[from=1-1, to=2-1]
	\arrow[hook, from=1-1, to=1-2]
	\arrow[from=1-2, to=2-2]
	\arrow[hook, from=2-1, to=2-2]
	\arrow[maps to, from=1-4, to=1-5]
	\arrow[maps to, from=1-4, to=2-4]
	\arrow[maps to, from=2-4, to=2-5]
	\arrow[maps to, from=1-5, to=2-5]
\end{tikzcd}\]
 It turns out that $f^{!!}B\to M$ is via this inclusion a Lie subalgebroid of the product Lie algebroid $TM\times B\to M\times N$.
 The anchor in $TM\times TN$ of $(v_m, b_{f(n)})\in f^{!!}B\subseteq TM\times B$ is 
 \[ (v_m, \rho_B(b_{f(m)}))=(v_m, T_mfv_m)\in T_{(m,f(m))}\operatorname{graph}(f).
 \]
It is easy to see that a section $\sigma$ of $TM\times B\to M\times N$ that restricts to a section of $f^{!!}B$ over $M\simeq \operatorname{graph}(f)$
 can without loss of generality\footnote{
 A priori a section of $TM\times B\to M\times N$ can be written
 \[ \sigma=\sum_jF_j\pr_M^!X_j+\sum_iH_i\pr_N^!b_i
 \]
 with $F_j, H_i\in C^\infty(M\times N)$, $X_j\in\mx(M)$ and $b_i\in\Gamma(B)$, for all $i,j$.
 $\sigma$ restricts to a section of $f^{!!}B\to \operatorname{graph}(f)$ if and only if for all $m\in M$,
 \[\sum_jF_j(m, f(m))X_j(m)=\sum_iH_i(m,f(m))\rho_B(b_i(f(m))).
 \]
 Define for all $i, j$ the smooth functions $f_j\in C^\infty(M)$ and $h_i\in C^\infty(M)$ by $f_j(m)=F_j(m,f(m))$ and $h_i(m)=H_i(m,f(m))$ for all $m\in M$, and set $X:=\sum_jf_jX_j\in\mx(M)$. Then on $\operatorname{graph}(f)$, the section $\sigma$ coincides with 
 \[\pr_M^!X+\sum_i\pr_M^*h_i\cdot\pr_N^!b_i.
 \]
 Since the anchor restricted to $f^{!!}B$ has image in $\operatorname{graph}(f)$, only the values of $\sigma$ over  $\operatorname{graph}(f)$ are relevant for computing the restriction to $\operatorname{graph}(f)$ of its Lie brackets with other sections with the same restriction property.} 
 be written
\[ \sigma= \pr_M^!X+\sum_{i}\pr_M^*h_i\cdot \pr_N^!b_i
\]
with $X\in\mx(M)$, $h_i\in C^\infty(M)$ and $b_i\in\Gamma(B)$, such that 
\[ T_mfX(m)=\sum_ih_i(m)\rho_B(b_i(f(m)))
\]
for all $m\in M$. This means that for all $s\in C^\infty(N)$, 
\[X(f^*s)=\sum_i h_i\cdot f^*(\rho_B(b_i)(s)),
\] 
see \cite[Proof of Proposition 4.4.3]{Mackenzie05}. Using this, it is easy to check as in \cite{Mackenzie05} that 
if 
\[ \tau= \pr_M^!Y+\sum_{j}\pr_M^*g_j\cdot \pr_N^!c_j
\]
is a second section of $TM\times B\to M\times N$ that restricts to a section of $f^{!!}B$ over $M\simeq \operatorname{graph}(f)$, then for all $s\in C^\infty(N)$
\begin{equation*}
\begin{split}
[X,Y](f^*s)=\,&X\left(\sum_j g_j\cdot f^*(\rho_B(c_j)(s))\right)-Y\left(
\sum_i h_i\cdot f^*(\rho_B(b_i)(s))
\right)\\
=\,&\sum_j X(g_j)\cdot f^*(\rho_B(c_j)(s))-
\sum_i Y(h_i)\cdot f^*(\rho_B(b_i)(s))\\
&+\sum_{i,j} h_ig_j\cdot f^*([\rho_B(b_i),\rho_B(c_j)](s)).
\end{split}
\end{equation*}
That is, the bracket 
\[ \left(\pr_M^![X,Y], \sum_j \pr_M^*(X(g_j))\cdot \pr_N^!c_j-
\sum_i \pr_M^*(Y(h_i))\cdot \pr_N^!b_i\\
+\sum_{i,j} \pr_M^*(h_ig_j)\cdot \pr_N^![b_i, c_j]
\right)
\]
 of $\sigma$ with $\tau$ in in $TM\times B\to M\times N$ restricts to a section of $f^{!!}B$ over $\operatorname{graph}(f)$.
      	
As a summary, the space of sections of $f^{!!}B$ as a vector bundle over $M$, is
      	\[\Gamma(f^{!!}B)=\left\{\left.\left(X, \sum_i h_i \cdot f^!b_i\right) \right| \begin{array}{c}
	X \in \Gamma(TM), b_i \in \Gamma(B), h_i \in C^{\infty}(M),  \text{ such that } \\
	Tf(X(m))= \sum_i h_i(m) \cdot \rho_B(b_i(f(m)))  \text{ for all } m\in M
	\end{array}\right\}.\]
      	and the Lie algebroid structure inherited from $TM\times B\to M\times N$ on this vector bundle is given by 
	\[\rho_{f^{!!}B}(v,b)=v
	\]
	for all $(v,b)\in f^{!!}B$ and 
      	 
      	 \begin{equation*}
      	 	\begin{split} 
		&\left[\left(X, \sum_i h_{i} \cdot f^! b_{i})\right), \left(Y, \sum_j g_{j} \cdot f^!c_{j}\right)\right]\\
		&=
%		 [X_1,X_2] \oplus \{\sum_{ij} h_{i1} h_{j2} \otimes [b_{i1},b_{i2}]+ \\
%      	 \sum_j \rho_{f^{!!}B}(X_1 \oplus (\sum_i h_{i1} \otimes b_{i1}))(h_{j2}) b_{j2} - \sum_i \rho_{f^{!!}B}(X_2 \oplus (\sum_j h_{j2} \otimes b_{j2}))(h_{i1}) b_{i1}\}=\\
      	 \left([X,Y], \sum_{ij} h_{i}g_j \cdot f^![b_{i},c_{j}]+ \sum_j X(g_{j})\cdot f^!c_{j}- \sum_i Y(h_{i})\cdot f^!b_{i}\right).
      	 \end{split}
      	 \end{equation*}
      	 
	 Take two Lie algebroids $A\to M$ and $B\to N$ and  a vector bundle map $(\Phi,\phi)$ 
     	\begin{equation*}
     	\begin{xy}
     		\xymatrix{A \ar[r]^{\Phi}\ar[d]& B \ar[d]\\
     			M \ar[r]_{\phi} & N}
     	\end{xy}.
       \end{equation*}
 If the Lie algebroid pullback $\phi^{!!}B\to M$ exists, then $\Phi$ is a Lie algebroid morphism if and only if 
 \[ \Phi^{!!}\colon A\to \phi^{!!}B, \qquad a\mapsto (\rho(a), \Phi(a))
 \]
 is a Lie algebroid morphism over the identity on $M$, see 
 \cite[Theorem 4.3.6]{Mackenzie05}.
 
 \medskip
      
   It is important to repeat here that the category of Lie algebroids is not closed under such pullback because in general  $f^{!!}B\to M$ is not a vector bundle.  
   A sufficient (but not necessary) condition to the existence of the pullback Lie algebroid is transversality between the map $f\colon M \to N$ and the anchor map $\rho_B \colon B \to TN$, namely 
   \[Tf(T_m M) + \rho_B(B_{f(m)})=T_{f(m)}N \qquad \forall m \in M.\]
	
		\begin{lemma}
		\label{le:0}
		Let  $B \to N$ be a Lie algebroid and let $f \colon M \to N$ be a smooth map such that the pullback Lie algebroid $f^{!!}B \to M$ exists. The map 		\begin{equation*}
			\begin{split}
		     p_{B,f}\colon f^{!!}B \to B, \qquad (v,b) \mapsto b
		    \end{split}
		\end{equation*}
		 is then a Lie algebroid morphism over the smooth map $f\colon M\to N$.
	\end{lemma}
	\begin{proof}
		This is left to the reader. See e.g.~\cite{Mackenzie05}.
	\end{proof}
	
	\subsection{The LA-cohomology}
	Let $(A, [\cdot,\cdot]_A, \rho_A)$ be a Lie algebroid over a base manifold $M$, with $\text{rank}(A)=n$. Define the space of \textbf{$A$-forms of degree $k\in\mathbb N$} as $\Omega^{k}(A):=\Gamma(\bigwedge^{k}(A^*))$ and the differential $\dr_A\colon \Omega^\bullet(A)\to\Omega^{\bullet+1}(A)$ by
		\begin{equation*}
		\begin{split}
			(\dr_A\omega)(a_1 \dots, a_{k+1})&:= \sum_{j=1}^{k+1} (-1)^{j+1}\, \rho_A(a_j)(\omega(a_1, \dots, \widehat{a_j}, \dots,  a_{k+1}))\\ 
			&\qquad +\sum_{j <l} (-1)^{j+l} \omega([a_j, a_l], a_1, \dots, \widehat{a}_j, \dots, \widehat{a}_l, \dots, a_{k+1})
		\end{split}
	\end{equation*}
    for $\omega\in \Omega^k(A)$ and $a_j \in \Gamma(A)$, $j=1, \dots, k+1$. Here, $\rho_A(a_j)(f):= \ldr{\rho_A(a_j)}(f)$ is the Lie derivative of the function $f \in C^{\infty}(M)$ with respect to $\rho_A(a_j) \in \mathfrak{X}(M)$ and $\widehat{a_j}$ means that the section $a_j$ is omitted. It is immediate by the definition of $\dr_A$ and by the Jacoby identity for $[\cdot, \cdot]_A$ that $\dr_A^2=0$, so the cochain complex
    \[0 \rightarrow \Omega^0(A)=C^\infty(M) \xrightarrow{\dr_A} \Omega^1(A) \xrightarrow{\dr_A} \cdots \Omega^n(A) \to 0\]
    induces the \textbf{Lie algebroid cohomology vector spaces} (or \textbf{LA-cohomology vector spaces}) of the Lie algebroid $A\to M$:
 
    \begin{equation*}
    	\begin{aligned}
    		\tH^k(A):= \text{ker}(\dr_k)/\text{im}(\dr_{k-1})
    	\end{aligned}
    \end{equation*}
	with the notation $\dr_k:= \dr_A\an{\Omega^k(A)}\colon \Omega^k(A) \to \Omega^{k+1}(A)$ for all $k\in\mathbb N$. $k$-forms in $\text{ker}(\dr_k)=: \Omega^k_{cl}(A)$ are called \textbf{closed}. The graded algebra of closed forms is written $\Omega^{\bullet}_{cl}(A):= \bigoplus_k \Omega^k_{cl}(A)$.\\
	\begin{remark}
		The LA-cohomology has been intensively studied for decades. See for instance the works of Crainic \cite{Cr03}, Ginzburg \cite{Ginz01}, Kubarski \cite{Kub02, Ku03, Kub13},  Mackenzie \cite{Mackenzie05}, Mishchenko \cite{KuMish04, Mish11}, Oliveira \cite{MishOliv23} and Waldron \cite{Wal23}. Some of these authors' specific contributions are mentioned along the paper.
	\end{remark}
	
	\begin{ex}
		\label{ex:LAcohom}
		\begin{itemize}
		 
		\item For $A=\lie {g} \to \{{\rm pt}\}$ the LA-cohomology is the Chevalley-Eilenberg cohomology $\tH^\bullet(\lie g)=\tH_{\text{ce}}^{\bullet}(\lie g)$.
		\item If $A=TM \to M$, then $\tH^{\bullet}(TM)=\tH^{\bullet}_{\rm dR}(M)$ by definition. For foliations, i.e.~involutive subbundles $F \hookrightarrow TM$, the Lie algebroid cohomology is isomorphic to the foliated de Rham cohomology. This is proved for instance in \cite[Corollary 3.3.10]{Schl21}. 		\item For Atiyah Lie algebroids associated to principal bundles, 
		 the LA-cohomology is isomorphic to the equivariant de Rham cohomology (see \cite[Proposition 5.3.11]{Mackenzie05}).
		 Also this cohomology has already been studied extensively, see in particular \cite{Greub76} and \cite{GoZo19}.
		 \end{itemize}
	\end{ex}
    \begin{example}
    	The LA-cohomology of a given Lie algebroid can be extremely ``large".
	 For instance, if $A\to M$ is a trivial Lie algebroid, i.e.~a Lie algebroid with zero anchor and zero bracket, then the differential $\dr_A$ vanishes by definition. Hence
    	\begin{equation*}
    		\begin{aligned}
    			\tH^0(A)=C^{\infty}(M) \quad \text{ and } \quad \tH^k(A)=\Omega^k(A)%= \Gamma(\bigwedge^{k}(A^*))
    		\end{aligned}
    	\end{equation*}
	for all $k>0$.
	\label{ex:1}
    \end{example}

\medskip
	
	As mentioned earlier, the notion of LA-morphism finds a justification also in relation with the LA-cohomology.  
	
	\begin{proposition}
		A vector bundle morphism $(\Phi,\phi)$ is a Lie algebroid morphism between $A \to M$ and $B \to N$ if and only if the pullback map on sections
		\[\Phi^* \colon \Omega^{\bullet}(B) \to \Omega^{\bullet}(A) \]
		\[(\Phi^*\omega_B)_m(a_1, \dots, a_{\bullet}):=(\omega_B)_{\phi(m)}(\Phi(a_1), \dots, \Phi(a_{\bullet})) \]
		is a cochain map (i.e.~it intertwines the differentials $\dr_A$ and $\dr_B$).
		\label{prop:0}
	\end{proposition} 
	\begin{proof}
		This is a standard result, which is due to Vaintrob \cite{Vaintrob97}.
	\end{proof}
	\medskip
	
	\begin{remark}
	Let $A\to M$ be a Lie algebroid. Then
		$\tH^{\bullet}(A):= \bigoplus_{k\in\mathbb N} \tH^k(A)$ is a graded algebra,
		with cup product induced by the wedge product of $A$-forms, exactly as in  the case of de Rham cohomology (see e.g.~\cite{BoTu82}).    
		A Lie algebroid morphism induces a morphism of rings in cohomology. In particular, the map
		\begin{equation*}
			\begin{aligned}
				\rho_A^*\colon \tH^{\bullet}_{\rm dR}(M) \to \tH^{\bullet}(A)
			\end{aligned}
		\end{equation*}
		induced by the anchor $\rho_A$ is a morphism of rings.
	\end{remark}

   \subsection{LA-cohomology with values in representations}
   
   A more general Lie algebroid cohomology can be defined using `twists' with respect to linear flat Lie algebroid connections. All necessary basic definitions are recalled before introducing the \emph{twisted LA-cohomology}, also called \emph{LA-cohomology with values on representations}.
    \begin{definition}
     	Let $A\to M$ be a Lie algebroid. A \textbf{linear $A$-connection} on the vector bundle $E \to M$ is an $\mathbb R$-bilinear map
     	\[\nabla \colon \Gamma(A) \times \Gamma(E) \to \Gamma(E)\]
     	such that 
     	\[\nabla_a(f\cdot e)= f\cdot \nabla_ae + \rho_A(a)(f)\cdot e\]
	and 
	\[\nabla_{f\cdot a}e= f\cdot \nabla_ae \]
     	for $a \in \Gamma(A)$, $e \in \Gamma(E)$ and $f\in C^\infty(M)$. 
     \end{definition}

    	Let $A \to M$ be a Lie algebroid and let $E \to M$ be a vector bundle. 
	Define \[\Omega^{\bullet}(A, E):=\Gamma(\wedge^{\bullet}A^* \otimes E).\] A linear $A$-connection $\nabla\colon \Gamma(A) \times \Gamma(E) \to \Gamma(E)$ is equivalent to an operator 
    \[\dr_{\nabla}\colon \Omega^{\bullet}(A,E) \to \Omega^{\bullet+1}(A, E)\] defined by the conditions
    \begin{equation*}
    		    \dr_{\nabla} e= \nabla_{\cdot}e \in \Omega^1(A,E)
    \end{equation*}
    \begin{equation*}
    	        \dr_{\nabla}(\omega \wedge \eta)= (\dr_A \omega) \wedge\eta + (-1)^k \omega \wedge \dr_{\nabla} \eta,
    \end{equation*}
     for $e \in \Gamma(E)=\Omega^0(A,E)$, $\omega \in \Omega^k(A)$ and $\eta \in \Omega^{\bullet}(A,E)$. 
    
    \begin{remark}
    	If $E= \R \times M\to M$ with the canonical connection $\nabla_{a}f= \mathcal{L}_{\rho_A(a)}(f)$, then $\Omega^{\bullet}(A,E) = \Omega^{\bullet}(A)$ and $\dr_{\nabla}=\dr_A$.
    \end{remark}

    The linear connection $\nabla$ is called \textbf{flat} if the corresponding curvature $R_\nabla\in\Omega^2(A, \operatorname{End}(E))$ defined by 
    \[R_{\nabla}(a_1, a_2):= \nabla_{a_1} \circ \nabla_{a_2}-\nabla_{a_2} \circ \nabla_{a_1}- \nabla_{[a_1,a_2]}\] 
     for every $a_1, a_2 \in \Gamma(A)$ vanishes identically. This is equivalent (see for instance \cite{Jotz22}) to
    \[0 \equiv \dr^2_{\nabla}\colon \Omega^{\bullet}(A,E) \to \Omega^{\bullet+2}(A,E). \]
    
     The cohomology induced by this differential is the \textbf{LA-cohomology with values in the representation $\nabla$} (or LA-cohomology with values in the representation $E$, or simply \textbf{twisted LA-cohomology}), since flat Lie algebroid connections are the \textbf{representations} of Lie algebroids.
      For a Lie algebroid $A\to M$, a vector bundle $E \to M$ and a flat connection $\nabla\colon \Gamma(A) \times \Gamma(E) \to \Gamma(E)$, the corresponding cohomology is denoted by
     \[\operatorname{H}_{\nabla}^k(A,E):= \text{ker}(\dr_k)/\text{im}(\dr_{k-1}) \] 
     with $\dr_k:=(\dr_{\nabla})_{|_{\Omega^k(A,E)}} \colon \Omega^k(A,E) \to \Omega^{k+1}(A,E)$. A form $\omega \in \text{ker}(\dr_k)=: \Omega^k_{cl}(A)$ is called a \textbf{closed form}.
     \begin{lemma}
     	$\operatorname{H}_{\nabla}^{\bullet}(A,E)$ is an $\operatorname{H}^{\bullet}(A)$-module. 
     \end{lemma}

     \begin{proof}
     	Define the module structure via
     	\begin{equation}\label{cup}
	\ov{\omega} \cup \ov{\eta}:= \ov{\omega \wedge \eta}
	\end{equation}
     	with $\omega \in \Omega_{\text{cl}}^{\bullet}(A)$, $\eta \in \Omega_{\text{cl}}^{\bullet}(A,E)$.
     	It remains to prove that this is well-defined. For $\omega \in \Omega_{\text{cl}}^{k}(A)$, $\eta \in \Omega_{\text{cl}}^{l}(A,E)$, 
     	\[\dr_{\nabla}(\omega \wedge \eta)= \dr_{A} \omega \wedge \eta + (-1)^k \omega \wedge \dr_{\nabla} \eta =0. \]
	That is, if $\omega \in \Omega^{\bullet}_{\text{cl}}(A)$ and $\eta \in \Omega^{\bullet}_{\text{cl}}(A,E)$, then  also $\omega\wedge \eta\in \Omega^{\bullet}_{\text{cl}}(A,E)$ and $\ov{\omega\wedge\eta}$ exists.
     	Moreover, for $\theta \in \Omega^{k-1}(A)$ and $\xi \in \Omega^{l-1}(A,E)$
     	\begin{equation*} \begin{split}
     			(\omega\,+\, \dr_{A} \theta) \wedge (\eta\, +\, \dr_{\nabla} \xi)&= \omega \wedge \eta \,+\, \omega \wedge \dr_{\nabla} \xi\, +\, \dr_A \theta \wedge \eta\, +\, \dr_{A} \theta \wedge \dr_{\nabla} \xi\\
     			&= \omega \wedge \eta\, +\, \dr_{\nabla}((-1)^k\omega \wedge \xi\, +\, \theta \wedge \eta \,+\, \theta \wedge \dr_{\nabla}\xi)
     	\end{split} \end{equation*}
     	where the last passage holds because the forms $\omega$ and $\eta$ are closed and $\dr_\nabla^2=0$. 
	Hence \eqref{cup} is well-defined.
     \end{proof}

   \medskip
  
  The notion of gauge equivalence of (flat) connections is crucial for one of the main statement of this paper (see Theorem \ref{th:hi} below).
   \begin{definition}\label{def:ga}
 Let $F,F'\to M$ be smooth vector bundles and let $A\to M$ be a Lie algebroid.
 Two $A$-connections $\nabla$ and $\nabla'$  on $F$ and $F'$ are called \textbf{gauge equivalent} if 
 there exists an isomorphism $\theta\colon F\to F'$ over the identity on $M$, such that 
 \begin{equation}\label{ga}\nabla_a e = \theta^{-1}\left(\nabla'_a (\theta(e))\right)\end{equation}
 for all $a\in\Gamma(A)$ and $e\in\Gamma(E)$.
 The map $\theta$ is then called a \textbf{gauge equivalence} from  $\nabla$ to $\nabla'$. 
 \end{definition}
 
If $\nabla$ and $\nabla'$ are gauge equivalent as above, then $\nabla$ is flat if and only if $\nabla'$ is flat. 
It is then easy to see that 
\[ \theta_*\colon \operatorname{H}^\bullet_\nabla(A,F)\to \operatorname{H}^\bullet_{\nabla'} (A,F'), \quad [\omega]\mapsto [\theta\circ \omega]
\]
is a well-defined isomorphism, called a \textbf{gauge isomorphism}.

\subsection{Pullback of $A$-forms with values in a representation}
\label{subs:2.4}
   The authors could not find in the literature a precise treatment of the notion of pullback of $A$-forms with values in a representation. 
      This subsection
   remedies to that and discusses these pullbacks in detail. 
   
  \begin{definition}\label{def_pullback_forms}
   Let $A \to M$, $B \to N$ and $E \to N$ be vector bundles.  Let $(\Phi, \phi)\colon A \to B$ be a vector bundle morphism. Then the pullback $\Phi^*\omega\in\Omega^k(A,\phi^!E)$ of  $\omega\in \Omega^{k}(B,E)$ is defined by
  	\begin{equation}
  		\begin{aligned}
  			(\Phi^*\omega)_m(a_1, \dots, a_k):= (m, \omega_{\phi(m)}(\Phi_m(a_1), \dots, \Phi_m(a_k))) \; \in(\phi^!E)_m ,
  			\label{eq:DR2}
  		\end{aligned}
  	\end{equation}
  	for all $m\in M$, and $a_1, \ldots, a_k \in A_m$.
  \end{definition}
 In the situation of the previous definition,  	\begin{equation*}
	\begin{split}
  		(\Phi^*(\omega \wedge e))_m(a_1, \dots, a_k)&=(m, (\omega \wedge e)_{\phi(m)}(\Phi(a_1), \dots, \Phi(a_k))\\
&		=((\Phi^*\omega) \wedge (\phi^!e))_m(a_1, \dots, a_k)
		\end{split}
  	\end{equation*}
  	for $\omega\in \Omega^{k}(B)$ and $e \in \Gamma(E)$, for all $m\in M$, and $a_1,\ldots, a_k \in A_m$.
	Hence 
	\begin{equation}\label{rk:1} \Phi^*(\omega\wedge e)=\Phi^*\omega \wedge \phi^!e
	\end{equation}
	for $\omega\in \Omega^{k}(B)$ and for $e \in \Gamma(E)$.
	In a similar way, it is easy to check that
	\begin{equation}\label{rk:2} \Phi^*(\omega\wedge \eta)=\Phi^*\omega \wedge \phi^* \eta
	\end{equation}
	for $\omega\in \Omega^{k}(B)$ and for $\eta \in \Omega^l(B,E)$.

\begin{remark}
The proofs of the two following properties of the pullback are left to the reader.
  	\label{re:5}
	\begin{enumerate}
  	\item The pullback in Definition \ref{def_pullback_forms} is functorial. That is, if $A\to M$, $B\to N$ and $C\to P$ are vector bundles, $(\Phi,\phi)\colon A\to B$ and $(\Psi,\psi)\colon B\to C$ are vector bundle morphisms and $E\to P$ is a vector bundle. Then for each $\omega\in\Omega^\bullet(C,E)$,
	\[ (\Psi\circ\Phi)^*\omega=\Phi^*(\Psi^*\omega).
	\]
	\item If $A\to M$ and  $B\to N$ are vector bundles, $(\Phi,\phi)\colon A\to B$ is a vector bundle morphism and $E,F\to N$ are vector bundles, then for each $\omega\in\Omega^\bullet(B,\operatorname{Hom}(E,F))$ 
	\[ \Phi^*\omega\quad \in\quad \Omega^\bullet(A, \operatorname{Hom}(\varphi^!E, \varphi^!F))=\Omega^\bullet(A, \varphi^!\operatorname{Hom}(E,F))
	\]
	and for each $e\in\Gamma(E)$,
	\[ \Phi^*(\omega(e))=(\Phi^*\omega)(\varphi^!e) \quad \in\quad \Omega^\bullet(A,\varphi^!F).
	\]
	\end{enumerate}
  \end{remark}

  \begin{example}
  	Consider $A=TM$, $B=TN$ and $E=N \times \R$ with the standard representation, as well as the vector bundle morphism $(T\varphi,\varphi)$ for a smooth map $\varphi\colon M\to N$.  In this case, for $\omega\in\Omega^k(TN,E)=\Omega_{\text{dR}}^k(N)$,
  	the form $(T\varphi)^*\omega$ is the usual pullback $\varphi^*\omega$, since by definition for  all $m\in M$ and all $v_1,\ldots, v_k\in T_mM$
	\[ ((T\varphi)^*\omega)_m(v_1,\ldots, v_k)=(m, \omega_{\varphi(m)}(T_m\varphi (v_1),\ldots, T_m\varphi (v_k)))\in (\varphi^!E)_m=\{m\}\times \mathbb R.
	\]
  \end{example}

  \medskip
 
  \begin{definition}\label{def_restriction}
  	Let  $i\colon N \hookrightarrow M$ be an embedded submanifold such that the pullback $i^{!!}A$ is as vector bundle.
  	Consider the canonical vector bundle morphism $p_{A,i}\colon i^{!!}A \to A$ (see Lemma \ref{le:0}). The \textbf{restriction of a form $\omega\in \Omega^\bullet(A,E)$ to $N$} is defined to be the form $p_{A,i}^*\omega \in \Omega^{\bullet}(i^{!!}A, i^{!}E)$. 
  \end{definition}
  If $N$ is open in  $M$, then
  the restriction of $\omega$ to $N$ is denoted by $\omega\an{N}$, since $i^{!!}A \cong A\an{N}$ and therefore 
  \[\Omega^k(A\an{N}, E\an{N}) := \Gamma(\wedge^{k}A^*\an{N}\otimes E\an{N}) \cong \Gamma(\wedge^{k}(i^{!!}A)^*\otimes i^!E)=: \Omega^k(i^{!!}A, i^!E).\qedhere  \]  
  \begin{remark}
  	Building restrictions as in Definition \ref{def_restriction} is transitive. More precisely, if $\omega \in \Omega^{\bullet}(A,E)$, $j\colon S \hookrightarrow N$ and $i\colon N \hookrightarrow M$ are embedded submanifolds and both $i^{!!}A$ and $j^{!!}(i^{!!}A)$ exist, then 
  	\[p_{A,i \circ j}^*\omega \cong p_{A,j}^*(p_{A,i}^*\omega).\]
  	It follows by functoriality of the pullback of Lie algebroids (see \cite{Mackenzie05}) and by functoriality of the pullback of forms (see Remark \ref{re:5}).
  \end{remark}
  
  \begin{remark}
  	 Recall that using restrictions of vector-valued forms to open subsets,  for each $k \in \N$ the space of forms $\Omega^{k}(A,E)$ defines a sheaf on $M$.
  	 \label{Mv:prop3}
  \end{remark}

  \bigskip
  
  Going back to the situation of Definition \ref{def_pullback_forms}, let now $A\to M$ and $B\to N$ be equipped with Lie algebroid structures, let $(\Phi,\varphi)\colon A\to B$ be a morphism of Lie algebroids\footnote{These conditions can of course be weakened to $A\to M$ and $B\to N$ being anchored vector bundles and $(\Phi,\varphi)$ being a morphism of anchored vector bundles.}, and consider a linear connection $\nabla\colon \Gamma(B) \times \Gamma(E) \to \Gamma(E)$.
  The definition of the pullback of $\nabla$ under $(\Phi, \varphi)$ is induced as follows by the previous notion of pullback of forms\footnote{Note that the pullback of a connection under a smooth map is standard. Here this is extended in a straightforward manner to a notion of pullback under a Lie algebroid morphism. This is discussed here in detail because this notion is central in this paper.}. The linear connection 
    \[(\Phi^*\nabla)\colon \Gamma(A) \times \Gamma(\phi^!E) \to \Gamma(\phi^!E) \]
    is defined on pullback sections by 
  \[(\Phi^*\nabla)_{a_m}(\phi^!e) :=(\Phi^*(\nabla_{\cdot}e))_m(a_m)=(m, \nabla_{\Phi(a_m)}e), \]
  for every $m\in M$, $a_m \in A$ and $e\in\Gamma(E)$. First, note that \eqref{eq:DR2}
   implies the second equality, since $\nabla_{\cdot}e \in \Omega^1(B,E)$. Then 
  check as follows that for $f \in C^{\infty}(N)$, $e \in \Gamma(E)$ and $a_m\in A$
  	\begin{equation}\label{pullback_connection_comp_module}
  		(\Phi^*\nabla)_{a_m} (\phi^!(fe))=\rho_A(a_m)(\phi^*f)\cdot \phi^!e +(\phi^* f)\cdot (\Phi^*\nabla)_{a_m}(\phi^!e).
  	\end{equation}
  For $m\in M$, simply compute  
  	\begin{equation*} 
  		\begin{split} (\Phi^*\nabla)_{a_m} (\phi^!(fe))&= (m, \nabla_{\Phi(a_m)}(fe))\\
		&=(m, \rho_B(\Phi(a_m))(f) \cdot e_{\phi(m)} + f(\phi(m)) \cdot \nabla_{\Phi(a_m)}e) \\
  			&=(m, T_m\phi(\rho_A(a_m))(f) \cdot e_{\phi(m)} + (\phi^*f)(m) \cdot \nabla_{\Phi(a_m)}e)\\ 
			&=\rho_A(a_m)(\phi^*f) \cdot (\phi^!e)(m) + (\phi^*f)(m)\cdot (\Phi^*\nabla)_{a_m}(\phi^!e). 
  		\end{split} 
  	\end{equation*}
  	As a consequence, the pullback connection $\Phi^*\nabla$ which was so far only defined on pullback sections can be extended via the Leibniz rule and  defined explicitly as 
  \begin{equation*}
  	\begin{aligned}
  		(\Phi^*\nabla)_{a}\left(\sum_i f_i \cdot \phi^!e_i \, \right)=\sum_i \left(\rho_A(a)(f_i) \cdot \phi^!e_i + f_i \cdot (\Phi^*\nabla)_{a} (\phi^!e_i)\right), 
  	\end{aligned}
  \end{equation*}
  for $e_i \in \Gamma(E)$, $f_i \in C^{\infty}(M)$.
  
  \begin{example}\label{example_pullback_trivial}
  	The pullback via a LA-morphism $\Phi\colon A \to B$ over $\phi\colon M \to N$ of the canonical flat $B$-connection $\nabla$ on the trivial bundle $\R \times N$ is the canonical flat $A$-connection on $\R \times M\simeq \varphi^!(\mathbb R\times N)$.
  	\label{dR:lemma0}
 Since the pullback of a connection is completely determined by its values on pullback sections and by the Leibniz identity, it is sufficient to compute
  	\begin{equation*} \begin{split}
  			(\Phi^*\nabla)_{a_m}(\phi^*f)&:= (m, \nabla_{\Phi(a_m)}f)= (m, (\rho_B \circ \Phi)(a_m)(f) )\\
			 &=(m, (T\phi \circ \rho_A)(a_m)(f))= \rho_A(a_m)(\phi^*f)(m).
  	\end{split}\end{equation*}
  for $f \in C^{\infty}(N)$ and $a \in \Gamma(A)$.

  	\label{ex:2}
   \end{example}

  The next lemma shows that building the pullback of a flat connection under a Lie algebroid morphism induces a map in cohomology.
  \begin{lemma}\label{pullback_ce}
  In the situation above,
  	the pullback of forms under the LA-morphism $(\Phi,\phi)$ intertwines the differentials $\dr_\nabla$ and $\dr_{\Phi^*\nabla}$:
	\[ \dr_{\Phi^*\nabla}\circ \Phi^*=\Phi^*\circ \dr_\nabla.
	\]
	\label{le:00}
  \end{lemma}
  \begin{proof}
  	For $ e \in \Gamma(E)$, $m\in M$ and $a\in A_m$
  	\[\Phi^*(\dr_{\nabla}e)(a)=(m, \dr_{\nabla}e(\Phi(a)))=(m, \nabla_{\Phi(a)}e)={(\Phi^*\nabla)_{a}}(\phi^!e)=\dr_{\Phi^*\nabla}(\phi^!e)(a). \]
  	More generally, this yields  for $\omega \in \Omega^{k}(B)$ and $e \in \Gamma(E)$   	
  		\begin{equation*}
		\begin{split}
  			\Phi^* (\dr_{\nabla}(\omega \wedge e))&=\Phi^* ((\dr_B \omega) \wedge e+ (-1)^k \omega \wedge \dr_{\nabla}e) \\
			&\overset{\eqref{rk:1} \eqref{rk:2} }{=} \dr_{A} (\Phi^*\omega) \wedge \phi^!e + (-1)^k  \Phi^*\omega \wedge \dr_{\Phi^*\nabla}(\phi^!e)\overset{\eqref{rk:1}}{=} \dr_{\Phi^*\nabla}(\Phi^*(\omega \wedge e)).
  		\end{split}
  	\end{equation*}
  	Proposition \ref{prop:0} is used in the second passage.
  \end{proof}

  \begin{corollary}
  Let $A\to M$ and $B\to N$ be equipped with Lie algebroid structures, let $(\Phi,\varphi)\colon A\to B$ be a morphism of Lie algebroids, and consider a linear connection $\nabla\colon \Gamma(B) \times \Gamma(E) \to \Gamma(E)$. Then 
  \[ \Phi^\star R_{\nabla}=R_{\Phi^\star\nabla} \quad \in\quad \Omega^2(A, \operatorname{End}(\varphi^! E))=\Omega^2(A, \varphi^! \operatorname{End}(E)).
  \]
  \end{corollary}
  
  \begin{proof}
  For all $e\in\Gamma(E)$, the form $R_{\Phi^\star\nabla}(\varphi^!  e)\in\Omega^2(A,\varphi^!E)$ is given by 
  \begin{equation*}
  \begin{split}
  R_{\Phi^*\nabla}(\varphi^! e)&=\dr_{\Phi^*\nabla}^2(\varphi^!e)=(\dr_{\Phi^*\nabla}^2\circ \Phi^*)(e)=(\Phi^*\circ \dr_\nabla^2)(e)=\Phi^*(R_\nabla(e)),
  \end{split}
  \end{equation*}
  with Lemma \ref{pullback_ce} used in the third equality.
 Then
  \[\Phi^*(R_\nabla(e))=(\Phi^*R_\nabla)(\varphi^! e)
  \]
  by Remark \ref{re:5}.
\end{proof}
The following result is then immediate.
  \begin{corollary}
  	In the situation above, if $\nabla$ is flat, then also $\Phi^*\nabla$ is flat.
  	\label{Cor:00}
  \end{corollary}
  This leads to the following corollary. 
  \begin{corollary}
  	Given a LA-morphism $\Phi\colon A \to B$  over $\phi\colon M \to N$ and given a vector bundle $E\to N$  endowed with a flat $B$-connection $\nabla$, the pullback of forms induces a map in cohomology 
  	\[ \Phi^*\colon \operatorname{H}^{\bullet}_\nabla(B;E) \to \operatorname{H}^{\bullet}_{\Phi^*\nabla}(A; \phi^!E)\]
  	\[ [\omega] \mapsto [\Phi^*\omega]     \]
  	for all closed $\omega \in \Omega_{\text{cl}}^{\bullet}(B,E)$. $\Phi^*$ is a morphism of modules over the ring homomorphism $\Phi^*\colon \tH^\bullet(B)\to \tH^\bullet(A)$.
  \end{corollary}

  \begin{proof}
  	The first part of the statement follows immediately by Corollary \ref{Cor:00} and Lemma \ref{le:00}. The extra structure as module homomorphism follows by
  	\begin{equation*} \begin{split}
  			(\Phi^*(\omega \wedge \eta))_m(a_1, \dots, a_{k+l})
  			&= (m, (\omega \wedge \eta)_{\phi(m)}(\Phi_m(a_1), \dots, \Phi_m(a_{k+l}))\\
  			& \overset{\eqref{rk:2}}{=} ((\Phi^*\omega) \wedge (\Phi^*\eta))_m(a_1, \dots, a_{k+l})
  		\end{split}
  	\end{equation*}
  	for $\omega \in \Omega_{\text{cl}}^k(B)$, and $\eta \in \Omega_{\text{cl}}^l(B,E)$. Notice that $(\Phi^*\omega) \in \Omega^{\bullet}_{cl}(A)$ by Example \ref{ex:2}.
  \end{proof}
  
  \begin{remark}
  Let  $\Phi\colon A \to B$ be an LA-morphism over a smooth map $\phi\colon M \to N$, 
  and let the $2$-term complex $\partial\colon E_0\to E_1$ of vector bundles over $N$ carry a representation up to homotopy $(\nabla^0, \nabla^1, \omega)$ of $B$.
  That is, $\nabla^i\colon \Gamma(B)\times \Gamma(E_i)\to \Gamma(E_i)$ is a linear connection for $i=0,1$ and $\omega\in\Omega^2(B, \operatorname{Hom}(E_1, E_0))$ is a form such that 
  \begin{enumerate}
  \item $\partial\circ\nabla^0 =\nabla^1\circ\partial$,
  \item $R_{\nabla^0}=\omega\circ\partial$ and $R_{\nabla^1}=\partial\circ\omega$,
  \item $\dr_{\nabla^{1,0}}\omega=0$,
  \end{enumerate}
  where $\nabla^{1,0}$ is the induced connection by $\nabla^0$ and $\nabla^1$ on $\operatorname{Hom}(E_1, E_0)$.
  It is  easy to check that $(\Phi^*\nabla^0, \Phi^*\nabla^1, \Phi^*\omega)$ is a $2$-term representation of $A$ on $\varphi^!\partial\colon \varphi^!E_0\to \varphi^!E_1$.
  \end{remark}

       \section{LA-homotopy}
     \label{sec:H}
    The definition of homotopy already appears sporadically in the literature, for instance in \cite{Meinrenken17}, \cite{Bal12} and \cite{Str05}. In the more general framework of dg-manifolds it appears in this form for instance in \cite{Sev05} and \cite{Royt10}, but it has remained largely unknown in the community. In the case of Lie algebroids, this paper shows that this notion has remarkable properties: the definition is compatible with related notions, a large class of examples can be found, two LA-homotopic LA-morphisms induce the same map in LA-cohomology with values on representations. This section introduces and discusses the definition itself, as well as some examples, and states the main theorem. The rest of the paper then focuses on applications.\\

    Let $A\to M$ be a Lie algebroid, and in the following let $I$ be an open interval containing $0$ and $1$. Then $TI\times A$ is equipped with the product Lie algebroid structure over $I\times M$. For $t\in I$ the inclusion LA-morphisms 
    \begin{equation}\label{inclusion}
    		\begin{xy}
    			\xymatrix{A \ar[r]^{\mathcal{I}_t\quad }\ar[d]& TI \times A \ar[d]\\
    				M \ar[r]_{\iota_j\quad } & I \times M}
    		\end{xy}
    	\end{equation}
    	are defined by $\mathcal{I}_t(a):=(0_t,a)$ and $\iota_t(m):=(t,m)$ for all $a\in A$ and all $m\in M$.
	For each $t\in I$ this Lie algebroid morphism is right-inverse to the Lie algebroid morphism 
	\begin{equation}\label{projection}
    		\begin{xy}
    			\xymatrix{TI\times A \ar[r]^{\Ppr_A}\ar[d]& A \ar[d]\\
    				I\times M \ar[r]_{\pr_M } & M}
    		\end{xy}
    	\end{equation}
	defined by 
	$\Ppr_A(v,a):=a$ and $\pr_M(t,m):=m$ for all $(t,m)\in I\times M$ and $(v,a)\in TI\times A$.

    \begin{definition}\label{def:hom1}
Let $A\to M$ and $B \to N$ be Lie algebroids. Two Lie algebroid morphisms $(\Phi_0,\phi_0)$, $(\Phi_1,\phi_1)$ from $A$ to $B$ are said to be \textbf{LA-homotopic} if there exists an LA-morphism $(\Phi,\phi)$ of the form
    	\begin{equation*}
    		\begin{xy}
    			\xymatrix{ TI \times A \ar[r]^{\quad\Phi}\ar[d]& B\ar[d]\\
    				I \times M \ar[r]_{\quad\phi} & N}
    		\end{xy}
    	\end{equation*}
	such that $\Phi_j= \Phi \circ \mathcal{I}_j$ for $j=0,1$. The pair $(\Phi,\phi)$ is then called an \textbf{LA-homotopy} between $(\Phi_0,\phi_0)$ and  $(\Phi_1,\phi_1)$.
    	    	    \end{definition}
		    
		     Fo simplicity, just the Lie algebroid morphism $\Phi\colon TI \times A \to B$ is often called a LA-homotopy. However, the reader is invited to bear in mind that the Lie algebroids $A$ and $B$ can have different base manifolds in general. 

    \begin{remark}
    	To make LA-homotopy an equivalence relation, relax the definition above by calling two Lie algebroid morphisms $(\Phi_j,\phi_j) \colon A \to B$, for $j=0,1$, LA-homotopic if there exist a continuous vector bundle morphism 
    	$\Phi\colon T[0,1] \times A \to B$ over a piecewise smooth map $\phi\colon [0,1] \times M \to N$, such that $\Phi\colon T[0,1] \times A \to B$ is a LA-morphism over each closed subinterval where $\phi$ is smooth.
	
	Since it is not important in this paper that LA-homotopy is an equivalence relation, the authors chose to keep the more simple Definition \ref{def:hom1}.
    \end{remark}
    	
    \begin{example}
    	For $A=TM$ and $B=TN$ all LA-morphisms are differentials of (piecewise) smooth maps between the base manifolds. This follows from the anchor condition in the definition of LA-morphisms. Consequently, Definition \ref{def:hom1} agrees with the usual notion of smooth homotopy between manifolds.
    \end{example}

    \begin{remark}
    	In the special case of $A$-paths, Crainic and Fernandes use an e\-qui\-va\-lent definition of LA-homotopy in \cite{CrFer01}. The same definition has been used by Zhu and Brahic in \cite{BrZhu11}, where they introduce homotopy (bundles of) groups for Lie algebroids, generalizing the standard notion of homotopy groups of a smooth manifold. 
    \end{remark}
    
    \begin{definition}
    	\label{def:he}
    	Two Lie algebroids $A\to M$ and $B\to N$ are said to be \textbf{LA-homotopy equivalent} if there exist Lie algebroid morphisms $F\colon A \to B$ and $G\colon B \to A$ such that $F \circ G$ is LA-homotopic to the identity map $\text{Id}_B$ and $G \circ F$ is LA-homotopic to the identity $\text{Id}_A$. 
    \end{definition}
       \begin{remark}
    	\label{productHE}
    	It is immediate to show that if $F_1\colon A_1 \to B_1$ and $F_2 \colon A_2 \to B_2$ define LA-homotopy equivalences between Lie algebroids, then $F_1 \times F_2 \colon A_1 \times A_2 \to B_1 \times B_2$ defines an LA-homotopy equivalence between the product Lie algebroids. 
    \end{remark}
       
   \begin{example}
   	\label{ex:homotopyequivalent}
   Let $f \colon M \to N$ be a homotopy equivalence between manifolds and let $B$ be a Lie algebroid over an arbitrary base $P$. Then $TM \times B$ and $TN \times B$ are LA-homotopy equivalent Lie algebroids, via the LA-homotopy equivalence $F \colon TM \times B \to TN \times B$, $F:=Tf \times \operatorname{Id}_{B}$ (see also Remark \ref{productHE}). For instance, for $m > n \in \N^+$, and $\mathfrak{h}$ a Lie algebra, the Lie algebroid $T\R^n \times 0_{\R^{m-n}} \times \mathfrak{h} \to \R^m$ is LA-homotopy equivalent to $0_{\R^{m-n}} \times \mathfrak{h} \to \R^{m-n}$.
    \end{example}

    \subsection{LA-homotopy invariance of twisted Lie algebroid cohomology}
    \label{sub:proof}
This section explains and proves the main result of this paper.
    \begin{theorem}
     Consider two LA-homotopic LA-morphisms $(\Phi_0,\phi_0),(\Phi_1,\phi_1)\colon A \to B$ via an LA-homotopy
     \[\Phi\colon TI\times A\to B\]
     over $\phi\colon I\times M\to N$, as well as a vector bundle $E\to N$ with a flat $B$-connection $\nabla\colon \Gamma(B) \times \Gamma(E) \to \Gamma(E)$.
    	Then the maps in twisted cohomology induced by the Lie algebroid inclusions $\mathcal{I}_t\colon A \to TI \times A$, $a\mapsto (0_t,a)$ (see \eqref{inclusion}),     	
	\begin{equation}\label{def_I}
	\overline{\mathcal{I}_{t}^*} \colon \tH_{\Phi^*\nabla}^{\bullet}(TI \times A, \phi^!E) \to \tH_{\Phi_t^*\nabla}^{\bullet}(A, \phi_t^!E) 
	\end{equation}
    		are isomorphisms for all $t\in I$, and the following diagram commutes
	% https://q.uiver.app/#q=WzAsMyxbMCwwLCJIX1xcbmFibGEoQixFKSJdLFsyLDAsIkhfe1xcUGhpXzFeKlxcbmFibGF9KEEsXFxwaGlfMV4hRSkiXSxbMiwyLCJIX3tcXFBoaV8wXipcXG5hYmxhfShBLFxccGhpXzBeIUUpIl0sWzAsMSwiXFxvdmVybGluZXtcXFBoaV8xXip9Il0sWzEsMiwiXFx0aGV0YV8qIl0sWzAsMiwiXFxvdmVybGluZXtcXFBoaV8wXip9IiwyXV0=
\[\begin{tikzcd}
	{H_\nabla(B,E)} && {H_{\Phi_0^*\nabla}(A,\phi_0^!E)} \\
	\\
	&& {H_{\Phi_1^*\nabla}(A,\phi_1^!E).}
	\arrow["{\overline{\Phi_0^*}}", from=1-1, to=1-3]
	\arrow["{\overline{\Phi_1^*}}"', from=1-1, to=3-3]
	\arrow["{\overline{\mathcal{I}_{1}^*} \circ (\overline{\mathcal{I}_{0}^*})}^{-1}", from=1-3, to=3-3]
\end{tikzcd}\]
Moreover, there exists a gauge equivalence $\theta\colon \phi_0^!E\to \phi_1^!E$ from $\Phi_0^*\nabla$ to $\Phi_1^*\nabla$ such that \[\overline{\mathcal{I}_{1}^*} \circ (\overline{\mathcal{I}_{0}^*})^{-1}= \theta_*.\]
    In other words, $(\Phi_0,\phi_0),(\Phi_1,\phi_1)\colon A \to B$ induce, up to a gauge isomorphism, the same morphism in LA-cohomology.

    	\label{th:hi}
    \end{theorem}

         \begin{remark} 
        	The gauge equivalence realizing the map $\overline{\mathcal{I}_{1}^*} \circ (\overline{\mathcal{I}_{0}^*})^{-1}$in cohomology is explicitly constructed as follows. Consider the flat $TI\times A$-connection $\Phi^*\nabla$ on $\phi^!E$. Then the derivation $(\Phi^*\nabla)_{(\partial t, 0^A)}$ of $\phi^! E$ defines a \emph{linear vector field} $X\in\mx(\phi^!E)$, see Section \ref{sec:6}. Its flow $\Xi$ makes the diagram 
        	% https://q.uiver.app/#q=WzAsNyxbMCwwLCJcXHRpbGRlXFxPbWVnYTo9XFx7KFxcbGFtYmRhLCBlX3socixtKX0pXFxpbiBcXG1hdGhiYiBSXFx0aW1lcyBcXHBoaV4hRVxcbWlkIFxcbGFtYmRhICtyXFxpbiBJXFx9Il0sWzIsMCwiXFxwaGleIUUiXSxbMiwxXSxbMiwyLCJJXFx0aW1lcyBNIl0sWzAsMiwiXFxPbWVnYTo9XFx7KFxcbGFtYmRhLChyLG0pKVxcaW4gXFxtYXRoYmIgUlxcdGltZXMgKElcXHRpbWVzIE0pXFxtaWQgXFxsYW1iZGEgK3JcXGluIElcXH0iXSxbMCwzLCIoXFxsYW1iZGEsKHIsbSkpIl0sWzIsMywiKFxcbGFtYmRhK3IsbSkiXSxbMSwzXSxbMCw0XSxbMCwxLCJcXFBzaSJdLFs0LDMsIlxccHNpIiwyXSxbNSw2LCIiLDIseyJzaG9ydGVuIjp7InNvdXJjZSI6MTAsInRhcmdldCI6MTB9LCJzdHlsZSI6eyJ0YWlsIjp7Im5hbWUiOiJtYXBzIHRvIn19fV1d
        	\[\begin{tikzcd}
        		{\tilde\Omega:=\{(\lambda, e_{(r,m)})\in \mathbb R\times \phi^!E\mid \lambda +r\in I\}} && {\phi^!E} \\
        		&& {} \\
        		{\Omega:=\{(\lambda,(r,m))\in \mathbb R\times (I\times M)\mid \lambda +r\in I\}} && {I\times M} \\
        		{(\lambda,(r,m))} && {(\lambda+r,m)}
        		\arrow["\Xi", from=1-1, to=1-3]
        		\arrow[from=1-1, to=3-1]
        		\arrow[from=1-3, to=3-3]
        		\arrow["\xi"', from=3-1, to=3-3]
        		\arrow[shorten <=10pt, shorten >=10pt, maps to, from=4-1, to=4-3]
        	\end{tikzcd}\]
        	commute, and is by isomorphisms of vector bundles, see Lemma \ref{lemma_linear_flow} or \cite{Jotz24}.
        	The smooth map $\theta$ is then defined by 
        	\[ \theta\colon \phi_0^!E=\iota_0^!(\phi^!E)\to \phi_1^!E=\iota_1^!(\phi^!E), \qquad \theta(e)=\Xi(1, e)\quad \in \quad (\phi^!E)\arrowvert_{\{1\}\times M}\]
        	for all $e\in \phi_0^!E=\iota_0^!(\phi^!E)=(\phi^!E)\arrowvert_{\{0\}\times M}$.
        	Its smooth inverse is similarly defined by the restriction to $(\phi^!E)\arrowvert_{\{1\}\times M}$ of $\Xi_{-1}$.
        \end{remark}

        \begin{remark}\label{relationswithPedroswork}
        \begin{enumerate}
\item    	The LA-cohomology with values on the trivial representation was first  found to be LA-homotopy invariant by Balcerzak in \cite{Bal12}, as a Corollary of a Stokes' theorem proved there. For dg-manifolds over the same base manifold, LA-homotopy invariance is proven  in \cite{Cam20} via an equivalent definition.
   
    \item	
    Theorem \ref{th:hi} is given in \cite[Corollary 4]{Fr19} in the case where $TI\times A$ is the pullback of $B$ under a smooth homotopy $\varphi\colon I\times M\to N$ of the base manifolds, $\Phi$ is the induced Lie algebroid morphism from $\varphi^{!!}B$ to $B$ (see Lemma \ref{le:0}), and under the additional assumption that for all $t\in I$, the smooth map $\varphi_t\colon M\to N$ is transverse to $B$.

 In fact, some results studied in \cite{Fr19} (in particular Theorem 2 and Theorem 4 in this paper) yield a different proof of Theorem \ref{th:hi} in the context of submersions by Lie algebroids. 
    Consider the Lie algebroid $TI \times A \to I \times M$ and the surjective submersion $p \colon I \times M \to I$. This pair is a \emph{trivial submersion by Lie algebroids}, see \cite[Definition 1]{Fr19}. Assume that $TI\times A$ is represented on $E\to I\times M$. By \cite[Theorem 1]{Fr19} and the discussion before \cite[Theorem 4]{Fr19} 
the inclusion of Lie algebroids \[\iota_t^{!!}(TI \times A) \simeq \{0_t\} \times A\hookrightarrow TI \times A\] induces an isomorphism 
\[\overline{\mathcal{I}_{t}^*} \colon \tH_{\nabla}^{\bullet}(TI \times A, E) \to \tH_{\mathcal I_t^*\nabla}^{\bullet}(A, \iota_t^!E) \]
in twisted Lie algebroid cohomology, for each $t \in I$.
    Theorem 4 in \cite{Fr19} interprets as well the composition $\overline{\mathcal{I}_{1}^*} \circ (\overline{\mathcal{I}_{0}^*})^{-1}$ as parallel transport along $c=\id_I\colon I \to I$ with respect to the complete Ehresmann connection $TI \times 0_M$.  
    \end{enumerate}\end{remark}

	The first part of the proof of Theorem \ref{th:hi} consists in showing that taking its direct product with the tangent bundle of an open interval does not affect the twisted Lie algebroid cohomologies of a Lie algebroid. In other words, the map $\overline{\mathcal I_t^*}$ in \eqref{def_I} is an isomorphism for $t=0,1$. The precise statement is formalized in Theorem \ref{th:4} below. Note that for twists given by the standard representation on the trivial line bundle, i.e.~for the usual Lie algebroid cohomology Theorem \ref{th:4} is sufficient to conclude using Example \ref{example_pullback_trivial}.

    \begin{theorem}
    	\label{th:4}
    	Let $A \to M$ be a Lie algebroid and let $F \to M$ be a vector bundle with an $A$-representation $\nabla$. Then
    	\[\overline{\operatorname{Pr}_A^*}\colon \operatorname{H}_{\nabla}^{\bullet}(A,F) \rightarrow \operatorname{H}^{\bullet}_{\operatorname{Pr}_A^*\nabla}(TI \times A, \operatorname{pr}_M^!F)  \]
    	is a module isomorphism over the ring homomorphism $\overline{\operatorname{Pr}_A^*}\colon \operatorname{H}^{\bullet}(A) \to \operatorname{H}^{\bullet}(TI \times A)$, where $\operatorname{Pr}_A\colon TI \times A \to A$ and $\operatorname{pr}_M\colon I \times M \to M$ are the projection maps to the second components. 
	The inverse of $\overline{\Ppr_A^*}$ is given by \[
	\overline{\mathcal I_t^*}\colon \operatorname{H}^{\bullet}_{\operatorname{Pr}_A^*\nabla}(TI \times A, \operatorname{pr}_M^!F)\to \operatorname{H}_{\mathcal I_t^*\operatorname{Pr}_A^*\nabla}^{\bullet}(A,\iota_t^!\pr_M^!F)=\operatorname{H}_{\nabla}^{\bullet}(A,F)
	\] for any $t\in I$, with $\mathcal I_t$ defined as in \eqref{inclusion} and right-inverse to $\Ppr_A$.
    \end{theorem}
  
    The proof of this theorem relies on the following lemma, which is standard.
     \begin{lemma}
     	The space $\Omega^k(TI \times A, \operatorname{pr}^!_MF)$ is (locally) generated as a $C^\infty(I\times M)$-module by forms of the following type:
     	\begin{enumerate}[(i)]
     		\item  $\operatorname{Pr}_A^* \omega$    \label{item:dR1}

		% \otimes \operatorname{pr}^!_M e=\operatorname{Pr}_A^*(\omega\otimes e)$    \label{item:dR1}
     		\item $\dr t \wedge \operatorname{Pr}_A^* \eta$ \label{item:dR2}
		% \otimes \operatorname{pr}^!_M e=\dr t \wedge\operatorname{Pr_A}^*(\eta\otimes e)$     \label{item:dR2}
     	\end{enumerate}
     	with $\omega \in \Omega^k(A,F)$ and $\eta \in \Omega^{k-1}(A,F)$.
     	\label{le:1}
     \end{lemma}

     \begin{remark}
     Note that the cohomology maps $\overline{\mathcal I_t^*}$ in Theorem \ref{th:hi} and in Theorem \ref{th:4} are not the same maps, although both induced by the Lie algebroid morphism $\mathcal I_t$ in \eqref{inclusion}. The first one is the cohomology map defined by pulling back $\Phi^*\nabla\colon \Gamma(TI\times A)\times \Gamma(\varphi^!E)\to \Gamma(\varphi^!E)$ along $\mathcal I_t$, with $\nabla\colon \Gamma(B)\times \Gamma(E)\to\Gamma(E)$ a general Lie algebroid representation. The second one is defined by pulling back $\Ppr_A^*\nabla\colon \Gamma(TI\times A)\times \Gamma(\pr_M^!F)\to\Gamma(\pr_M^!F)$, for $\nabla\colon \Gamma(A)\times\Gamma(F)\to \Gamma(F)$ a Lie algebroid representation. However which version of $\overline{\mathcal I_t^*}$ is meant is clear from the context, so the notation is taken to be the same for simplicity.
%        	The notation 
%        	\[\overline{\mathcal{I}_{j,}^*}_{\Phi^*\nabla} \colon \tH_{\Phi^*\nabla}^{\bullet}(TI \times A, \phi^!E) \to \tH_{\Phi_j^*\nabla}^{\bullet}(A, \phi_j^!E)\]
%        	 is chosen to stress the domain of the map. In the first part of the proof of Theorem \ref{th:hi}, it is checked that the map is an isomorphism in cohomology if the domain is twisted via a product representation, i.e.~if the map is of the form
%        	\[\overline{\mathcal I_j^*}:=\overline{\mathcal{I}_{j,}^*}_{\operatorname{Pr}_A^*\tilde{\nabla}}\colon \tH_{\operatorname{Pr}_A^*\tilde{\nabla}}^{\bullet}(TI \times A, \operatorname{pr}_M^!F) \to \tH_{\nabla}^{\bullet}(A, F),\]
%        	for $j=0,1$ and for any flat $A$-connection $\tilde{\nabla}$ with values on a vector bundle $F$. This is done by defining an explicit homotopy operator. In the second part of the proof of Theorem \ref{th:hi}, it is proven that $\overline{\mathcal{I}_{j,}^*}_{\Phi^*\nabla}$ is an isomorphism. This is achieved by defining an explicit gauge equivalence between the representation $\Phi^*\nabla$ and a product representation, which then leads also to the interpretation of the map $\overline{\mathcal{I}_{1,}^*}_{\Phi^*\nabla} \circ (\overline{\mathcal{I}_{0,}^*}_{\Phi^*\nabla})^{-1}$ as a gauge equivalence.
        \end{remark}

      \begin{proof}[Proof of Theorem \ref{th:4}]
      	      		For simplicity, use the notation \[\nabla^{TI \times A}:= \operatorname{Pr}_A^*\nabla\colon\Gamma(TI\times A)\times \Gamma(\pr_M^!F)\to\Gamma(\pr_M^!F).\] Recall the Lie algebroid morphisms (inclusions) $\mathcal{I}_t\colon A \to TI \times A$, $\mathcal{I}_t(a):= (0_t, a)$ over  $\iota_t\colon M \to I \times M,\, \iota_t(m):=(t,m)$ for all $t\in I$.\medskip

	The map $\overline{\operatorname{Pr}^*_A}$ is injective because the identities $\pr_A \circ \iota_t= \text{id}_M$ and $\operatorname{Pr}_A \circ \mathcal{I}_t= \text{id}_A$  imply $\iota_t^!(\pr_A^!F)=(\pr_A \circ \iota_t)^!F=F$ and 
      	$\overline{\mathcal{I}_t^*} \circ \overline{\operatorname{Pr}^*_A}= \text{id}_{H^{\bullet}_\nabla(A, F)}$ in cohomology, with
      	\[\operatorname{H}_{\nabla}^{\bullet}(A,F)  \xrightarrow{\ov{\operatorname{Pr}^*_A}} \operatorname{H}_{\nabla^{TI \times A}}^{\bullet}(TI \times A, \pr_M^!F) \xrightarrow{\ov{\mathcal{I}_t^*}} \operatorname{H}_{\nabla}^{\bullet}(A,F).\]

To show the surjectivity of $\overline{\operatorname{Pr}^*_A}$, define as follows an explicit homotopy operator. 
Inspired by \cite{BoTu82} and using Lemma (\ref{le:1}), define the operator 
\[K\colon  \Omega^k(TI \times A, \pr_M^!F) \to \Omega^{k-1}(TI \times A, \pr_M^! F)\]
by
\[K (h\cdot\Ppr_A^*\omega):= 0\]
and   
\[K(g\cdot \dr t \wedge \Ppr_A^*\eta)\arrowvert_{(t,m)}:= (-1)^{k-1}\,\int_0^t g(s,m)ds\cdot (\Ppr_A^*\eta)\arrowvert_{(t,m)},\] for all $t \in I$ and $m\in M$,
where  $h,g \in C^{\infty}(I \times M)$, $\omega\in \Omega^k(A,F)$, $\eta\in\Omega^{k-1}(A,F)$.

\bigskip
Now check as follows that $K$ is indeed a homotopy operator.
On a  form $\xi=  h\cdot \text{Pr}_A^* \omega$ of type (\ref{item:dR1}), 
\[((\text{id}-\operatorname{Pr}_A^* \circ \mathcal{I}^*_0)(\xi))\arrowvert_{(t,m)}= (h(t,m)-h(0,m)) \cdot(\text{Pr}^*_A \omega)\arrowvert_{(t,m)}.\]
Let $k$ be the degree of $\omega$. Compute using $K(\xi)=0$
	\begin{equation*}
	 \begin{split}
	&(\dr_{\nabla^{TI \times A}}K-K\dr_{\nabla^{TI \times A}})(\xi)=-K\, \dr_{\nabla^{TI \times A}}(\xi)\\
	&\,=
	-K (\dr_{TI \times A}h\wedge \Ppr_A^*\omega+
	(-1)^k h\cdot \dr_{\nabla^{TI \times A}}(\Ppr_A^*\omega))\\
	&\,=-K (\dr_{TI \times A}h\wedge \Ppr_A^*\omega+
	(-1)^k h\cdot \Ppr_A^*(\dr_{\nabla}\omega))\\
	&\,=-K (\dr_{TI \times A}h\wedge \Ppr_A^*\omega).
	\end{split}
   \end{equation*}
Choose a local frame $(a_1,\ldots, a_r)$ of $A$ and the dual frame $(\alpha_1, \ldots, \alpha_r)$ of $A^*$.
   Then 
   \[ \dr_{TI \times A}h=\partial_t h\cdot \dr t+
   \sum_{j=1}^r\rho(a_j)(h)\cdot \Ppr_A^*\left(\alpha_j
   \right).
   \]
   Hence 
   \[(\dr_{\nabla^{TI \times A}}K-K\dr_{\nabla^{TI \times A}})(\xi) =-K (\partial_t h\cdot \dr t\wedge \Ppr_A^*\omega).
 \]  
which is 
   \begin{equation*}
	 \begin{split}
	& -(-1)^k \int_0^t \partial_sh(s,m) ds\cdot (\Ppr^*_A \omega)\arrowvert_{(t,m)}=-(-1)^k(h(t,m)-h(0,m)) (\Ppr^*_A \omega)\arrowvert_{(t,m)}
	\end{split}
   \end{equation*}
   when evaluated at $(t,m)\in I\times M$. This shows that 
   \[(\text{id}-\operatorname{Pr}_A^* \circ \mathcal{I}^*_0)(\xi)=(-1)^{|\xi|-1}\cdot (\dr_{\nabla^{TI \times A}}K-K\dr_{\nabla^{TI \times A}})(\xi)
   \]
for $\xi$ of type (\ref{item:dR1}).

\medskip

On a  form of type $\xi=g\cdot \dr t \wedge \text{Pr}_A^* \eta$ of type (\ref{item:dR2}) with $\eta\in\Omega^{k-1}(A,F)$, compute
\[(\text{id}- \text{Pr}_A^* \circ \mathcal{I}^*_0)(\xi)=\xi \]
because $\mathcal{I}^*_0(\dr t)=\dr_{\nabla}(\iota^*_0t)=\dr_{\nabla} (0)=0$.  The identity
\[\dr_{TI \times A}(g \dr t)= \sum_{j=1}^r\rho(a_j)(g)\cdot \Ppr_A^*(\alpha_j)\wedge \dr t \]
 yields
 
\begin{equation*}
	\begin{split}
	\dr_{\nabla^{TI \times A}}\xi
	=& -\sum_{j=1}^r\rho(a_j)(g)\dr t\wedge \Ppr_A^*(\alpha_j)\wedge \Ppr_A^*(\eta) -g\dr t \wedge \dr_{\nabla^{TI \times A}}(\Ppr_A^*\eta)\\
	=& -\sum_{j=1}^r\rho(a_j)(g)\dr t\wedge \Ppr_A^*(\alpha_j\wedge \eta)-g\dr t \wedge \Ppr_A^*(\dr_{\nabla}\eta)
	\end{split}
\end{equation*}
and so
\begin{equation*}
	\begin{split}
	K(\dr_{\nabla^{TI \times A}}\xi)\arrowvert_{(t,m)}&=-(-1)^{k} \sum_{j=1}^r\int_0^t \rho(a_j)(g)(s,m)ds\cdot \Ppr_A^*(\alpha_j\wedge \eta)\arrowvert_{(t,m)}\\
&\qquad -(-1)^{k}\int_0^t g(s,m)ds \cdot (\Ppr_A^*(\dr_{\nabla}\eta))\arrowvert_{(t,m)}.
     \end{split}
\end{equation*}

On the other hand, 
\[K(\xi)\arrowvert_{(t,m)}=(-1)^{k-1}\,\int_0^t g(s,m)ds\cdot  (\Ppr_A^* \eta)\arrowvert_{(t,m)},\] for all $t \in I$ and $m\in M$,
and the fact that $g$ is smooth yields as above in local coordinates
\begin{equation*} \begin{split}
	(\dr_{\nabla^{TI \times A}}(K(\xi)))\arrowvert_{(t,m)}%= \nabla^{TI \times A}((-1)^k\int_0^t g(t,m)dt\,  (\text{Pr}_A^* \omega_A \otimes \pr^!_M e))=\\ 
	=&\,
	(-1)^{k-1}\cdot \sum_{j=1}^r\int_0^t\rho(a_j)g(s,m)ds\cdot\Ppr_A^*(\alpha_j\wedge\eta)\arrowvert_{(t,m)} \\ 
	&+(-1)^{k-1}\cdot g(t,m)\cdot\dr t\wedge \Ppr_A^*(\eta)\arrowvert_{(t,m)} \\ 
	&+(-1)^{k-1}\int_0^t g(s,m)ds\cdot \Ppr_A^*(\dr_\nabla\eta)\arrowvert_{(t,m)}\\ 
	=&\,K(\dr_{\nabla^{TI \times A}}\xi)\arrowvert_{(t,m)}+(-1)^{k-1}\xi\arrowvert_{(t,m)}.
\end{split}\end{equation*}
This shows 
\[\dr_{\nabla^{TI \times A}}(K(\xi))- K(\dr_{\nabla^{TI \times A}}(\xi))=(-1)^{|\xi|-1}\cdot\xi, \]
and so
 \[(\text{id}-\operatorname{Pr}_A^* \circ \mathcal{I}^*_0)(\xi)=(-1)^{|\xi|-1}\cdot (\dr_{\nabla^{TI \times A}}K-K\dr_{\nabla^{TI \times A}})(\xi)
   \]
for $\xi$ of type (\ref{item:dR2}).
      \end{proof}

The second part of the proof of Theorem \ref{th:hi} is done in Section \ref{ap:A},  using results on the flows of linear vector fields on vector bundles and of linear Lie algebroid derivations, that are proved in Section \ref{sec:6}.

		 \bigskip
    This section ends with some examples.
    
    \begin{example}
    	\label{ex:Gbundles}
         Let $P \to M$ and $Q \to N$ be $G$-principal bundles and let $\phi_0,\phi_1 \colon P \to Q$ be $G$-equivariant maps. Suppose that there exists a smooth $G$-equivariant homotopy $\phi \colon I \times P \to Q$, such that $\phi(0,\cdot):=\phi_0$ and $\phi(1,\cdot):=\phi_1$, where the $G$-action is extended to $I \times P$ trivially on the first factor, then it is known that $\phi_0$ and $\phi_1$ induce the same map in equivariant de Rham cohomology (see e.g.~ \cite{GoZo19}). Theorem \ref{th:hi} generalizes this result to the equivariant de Rham cohomology with coefficients. The proof goes as follows. The above mentioned $G$-equivariant maps factor to maps $\bar{\phi}\colon I \times M \to N$ and $\bar \phi_0, \bar{\phi_1} \colon M \to N$. They induce LA-morphisms between the corresponding Atiyah Lie algebroids $\overline{T\phi} \colon A(I \times P) \cong TI \times A(P) \to A(Q)$ and $\overline{T\phi_0}, \overline{T\phi_1} \colon A(P) \to A(Q)$ (see Example \ref{ex:LAmorph}). By Theorem \ref{th:hi} conclude that $\overline{T\phi_0^*} \equiv \overline{T\phi_1^*}$ in the twisted Lie algebroid cohomology, which is isomorphic to the twisted equivariant de Rham cohomology (see Example \ref{ex:LAcohom}).
         \end{example} 
   \begin{example}
   	\label{ex:foliations}
    	It is well known that the foliated (de Rham) cohomology is not homotopy invariant (see \cite{Hae71}). 
   For instance, the bundle of rank zero $0_U \to U$ over a contractible manifold has zero differential, because both its bracket and its anchor map are zero. Its foliated de Rham cohomology therefore coincides with the space of forms. In particular, $\operatorname{H}^{0}(0_U)= C^{\infty}(U)$, and so the Poincar\'e Lemma cannot hold here.
    	
    	However, Theorem \ref{th:hi} shows that the foliated de Rham cohomology is LA-homotopy invariant, since foliations are special cases of Lie algebroids. The notion of LA-homotopy in the category of foliations reads as follows.
    	
    	Let $(M,F_M)$ and $(N, F_N)$ be foliated manifolds, i.e.~$F_M\subseteq TM$ and $F_N\subseteq TN$ are involutive subbundles. Consider smooth maps $\phi_0,\phi_1\colon M \to N$ such that 
    	$T\phi_j(F_M)\subseteq F_N$ for $j=0,1$, i.e.~Lie algebroid morphisms $T\phi_j\colon F_M\to F_N$ over $\phi_j\colon M\to N$. 
    	Equivalently, for each leaf $L$ of $F_M$ there exists leaves $K_0$ and $K_1$ of $F_N$ such that $\phi_j(L) \subseteq K_j$, for $j=0,1$.
    	Note that in the following, \emph{a leaf of $F_M$} is short for \emph{a maximal integral leaf of $F_M$}.
    	
    	Then if $\phi_0$ and $\phi_1$ are LA-homotopic, then there exists a Lie algebroid morphism $\Phi\colon TI \times F_M \to  F_N$
    	over a smooth map $\phi\colon I \times M \to N$. But then the anchor condition 
    	% https://q.uiver.app/#q=WzAsNixbMCwxLCJUSVxcdGltZXMgRl9NIl0sWzIsMSwiRl9OIl0sWzAsMiwiSVxcdGltZXMgTSJdLFsyLDIsIk4iXSxbMSwwLCJUSVxcdGltZXMgVE0iXSxbMywwLCJUTiJdLFsyLDMsIlxccGhpIiwxXSxbMCwxLCJcXFBoaSJdLFswLDJdLFsxLDNdLFswLDQsIiIsMSx7InN0eWxlIjp7InRhaWwiOnsibmFtZSI6Imhvb2siLCJzaWRlIjoidG9wIn19fV0sWzEsNSwiIiwxLHsic3R5bGUiOnsidGFpbCI6eyJuYW1lIjoiaG9vayIsInNpZGUiOiJ0b3AifX19XSxbNCw1LCJUXFxwaGkiLDFdLFs0LDIsIiIsMSx7InN0eWxlIjp7ImJvZHkiOnsibmFtZSI6ImRvdHRlZCJ9fX1dLFs1LDNdXQ==
    	\[\begin{tikzcd}
    		& {TI\times TM} && TN \\
    		{TI\times F_M} && {F_N} \\
    		{I\times M} && N
    		\arrow["T\phi"{description}, from=1-2, to=1-4]
    		\arrow[dotted, from=1-2, to=3-1]
    		\arrow[from=1-4, to=3-3]
    		\arrow[hook, from=2-1, to=1-2]
    		\arrow["\Phi", from=2-1, to=2-3]
    		\arrow[from=2-1, to=3-1]
    		\arrow[hook, from=2-3, to=1-4]
    		\arrow[from=2-3, to=3-3]
    		\arrow["\phi"{description}, from=3-1, to=3-3]
    	\end{tikzcd}\]
    	shows that $\Phi=T\phi\arrowvert_{TI\times F_M}$ and that $T\phi(TI\times F_M)\subseteq F_N$.
    	A leaf of $TI\times F_M\subseteq TI\times TM$ is a product $I\times L$ with $L$ a leaf of $F_M$. The inclusion 
    	$T\phi(TI\times F_M)\subseteq F_N$ implies that for each leaf $L$ of $F_M$, there exists a leaf $K\subseteq N$ such that 
    	\[ \phi(I\times L)\subseteq K.
    	\]
    	In particular, 
    	\[ \phi_j(L)=\phi(\{j\}\times L)\subseteq K
    	\]
    	for $j=0,1$, i.e.~$\phi_0$ and $\phi_1$ map each leaf of $F_M$ in  the same leaf of $F_N$.
    	
    	This shows that a LA-homotopy between morphisms $\phi_0,\phi_1\colon (M,\mathcal F_M)\to (N,\mathcal F_N)$ of foliated manifolds 
    	is a \textbf{homotopy along the foliations}, i.e.~a smooth map $\phi\colon I\times M\to N$ such that $\phi\circ\iota_j=\phi_j$ for $j=0,1$ and for each leaf $L$ of $\mathcal F_M$ there exists a leaf $K$ of $\mathcal F_N$ such that $\phi(I\times L)\subseteq K$. By Theorem \ref{th:hi}, if such a smooth homotopy exist, then the maps induced by $T\phi_0\arrowvert{F_M}$ and $T\phi_1\arrowvert{F_M}$ coincide in foliated cohomology.

 LA-homotopy in the case of foliations coincides with the notion of integrable homotopy, introduced by Haefliger in \cite{Hae71}. The name integrable homotopy is first used in \cite{Az83}. Foliated cohomology is known there to be invariant under integrable homotopy, see \cite{Schl21} for more details.
    	    \end{example}

    \subsection{Pullback of Lie algebroids via homotopies}
    \label{subs:3.60}

%    This subsection studies the homotopy invariance of pullback Lie algebroids
    Pullback Lie algebroids are useful to systematically construct examples of LA-morphisms, and LA-homotopies (see Corollary \ref{cor:00}). In the transitive case, similar constructions are done in \cite{Mein21}. 
    %The same proving techniques are carefully adapted to the non-transitive case. %The techniques are in the spirit of the study of local normal form for Lie algebroid transversal in \cite{BLM19}. But slighly different

    \begin{proposition}\label{pr:mein}
	Let $B \to N$ be a Lie algebroid and let  $M$ be a smooth manifold. Consider a smooth homotopy 
	\begin{equation*}
		\begin{split}
			\phi \colon I \times M &\to N, \qquad (t, m) \mapsto \phi_t(m) 
		\end{split}
	\end{equation*}
	such that $\phi^{!!}B$ exists.
	%and  $\phi_t^{!!}B$ exist for each $t\in I$. 
	If there exists a smooth section $\tilde y$ of $\phi^{!!}B$ of the form $(s,m) \mapsto (\partial_t\an{s}, 0_m,y(s,m)) \in (T_sI \times T_mM) \times_{TN} B_{\phi(t,m)} $, for some $y \colon I \times M \to B$, then the induced pullback Lie algebroids
	$\phi^{!!}_tB \to M$ are all isomorphic.
\end{proposition}

Note that in the proposition above,  the existence of the section $\tilde{y}$  implies that $TI \times 0 \subseteq \operatorname{Im}(\rho_{\phi^{!!}B})$. As a consequence, the inclusions $\iota_t \colon M \to I \times M$, $m\mapsto (t,m)$, are transversal to the anchor map $\rho_{\phi^{!!}B}$, and so the pullback Lie algebroids $\iota_t^{!!}\phi^{!!}B\cong \phi_t^{!!}B$ exist for all $t\in I$.

\begin{remark}\label{remark_pr:mein}
	In Proposition \ref{pr:mein}, a necessary condition for the existence of the section $\tilde{y}$ is $TI \times 0 \subseteq \operatorname{Im}\rho_{\phi^{!!}B}$.
	This is equivalent to $\operatorname{Im}(T\phi\an{(TI \times 0)}) \subseteq \operatorname{Im}(\rho_B)$, i.e.~to the homotopy $\phi$ deforming along the, possibly singular, foliation defined by the Lie algebroid.  Indeed, recall that
	\[\phi^{!!}B:=\{((r,v),b) \in T_tI \times T_mM \times B_{\phi(t,m)} \mid 
	T_{(t,m)}\phi(r,v)= \rho_B(b), (t,m)\in I \times M\}\] with anchor map $\rho_{\phi^{!!}B}((r,v),b):=(r,v)$. The condition  $\text{Im}(T\phi\an{(TI \times 0)}) \subseteq \text{Im}(\rho_B)$ means that for each vector $r \in T_tI$ and each $m\in M$ there exists $b \in B_{\phi(t,m)}$ such that $T\phi(r,0_m)= \rho_B(b)$. Then $(r,0_m)= \rho_{\phi^{!!}B}((r,0_m),b)\in \text{Im}(\rho_{\phi^{!!}B})$. 
	This shows the inclusion $TI \times 0 \subseteq \text{Im}(\rho_{\phi^{!!}B})$. The inverse implication follows from the definition of the anchor map of the pullback Lie algebroid. For each vector $r \in T_tI$ and each $m \in M$ the condition $(r,0_m) \in \operatorname{Im}(\phi^{!!}B)$ implies that there exists an element $b \in B_{\phi(t,m)}$ such that $T\phi(r,0_m)= \rho_B(b)$. 
\end{remark}

\begin{proof}[Proof of Proposition \ref{pr:mein}]
Consider the derivation $D:=[\tilde{y}, \cdot]$ of the Lie algebroid $\phi^{!!}B$, with symbol $X:=\frac{\partial}{\partial t}\in\mx(I\times M)$. Then $\widehat{D}\in\mx^l(\phi^{!!}B)$ is linear over $X$.
	Set $\Omega:=\{(t,s)\in\mathbb R\times I\mid t+s\in I\}$. Since $X$ has the flow 
	\[\psi\colon \Omega\times M \ni (t, s,m) \mapsto (s+t, m) \in I \times M,\]
	the flow $\Psi$ of $\widehat{D}\in\mathfrak{X}^l(\phi^{!!}B)$ is
	a map
	\[ \Psi\colon \Omega\times \phi^{!!}B\to \phi^{!!}B
	\]
	over $\psi$ by Lemma \ref{lemma_linear_flow}.
		By Proposition \ref{invariance_a}, for each $t \in \mathbb R$ the map \[\Psi_t\colon \phi^{!!}B\arrowvert_{(I\cap(I-t))\times M}\to \phi^{!!}B\arrowvert_{(I\cap(I+t))\times M}\] is further an isomorphism of Lie algebroids over \[\psi_t\colon (I\cap(I-t))\times M\to (I\cap(I+t))\times M, \qquad(s,m)\mapsto (s+t,m).
\]
	
	Take $s\in I$ and consider the Lie algebroid $\phi_s^{!!}B\to M$. The underlying vector bundle is given by
	\[\phi_s^{!!}B:=\{(v,b)\in T_mM\times B_{\phi_s(m)}\mid T_m\phi_s(v)=\rho_B(b), m\in M\}.	\]
	Since $T_m\phi_s(v)=T_{(s,m)}\phi(0_s,v)$ for all $m\in M$ and $v\in T_mM$, the vector bundle $\phi_s^{!!}B$ comes with the embedding $\phi_s^{!!}B\hookrightarrow \phi^{!!}B$
	\[ (v,b)\mapsto (0,v,b)
	\]
	 over $M\hookrightarrow I\times M$, $m\mapsto (s,m)$. This embedding is a morphism of Lie algebroids by definition. Take $t\in I$. The image of $\phi_s^{!!}B$ under $\Psi_{t-s}$ equals $\phi_{t}^{!!}B$: for all $(v,b)\in \phi_s^{!!}B$
	 \[ \rho_{\phi^{!!}B}(\Psi_{t-s}(0_s,v,b))=T\psi_{t-s}(\rho_{\phi^{!!}B}(0_s,v,b))=T\psi_{t-s}(0_s,v)=(0_{t},v).
	 \]
	 Hence $\Psi_{t-s}(0_s,v,b)=(0_{t},v,c)$ for some $c\in B_{\phi(t,m)}$, i.e.~$\Psi_{t-s}(0_s,v,b)\in \phi_{t}^{!!}B\subseteq \phi^{!!}B$. Since $\Psi_{t-s}\colon \phi^{!!}B\to \phi^{!!}B$ is an isomorphism of Lie algebroids, it thus restricts to an isomorphism of the Lie algebroids
	 \[ \Psi_{t-s}\arrowvert_{\phi_s^{!!}B}\colon \phi_s^{!!}B\to \phi_{t}^{!!}B
	 \] 
	 over $\psi_{t-s}\colon \{s\}\times M\simeq M\to \{t\}\times M\simeq M$.
		\end{proof}

\begin{remark}\label{re:1}
Proposition \ref{pr:mein} holds for a regular Lie algebroid $B\to N$ and a smooth homotopy $\phi \colon I \times M \to N$, with $\text{Im}(T\phi\an{(TI \times 0)}) \subseteq \text{Im}(\rho_B)$ and such that $\phi^{!!}B$  exists. Indeed, since $B\to N$ is a regular Lie algebroid and $q \colon \phi^{!!}B \to I \times M$ is its (by assumption existing) pullback Lie algebroid, the latter is also regular. More precisely, the canonical Lie algebroid morphism 
	\[ \phi^{!!}B\to B, \qquad T_tI\times T_mM\times B_{\varphi_t(m)}\ni (r,v,b)\mapsto b
	\]
	over $\phi\colon I\times M\to N$ restricts fibrewise to vector space isomorphisms 
	\[ \text{ker}(\rho_{\phi^{!!}B})\arrowvert_{(t,m)} \to  \text{ker}(\rho_B)\arrowvert_{\varphi_t(m)}
	\]
	for all pairs $(t,m)\in I\times M$.\\
    A consequence of the regularity of $\phi^{!!}B$ is that
	\[\text{ker}(\rho_{\phi^{!!}B}) \hookrightarrow \phi^{!!}B \overset{\rho_{\phi^{!!}B}}{\longrightarrow}  \text{Im}(\rho_{\phi^{!!}B})\]
	is a short exact sequence of vector bundles over $I\times M$.
	Choose a smooth vector bundle morphism $j \colon \text{Im}(\rho_{\phi^{!!}B}) \to \phi^{!!}{B}$ splitting that sequence, i.e.~such that 
	\[\text{Id}_{\operatorname{Im}(\rho_{\phi^{!!}B})}= \rho_{\phi^{!!}B} \circ j\]
	Define a smooth section of the vector bundle $\phi^{!!}B$ by lifting $X:=\frac{\partial}{\partial t}$ to $j(X) \in \Gamma(\phi^{!!}B)$. Denote it by $\tilde{y}:=j(X)$ and apply Proposition \ref{pr:mein}.

\end{remark}
  
    \begin{remark}
    	\label{re:mein}
    	The condition $\text{Im}(T\phi\an{(TI \times 0)}) \subseteq \text{Im}(\rho_B)$ cannot be dropped. Choose a non-abelian Lie algebra $\mathfrak{g}$ and define the regular Lie algebroid with underlying vector bundle
    	\[B:=I\times \mathfrak{g} \to I, \]
    	the zero anchor and the (pointwise) Lie bracket $[(t,v), (t,w)]:=(t,t[v,w]_{\mathfrak{g}})$.
    	Choose the smooth homotopy $\phi = \text{Id}_{I} \colon I\to I$. 
	Then it is easy to see that $\phi^{!!}B=0_I\times_I B\simeq B$, and that $\phi_t^{!!}B=0_{\{t\}}\times B_t\simeq \{t\}\times \mathfrak g\simeq \mathfrak g$ for all $t\in I$, with the Lie bracket
	\[ [v,w]_t=t[v,w]
	\]
	for all $v,w\in\mathfrak g$.
	Then $\phi^{!!}_0B$ is an abelian algebra and $\phi^{!!}_1B$ is not, since it is the Lie algebra $\mathfrak g$. Therefore, $\phi^{!!}_0B$ and $\phi^{!!}_1B$ cannot be isomorphic. 
    \end{remark}
    
    \bigskip
   
    If $B\to N$ is a transitive Lie algebroid, pullbacks always exist and are transitive, and the required inclusion between the images of $T \phi \an{(TI \times 0)}$ and $\rho_B$ holds. This particular case of Proposition \ref{pr:mein} and Remark \ref{re:1} is known, see for instance \cite{Mein21}.

Proposition \ref{pr:mein} yields the following splitting  result for transitive Lie algebroids. 
The paper \cite{Mein21} proves Proposition \ref{pr:mein} in the transitive case, as well as this corollary, so its proof is omitted here.
See also \cite{Mackenzie05} and \cite{CrFer01} for different proofs.
	 	\begin{corollary}
	 		\label{localtriviality}
	 		Let $B \to N$ be a transitive Lie algebroid with anchor $\rho_B\colon B\to TN$. For each $x \in N$ the vector space $\operatorname{ker}(\rho_B)_x=:\mathfrak g_x$ inherits a Lie algebra structure, the \textbf{isotropy Lie algebra of $B$ at $x$}. 
	 		\begin{enumerate}
	 			\item If $N$ is connected, then for each $x,y \in N$, $\operatorname{ker}(\rho_B)_x$ and $\operatorname{ker}(\rho_B)_y$ are isomorphic Lie algebras. 
				\end{enumerate}
				Denote the obtained constant isotropy Lie algebra by $\mathfrak g$. It  is the \textbf{isotropy Lie algebra of $B$}.
\begin{enumerate}\setcounter{enumi}{1}
	 			\item If $N$ is contractible, then $B$ is isomorphic to the direct product Lie algebroid $TN \times \mathfrak{g}$.
	 		\end{enumerate}
	 	\end{corollary}
		
%					 	\begin{proof}
%				\begin{enumerate}
%				\item Take a smooth curve $\gamma\colon I\to M$ with $\gamma(0)=x$ and $\gamma(1)=y$. Then $\gamma$ can be considered as a smooth homotopy $I\times \{*\}\to M$ with $\gamma_0(*)=x$ and $\gamma_1(*)=y$. Then 
%				\[ \gamma_0^{!!}B=\mathfrak g_x \qquad \text{ and } \qquad  \gamma_1^{!!}B=\mathfrak g_y.
%				\]
%		By Proposition \ref{pr:mein}, $\gamma_0^{!!}B \cong \gamma_1^{!!}B$ as Lie algebroids, so $\mathfrak g_x$ and $\mathfrak g_y$ are isomorphic Lie algebras. This proves the first part.\\
%	 		
%\item	 		Fix $x\in N$ and set $\ker(\rho_B)_x=: \mathfrak{g}$. Then  the pullback under the smooth map $h \colon N \to \{x\}$ of $\mathfrak g$ is 
%$h^{!!}\mathfrak g=TN \times \mathfrak{g}$. Consider the smooth inclusion map $f\colon \{x\}\to N$. Then $f^{!!}B=\mathfrak g$. As a consequence, the pullback of $B$ under $f \circ g \colon N \to N$, $z\mapsto x$ for all $z\in N$, is $(f\circ g)^{!!}B=TN \times \mathfrak{g}$. Since $f \circ g$ is smoothly homotopic to the identity map of $N$, apply again Proposition \ref{pr:mein} and conclude that $B \cong TN \times \mathfrak{g}$.
%\end{enumerate}
%	 	\end{proof}

    	Consider a transitive Lie algebroid $B$ over a contractible base manifold $N$. Then as above $B\simeq TN \times \mathfrak{g}$, where $\mathfrak{g}$ is the isotropy Lie algebra of $B$ at a fixed $x\in N$. Example \ref{ex:homotopyequivalent} then yields that $B$ and $\mathfrak{g}$ are LA-homotopy equivalent. In particular, this provides a proof for the generalized Poincar\'e Lemma, see Lemma \ref{Th:PL} later on.

  \subsection{A construction of examples}
  \label{subs:3.6}  
 Proposition \ref{pr:mein} can be used to construct examples of LA-homotopies.

	\begin{corollary}
		\label{cor:00}
		In the situation of Proposition \ref{pr:mein}, 
		\begin{enumerate}
			\item there exists an LA-isomorphim $\Pi\colon TI \times \phi_0^{!!}B\to \phi^{!!}B$ over the identity on $I\times M$, and 
			\item the composition $p_{B,\phi}\circ \Pi$ 
					  defines an LA-homotopy 
			\begin{equation*}
				\begin{xy}
					\xymatrix{TI \times \phi_0^{!!}B \ar[rr]^{\,\,p_{B,\phi}\circ \Pi} \ar[d]&& B \ar[d]\\
					I\times M \ar[rr]_{\phi} && N,}
				\end{xy}
			\end{equation*} between $p_{B,\phi} \circ \Pi \circ \mathcal{I}_0 \colon \phi^{!!}_0 B \to B$ and $p_{B,\phi} \circ \Pi \circ \mathcal{I}_1 \colon \phi^{!!}_0B \to B$.
		\end{enumerate}
	\end{corollary}
	\begin{proof}		Consider the canonical Lie algebroid morphism
			\begin{equation*}
				\begin{xy}
					\xymatrix{\phi^{!!}B:= (TI \times TM)\times_{TN}B  \ar[r]^{\qquad \qquad \qquad p_{B,\phi}}\ar[d]& B \ar[d]\\
					I\times M \ar[r]_{\phi} & N,}
				\end{xy}
			\end{equation*}
			see Lemma \ref{le:0}.
			Since $\phi_t$ equals $\phi\an{\{t\} \times M}$, the vector bundle $\phi_t^{!!}B$ can be identified with the subbundle  $((\{0_t\} \times TM) \times_{TN}B)$ of $\phi^{!!}B$ over $\{t\}\times M$, and the Lie algebroid morphism $p_{B,\phi}$ above restricts to the Lie algebroid morphism $p_{B,\phi_t}= p_{B,\phi}\an{(\{0_t\} \times TM) \times_{TN}B}$.
			
			Moreover, $\phi$ is homotopic to the function  $p \colon I \times M \to N$, $p(t,m):=\phi_0(m)$ via the homotopy \[\alpha \colon I \times (I \times M)\to N, \qquad \alpha(s,t,m):= \phi(st, m).\] Define the surjective submersions $\pi \colon  I^2\times M \to   I\times M$, $((s,t),m)\mapsto (st,m)$ and $\pi_s \colon I\times M\to I\times M$, $(t,m) \to (st,m)$ (for $s\ne 0$). Then $\alpha= \phi \circ \pi$ and $\alpha_s=\phi \circ \pi_s$ for all $s\in I\setminus\{0\}$. Since the pullback of a Lie algebroid under a surjective submersion always exists, the pullback Lie algebroids $\alpha^{!!}B$ and $\alpha_s^{!!}B$ exist by functoriality, for  all $s\in I\setminus\{0\}$. For $s=0$, $\alpha_0^{!!}B$ is obtained by pulling back $\phi_0^{!!}B$ via the map $0 \times \operatorname{id}_M\colon I\times M\to \{0\}\times M$.
			Since any Lie algebroid $A\to \{0\}\times M\simeq M$ pulls back under $0\times \id_M$ to the Lie algebroid $TI\times A\to I\times M$,
			 also $\alpha_0^{!!}B$ exists as a pullback Lie algebroid. 
			 
			 Observe that $(T\alpha)(TI \times 0 \times 0) = (T\phi)(TI \times 0) \subseteq \rho_B(B)$ by definition of $\alpha$. All the hypotheses of Proposition \ref{pr:mein} for the homotopy $\alpha$ are hence valid, and so 
			 $\alpha_0^{!!}B=p^{!!}B=TI\times \phi_0^{!!}B$ and $\alpha_1^{!!}B= \phi^{!!}B$ are isomorphic Lie algebroids, via an  isomorphism $\Pi \colon TI \times \phi_0^{!!}B \to \phi^{!!}B$ over the identity on $I\times M$. (Which is not unique in general, and can be constructed as in the proof of Proposition \ref{pr:mein}.) This proves the first statement.
			 
			 \medskip
			Composing $p_{B,\phi}$ with $\Pi$ defines  the LA-homotopy
	         \[\Phi:= p_{B,\phi}\circ \Pi\colon TI \times \phi_0^{!!}B \to B\]
	         between
	         \[
	         \Phi\circ\mathcal I_0=p_{B,\phi}\circ \Pi\circ \mathcal I_0 \colon \phi_0^{!!} B \to B 
	         \]
	        and
	          \[
	         \Phi\circ\mathcal I_1=p_{B,\phi}\circ \Pi\circ \mathcal I_1\colon \phi_0^{!!} B \to B 
	         \]

\end{proof}

\begin{lemma}
	\label{cor:000}
		In the situation of Proposition \ref{pr:mein}, assume that $M=N$ and $\phi_0= \operatorname{id}_N$. Then there always exists an LA-isomorphism $\tilde{\Pi} \colon TI \times B \to \phi^{!!}B$ such that the composition $p_{B,\phi} \circ \tilde{\Pi} \colon TI \times B \to B$ defines an LA-homotopy between $\operatorname{Id}_B\colon B \to B$ and $p_{B,\phi} \circ \tilde{\Pi} \circ \mathcal{I}_1\colon B \to B$.   
\end{lemma}
\begin{proof}
Since $\phi_0=\id_N$, the Lie algebroid $B\to N$ is canonically isomorphic to $\phi_0^{!!}B\to N=M$ via the identification  $B \to \phi_0^{!!}B=\operatorname{id}_N^{!!}B$, $b \mapsto (\rho_B(b),b)$.
Applying Corollary \ref{cor:00} to $\phi$ and $B\to N$ then yields an LA-morphism
	 \[\Phi \colon TI \times B \xrightarrow{\Pi} \phi^{!!}B \xrightarrow{p_{B,\phi}} B,\]
	\[(r_t,b_n) \mapsto (r_t, \rho_{B}(b_n), y(r_t,b_n)_{\phi(t,n)}) \mapsto y(r_t,b_n)_{\phi(t,n)},\]
	defined over $\phi \colon I \times N \to N$, with $\Pi$ a (non-canonical) Lie algebroid isomorphism over the identity map $\operatorname{id}_{I \times N}$. For each pair $(r_t, b_n)\in TI\times B$, the element
	  $y(r_t,b_n)_{\phi(t,n)} \in B_{\phi(t,n)}$ satisfies \[T\phi(r_t, \rho_B(b_n))= \rho_B(y(r_t,b_n)_{\phi(t,n)}).\] $\Phi$ is an LA-homotopy between 
	$\Phi \circ \mathcal{I}_0=p_{B, \phi} \circ \Pi \circ \mathcal{I}_0$, defined over $\operatorname{id}_N$, and $\Phi \circ \mathcal{I}_1=p_{B, \phi} \circ \Pi \circ \mathcal{I}_1$, defined over $\phi_1$.
	
	It is then possible to modify the above as follows to another LA-homotopy $\tilde{\Phi}\colon TI \times B \to B$, such that  $\tilde{\Phi} \circ \mathcal{I}_0=\operatorname{Id}_B$ and $\tilde{\Phi} \circ \mathcal{I}_1=p_{B,\phi_1} \circ \tilde{\Pi} \circ \mathcal{I}_1$, for some isomorphism $\tilde{\Pi}\colon TI \times B \to \phi^{!!}B$.
	First prove that the map
	\[\Phi \circ \mathcal{I}_0 \colon B \xrightarrow{\mathcal{I}_0} TI \times B \xrightarrow{\Pi} \phi^{!!}B \xrightarrow{p_{B,\phi}} B \]
	\[b_n \mapsto (0_0, b_n) \mapsto (0_0, \rho_B(b_n),y(0_0,b_n) ) \mapsto y(0_0,b_n)\in B_n\]
	is an LA-automorphism of $B$ over the identity map. In particular, then
	$\rho_B(b_n)= \rho_B(y(0_0,b_n))$.
	
	Since $\Pi$ is an isomorphism, $\Pi \circ \mathcal{I}_0 = \Pi\an{0_0 \times B}\colon 0_0 \times B \to \phi^{!!}B$ is injective and it is an isomorphism to its image, which is contained in $\phi^{!!}_0(B) \subseteq \phi^{!!}B$. Since $\phi_0=\id_N$, $\phi_0^{!!}B\simeq B$ and so $\Pi\circ\mathcal I_0$ must map surjectively to $\phi_0^{!!}B$ for dimension reasons.
	
	Since $\phi_0=\id_N$, the LA-morphism $p_{B,\phi} \an{\phi^{!!}_0B} \colon \phi^{!!}_0B \to B$ is the isomorphism
	$p_{B,\phi_0}\colon \phi_0^{!!}B\to B$ with inverse
	$b \mapsto (\rho_B(b), b)$. This shows that $\Phi\circ\mathcal I_0$ is an isomorphism of Lie algebroids.

	\bigskip
	Use the inverse $(\Phi\circ \mathcal I_0)^{-1}$  of the Lie algebroid automorphism $\Phi\circ \mathcal I_0$ to define the new isomorphism 
	$\tilde \Pi\colon TI\times B\to \phi^{!!}B$ of Lie algebroid as in the following diagram
	% https://q.uiver.app/#q=WzAsMyxbMCwwLCJUSVxcdGltZXMgQiJdLFszLDAsIlRJXFx0aW1lcyBCIl0sWzMsMiwiXFxwaGleeyEhfUIiXSxbMSwyLCJcXFBpIl0sWzAsMSwiXFxvcGVyYXRvcm5hbWV7aWR9X3tUSX1cXHRpbWVzKFxcUGhpXFxjaXJjXFxtYXRoY2FsIElfMCleey0xfSJdLFswLDIsIlxcdGlsZGVcXFBpIiwyXV0=
\[\begin{tikzcd}
	{TI\times B} &&& {TI\times B} \\
	\\
	&&& {\phi^{!!}B}
	\arrow["{\operatorname{id}_{TI}\times(\Phi\circ\mathcal I_0)^{-1}}", from=1-1, to=1-4]
	\arrow["{\tilde\Pi}"', from=1-1, to=3-4]
	\arrow["\Pi", from=1-4, to=3-4]
\end{tikzcd}\]
	\[(r_t,b_n) \mapsto (r_t, \rho_B(b_n), y(r_t, (\Phi \circ \mathcal{I}_0)^{-1}(b_n))) \] 
	over the identity map $\operatorname{id}_{I \times N}$. 
		It is an isomorphism of Lie algebroids as a composition of Lie algebroid isomorphisms over the identity on $I\times N$.
	
	Use $\tilde{\Pi}$ to define the LA-homotopy 
	\[\tilde{\Phi} \colon TI \times B \xrightarrow{\tilde{\Pi}} \phi^{!!}B \xrightarrow{p_{B,\phi}} B\]
	\[(r_t, b_n) \mapsto (r_t, \rho_B(b_n), y(r_t, (\Phi \circ \mathcal{I}_0)^{-1}b_n) \mapsto y(r_t, (\Phi \circ \mathcal{I}_0)^{-1}b_n) \]
	over $\phi\colon I\times N\to N$.
	The map $\tilde{\Phi}$ can also be written  as
	\[\tilde{\Phi}(r_t, y):= \Phi(r_t, (\Phi \circ \mathcal{I}_0)^{-1}(y)). \]
	It is an LA-homotopy between \[\tilde{\Phi} \circ \mathcal{I}_0= p_{B,\phi} \circ \tilde{\Pi} \circ \mathcal{I}_0= \Phi \circ \mathcal{I}_0 \circ (\Phi \circ \mathcal{I}_0)^{-1}=\operatorname{Id}_B\] and 
	\[\tilde{\Phi} \circ \mathcal{I}_1=p_{B,\phi} \circ \tilde{\Pi} \circ \mathcal{I}_1= p_{B,\phi}\an{\phi^{!!}_1B} \circ \tilde{\Pi} \circ \mathcal{I}_1=  p_{B,\phi_1} \circ \tilde{\Pi} \circ \mathcal{I}_1.\]
\end{proof}

Note that later Lemma \ref{cor:000} will find a direct application in Theorem \ref{cor:0}.

\medskip

 Finally, Corollary \ref{cor:00} yields examples of LA-homotopy between Atiyah Lie algebroids.
    \begin{example}
    	Let $A(P) \to N$ be an Atiyah Lie algebroid and let $\phi \colon I \times M \to N$ be a smooth homotopy. Corollary \ref{cor:00} implies that there exists a Lie algebroid of the form
    	\[TI \times A(\phi_0^*P) \to A(P).\]
 
 This is an immediate consequence of the well known fact that $A(\phi_0^*P) \cong \phi_0^{!!}A(P)$ (see for instance \cite{Mackenzie05}) and Corollary \ref{cor:00}. All hypotheses are satisfied because Atiyah Lie algebroids are transitive.
      \end{example}

        \section{The Lie algebra case} 
        
Lie algebras are prototypical examples of Lie algebroids. They are probably the most ``rigid'' class of Lie algebroids, i.e.~where examples of LA-homotopies are the rarest. Nevertheless, it is interesting to measure how restrictive this notion is and how it relates to the classical deformation theory of morphisms of Lie algebras.This section shows that the study of LA-homotopy in the subcategory of Lie algebras provides an equivalent characterization expressed in classical Lie theory terms, so accessible to a reader with no interest in general Lie algebroids.

       \subsection{LA-homotopy for Lie algebras}
    \label{subs:3.1}
    This section proves that LA-homotopies between Lie algebra morphisms are smooth curves of Lie algebra morphisms, the variations of which are controlled by the adjoint representation of the target Lie algebras.

    \begin{proposition}\label{homotopy_lie_algebras}
    	Let $\mathfrak{g}$ and $\mathfrak{h}$ be Lie algebras. Consider a vector bundle morphism
    	\begin{equation*}
    		\begin{xy}
    			\xymatrix{TI \times \mathfrak{g} \ar[r]^{\Phi}\ar[d]& \mathfrak{h} \ar[d]\\
    				I\times \{*\} \ar[r]_{\phi} & \{*\}}
    		\end{xy}
    	\end{equation*}
    Then $(\Phi,\phi)$ is a LA-morphism, i.e.~an LA-homotopy between $\Phi_0:=\Phi(0\an{0},\cdot)$ and $\Phi_1:=\Phi(0\an{1},\cdot)\colon \mathfrak g\to\mathfrak h$, if and only if 
    	it can be rewritten by means of smooth curves $c\colon I \rightarrow \mathfrak{h}$ and $\psi\colon I \rightarrow \operatorname{Hom}(\mathfrak{g},\mathfrak{h})$, as 
    	\begin{equation}
    		\Phi(r\partial_t{_{|_s}}, x):= r c(s)+ \psi_s (x)
    		\label{eq:1}
    	\end{equation}
    	with $r\in\R$, $s\in I$  and $x\in\mathfrak{g}$,  
    	satisfying the partial differential equation 
    	\begin{equation}(\partial_t \psi)_s(x)=-\operatorname{ad}_{c(s)}(\psi_s(x)):=-[c(s),\psi_s(x)]
	\label{eq:7865}
    	\end{equation}
    	\label{pr:lah}
	for each $s\in I$ and each $x \in \mathfrak{g}$.
    \end{proposition}
\begin{remark}\label{rem_LA_hom_LAlgebras}
   \begin{enumerate}
   \item The curve $\psi\colon I \rightarrow \operatorname{Hom}(\mathfrak{g},\mathfrak{h})$ is automatically a smooth curve of morphisms of Lie algebras. In order to see this, use the differential equation \eqref{eq:7865} and the Jacobi identity and compute
    	\begin{equation}
    		\label{eq:curve of Lie algebra Hom}
    		\begin{split}
    			&\partial_t\left(\left[\psi_t(x),\psi_t(y)\right]-\psi_t[x,y]\right)
    			=\left[c(t),\psi_t[x,y]-\left[\psi_t(x),\psi_t(y)\right]\right]
    		\end{split}
    	\end{equation}
        for $x,y\in\mathfrak g$. It is a linear nonautonomous ordinary differential equation with initial condition 
    	\[\left[\psi_0(x),\psi_0(y)\right]-\psi_0[x,y]=0 
    	\]
    	since $\psi_0$ is a morphism of Lie algebras.
    	The constant curve $I\to \mathfrak h$, $t\mapsto 0$ is a solution of \eqref{eq:curve of Lie algebra Hom} for each $x,y\in \mathfrak g$, and so it is its unique solution.
    
    \item	By definition, $\Phi_j=\psi_j$ for $j=0,1$.
    \end{enumerate}
    \end{remark}

    \begin{proof}
    	$\Phi$ is a Lie algebroid morphism if and only if the vector bundle morphism
    	\[ \Phi^{!!}\colon TI \times \mathfrak{g} \rightarrow TI \times \mathfrak{h}=\varphi^{!!}\lie h \]
    	\[(r{{\partial_t}_{|}}_s, x)  \mapsto (r {{\partial_t}_{|}}_s, r \, c(s)+ \psi_s(x)) \] 
	for all $r\in\mathbb R$, $s\in I$ and all $x\in \lie g$, is a Lie algebroid morphism (see e.g. \cite{Mackenzie05}).
	
	A pair of a smooth function $f\in C^\infty(I)$ and $x\in C^\infty(I,\lie g)$ defines a smooth section
	$(f\partial_t,x)\colon s\to (f(s)\partial_t\an{s},x(s))$ of $TI\times \lie g\to I\times\{*\}$. This smooth section is sent by $\Phi^{!!}$ to 
	the smooth section $s\mapsto (f(s)\partial_t\an{s}, f(s)c(s)+\psi_s(x(s)))$ of $TI\times \lie h\to I\times\{*\}$.
	
    	In the following calculations, use the usual simplified notation $f':= \partial_t f$, $x':=\partial_tx$ and  $\psi':=\partial_t\psi$. Then
	for $f,g\in C^\infty(I)$ and $x,y\in C^\infty(I,\lie g)$, 
    	\begin{equation*}\begin{split} \Phi^{!!}[(f \partial_t,x),(g \partial_t,y)]&=\Phi^{!!}((fg' -gf') \partial_t, [x,y] +f y'-gx')\\
    	   			&=((fg' -gf') \partial_t,(fg'-gf')\,c +\psi([x,y]+fy'-gx')) \label{eq:Q} \end{split} 
    	\end{equation*}    	
    	Moreover, 
    	\begin{equation*} \begin{split}
    			&[\Phi^{!!}(f \partial_t,x), \Phi^{!!}(g \partial_t,y)]= [(f \partial_t,fc+\psi(x)),( g \partial_t,gc+\psi(y))]\\
    			&=((f g'-g f')\partial_t,f \, \partial_t(gc+\psi(y))-g \, \partial_t(fc+\psi(x))+[fc+\psi(x),gc+\psi(y)].
    	\end{split} \end{equation*}
    	
    	By imposing equality between the last two expressions for all smooth functions $f,g\in C^\infty(I)$, the following conditions are found:
    	\[\psi([x,y])=[\psi(x),\psi(y)] \]
	by setting $f=0=g$ in the equations above, and 
    	\[\cancel{\psi(y')} = [c,\psi(y)] + \partial_t(\psi(y))= \psi'(y) +\cancel{\psi(y')}+[c,\psi(y)]  \]
	by then setting $f=1$ and $g=0$ in the equations above.
    So $\psi$ is a curve of LA-morphisms satisfying the  claimed partial differential equation with the curve $c\colon I\to \mathfrak h$.
    \end{proof}
    
    \begin{corollary}
    	\label{co:2}
    	In the situation of Proposition \ref{homotopy_lie_algebras}, assume that $\mathfrak{h}$ is abelian. Then $\partial_t \psi \equiv 0$. It implies that 
    	two Lie algebra morphisms $\Phi_0,\Phi_1\colon \mathfrak{g} \to \mathfrak{h}$ are LA-homotopic if and only if they are equal. 
    	    \end{corollary}
    \begin{proof}
    Since $\partial_t \psi \equiv 0$, the curve $\psi\colon I\to \operatorname{Hom}_{\rm LA}(\mathfrak{g},\mathfrak{h})$ is constant. Hence $\Phi$ is given by
    \[\Phi(r\partial_t{_{|_s}}, x):= r c(s)+ \psi (x)\]
    with some fixed $\psi\in  \operatorname{Hom}_{\rm LA}(\mathfrak{g},\mathfrak{h})$. Since $\Phi_j=\Phi\circ \mathcal{I}_j$ for $j=0,1$, this yields 
    $\Phi_j(x)=\Phi(0\cdot \partial_t\an{j},x)=\psi(x)$ for $j=0,1$ and all $x\in\lie g$.
    \end{proof}
 
    \subsection{LA-homotopy for Lie algebras versus trivial deformations}
    \label{subs:3.2}
    
%   
%   
%     This section shows that LA-homotopies in the category of Lie algebras are obtained by post-composing the morphism at time zero with time-dependent (inner) automorphisms of the codomain Lie algebra.  In particular, two Lie algebras are LA-homotopy equivalent if and only if they are isomorphic. 
%     
     
%     On the other hand, this result suggests a compatibility with the deformation theory of morphisms of Lie algebras, which is investigated in Section \ref{subs:3.2}.
 Let $\mathfrak g$ and $\mathfrak h$ be two finite-dimensional Lie algebras
 and let $\psi_0\colon \mathfrak g\to \mathfrak h$ be a morphism of Lie algebras. Let $\psi\colon I \to \mathfrak g^*\otimes \mathfrak h$ be a smooth curve defined on an open interval $I$ containing $0$, such that $\phi(t)\colon \mathfrak g\to \mathfrak h$ is a morphism of Lie algebras for all $t\in I$, and such that $\psi(0)=\psi_0$. Then $\psi$ is a \textbf{deformation} of $\psi_0$, see \cite{CrScSt14} and references therein. Let $H$ be the simply connected Lie group integrating $\mathfrak h$. Then $\psi$ is a \textbf{trivial deformation} of $\psi_0$ if there is a smooth curve $h\colon I\to H$ such that 
    \[ \psi(t)=\operatorname{Ad}_{h(t)}\circ \psi_0
    \]
    for all $t\in I$. This section proves the following result,  which confirms that the notion of LA-homotopy is extremely rigid in the case of Lie algebras.
    \begin{proposition}\label{LA_Hom_def}
    Let $\mathfrak g$ and $\mathfrak h$ be two finite-dimensional Lie algebras and let 
  $\Phi\colon TI\times\mathfrak g\to \mathfrak h$ be an LA-homotopy. Then with the notation of Proposition \ref{homotopy_lie_algebras}, 
  the curve $\psi\colon I\to \mathfrak g^*\otimes \mathfrak h$ is a trivial deformation of $\Phi_0=\psi(0) \colon \mathfrak g\to \mathfrak h$.
     \end{proposition}
     Recall that by Remark \ref{rem_LA_hom_LAlgebras}, the fact that the curve $\psi\colon I \rightarrow \operatorname{Hom}(\mathfrak{g},\mathfrak{h})$ of Proposition \ref{homotopy_lie_algebras} has image in $ \operatorname{Hom}_{LA}(\mathfrak{g},\mathfrak{h})$ actually follows from \eqref{eq:7865}.

\begin{proof}
    This proof studies the  non-autonomous ordinary differential equation \eqref{eq:7865}
     with initial  condition $\psi(0)=\psi_0\in \operatorname{Hom}_{LA}(\mathfrak{g},\mathfrak{h})$. Let $H$ be the simply-connected Lie group integrating the finite-dimensional Lie algebra $\mathfrak h$. $I$ denotes as usual an open interval containing $[0,1]$.
    Consider the time-dependent vector field $X\colon I\times H\to TH$
    \[ X(t,h)=T_er_h(-c(t))
    \]
    for all $h\in H$ and all $t\in I$. The flow line
     $\Phi_{(0,e)}\colon I_{(0,e)} \to H$ starting at $e\in H$ at time $0$ satisfies then 
    \[ \dot \Phi_{(0,e)}(t)=X\left(t, \Phi_{(0,e)}(t)\right)=T_er_{ \Phi_{(0,e)}(t)}(-c(t)) \quad \Leftrightarrow \quad -c(t)=\dot\Phi_{(0,e)}(t)\cdot (\Phi_{(0,e)}(t))^{-1}
    \]
    for all $t\in I$, where $I_{(0,e)}\subseteq I$ is an interval containing $0$.
    According to \cite[Proposition 1.13.4]{DuKo00} there exists a $C^1$-path $h\colon [0,1]\to H$ starting in $e_H$ at time $0$ such that 
    \[ -c(t)=(T_er_{h(t)})^{-1}\dot h(t)=T_{h(t)}r_{h(t)^{-1}}\dot h(t)
    \]
    for all $t\in [0,1]$. But then this path is smooth since $\dot h(t)=X(t,h(t))$
    for all $t\in [0,1]$, and the flow line $\Phi_{(0,e)}$ is defined on an open interval $I_{(0,e)}$ containing $[0,1]$. For simplicity denote by $h\colon I_{(0,e)}\mapsto H$ the flow line $\Phi_{(0,e)}$. Since only the solution of \eqref{eq:7865} on $[0,1]$ is relevant, assume without loss of generality that $I_{(0,e)}=I$.

    \medskip
    
    Consider the curve 
    \[ \psi\colon I\to \operatorname{Hom}(\mathfrak g, \mathfrak h), \qquad  t\mapsto \operatorname{Ad}_{h(t)}\circ \psi_0
    \]
    i.e.~\[\psi(t)(x)=\operatorname{Ad}(h(t),\psi_0(x))
    \]
    for all $x\in \mathfrak g$ and all $t\in I$.
    Then for all $x\in\mathfrak g$
    \begin{equation*}
    	\begin{split}
    		\dot\psi(t)(x)&=T_{(h(t), \psi_0(x))}\operatorname{Ad}\left(\dot h(t), 0_{\psi_0(x)}\right)\\
		%=T_{(h(t), \psi_0(x))}\operatorname{Ad}\left(T_er_{h(t)}(-c(t)), 0_{\psi_0(x)}\right)\\
    		%&=T_{(h(t), \psi_0(x))}\operatorname{Ad}\left(T_er_{h(t)}(-c(t)), T_{\psi_0(x)}\id_{\mathfrak h}(0_{\psi_0(x)})\right)\\
    		&=T_{(e_H, \psi_0(x))}(\operatorname{Ad}\circ(r_{h(t)}\times\id_{\mathfrak h}))\left(-c(t), 0_{\psi_0(x)}\right)
    		\\
    		&=T_{(e_H, \psi_0(x))}(\operatorname{Ad}\circ(\id_H\times\operatorname{Ad}_{h(t)}))\left(-c(t), 0_{\psi_0(x)}\right)
    		\\
    		%&=T_{(e_H, \operatorname{Ad}_{h(t)}\psi_0(x))}\operatorname{Ad}\left(-c(t), 0_{\operatorname{Ad}_{h(t)}\psi_0(x)}\right)\\
    		&=T_{(e_H, \psi(t)(x))}\operatorname{Ad}\left(-c(t), 0_{\psi(t)(x)}\right)	 =\operatorname{ad}_{-c(t)}(\psi(t)(x)) %=\operatorname{ad}_{-c(t)}\psi(t)(x)
    	\end{split}
    \end{equation*}
    for all $t\in I$ shows that $\psi$ is a solution of \eqref{eq:7865} starting in $\psi_0$ at time $0$.
    \end{proof}

        \section{A generalized Poincaré lemma for Lie algebroid cohomology}
   
   This section  studies the LA-cohomology of pullback Lie algebroids under smooth homotopies, and then proves a Poincar\'e Lemma for Lie algebroid cohomology in the transitive case.
         \subsection{LA-cohomology of pullback Lie algebroids}
     \label{subs:5.1}
      Corollary \ref{cor:00} and Lemma \ref{cor:000} have a direct application in the computation of the twisted Lie algebroid cohomology of some Lie algebroids.
      	
      	\begin{theorem}
      		\label{cor:0}
      		Let $f\colon M \to N$ and $g\colon N \to M$ define a smooth homotopy equivalence between the manifolds $M$ and $N$, with homotopies $\eta\colon I \times M \to M$ and $\varphi\colon I \times N \to N$.  Let $B\to N$ be a Lie algebroid. Assume that  $f^{!!}B$ exists and the pairs $B$ and $\varphi$, and $f^{!!}B$ and $\eta$ satisfy the hypotheses of Proposition \ref{pr:mein}.
      		
      		         		Then for each flat $B$-connection $\nabla\colon \Gamma(B) \times \Gamma(E) \to \Gamma(E)$ on a vector bundle $E \to N$, the morphism
      		\[ \operatorname{H}_{\nabla}^{\bullet}(B,E) \xrightarrow{\overline{p^*_{B,f}}} \operatorname{H}^{\bullet}_{p_{B,f}^*\nabla}(f^{!!}B,f^!E) \]
      		is an isomorphism of modules over the ring isomorphism $\overline{p_{B,f}^*}\colon \operatorname{H}^{\bullet}(B) \to \operatorname{H}^{\bullet}(f^{!!}B)$.
      	\end{theorem}
	
	\begin{remark}
	Note that Theorem \ref{cor:0} holds in particular in the following two situations.
	\begin{enumerate}
	 \item $B\to N$ is a regular Lie algebroid such that $f^{!!}B$,		    $\varphi^{!!}B$ and $\eta^{!!}f^{!!}B$ exist, and such that $T\eta(TI \times 0) \subseteq \operatorname{Im}(\rho_{f^{!!}B})$ and $T\varphi(TI \times 0) \subseteq \operatorname{Im}(\rho_{B})$, i.e.~ the homotopies deform along the leaves of the target manifold.
      		    \item $B\to N$ is transitive.
      		 \end{enumerate}
The transitivity of $B\to N$ implies the condition of the theorem,  because the hypotheses of Proposition \ref{pr:mein} are satisfied for a transitive Lie algebroid and the pullback of a transitive Lie algebroid exists and is transitive. Similarly, (1) implies the condition of the theorem by Remark \ref{re:1}. 
	\end{remark}

      \begin{proof}[Proof of Theorem \ref{cor:0}]
      	Consider the canonical LA-morphisms
      	% https://q.uiver.app/#q=WzAsOCxbMCwwLCJUTVxcdGltZXNfe1ROfUI9OmZeeyEhfUIiXSxbMiwwLCJCIl0sWzAsMSwiTSJdLFsyLDEsIk4iXSxbMywwLCJmXnshIX1CIl0sWzMsMSwiTSJdLFs1LDEsIk4iXSxbNSwwLCJUTlxcdGltZXMgX3tUTX1mXnshIX1CXFxzaW1lcSBCIl0sWzAsMSwicF97QixmfSJdLFsyLDMsImYiLDJdLFswLDJdLFsxLDNdLFs0LDVdLFs2LDUsImciLDJdLFs3LDZdLFs3LDQsInBfe2ZeeyEhfUIsZ30iLDJdXQ==
      	\begin{equation*}
      		\begin{xy}
      			\xymatrix{
      			f^{!!}B  \ar[r]^{\quad p_{B,f} \quad } \ar[d]& B \ar[d]
      			& f^{!!}B \ar[d] &&\ar[ll]_{\quad p_{f^{!!}B,g} \quad}  g^{!!}(f^{!!}B) \ar[d]\\
      			M \ar[r]^{f} & N & M && \ar[ll]_ {g} N }
      		\end{xy}
      	\end{equation*}
      	Let the smooth homotopy $\phi\colon  I \times N \to N$  be such that $\phi_0=\text{id}_N$ and $\phi_1= f \circ g$.
      	By Lemma \ref{cor:000} there exists an LA-morphism
      	$\tilde{\Phi}\colon TI \times B \to B$
      	\[\tilde{\Phi} \colon TI \times B \xrightarrow{\tilde{\Pi}} \phi^{!!}B \xrightarrow{p_{B,\phi}} B,\]
      	which is an LA-homotopy between
      	\[\tilde{\Phi} \circ \mathcal{I}_0= p_{B,\phi} \circ \tilde{\Pi} \circ \mathcal{I}_0= \Phi \circ \mathcal{I}_0 \circ (\Phi \circ \mathcal{I}_0)^{-1}=\operatorname{Id}_B\] and 
      	\[\tilde{\Phi} \circ \mathcal{I}_1=p_{B,\phi} \circ \tilde{\Pi} \circ \mathcal{I}_1= p_{B,\phi}\an{\phi^{!!}_1B} \circ \tilde{\Pi} \circ \mathcal{I}_1=  p_{B,f} \circ p_{f^{!!}B,g} \circ \tilde{\Pi} \circ \mathcal{I}_1.\] 
	The above uses that the image of $\tilde{\Pi} \circ \mathcal{I}_1$ is contained in  $\phi^{!!}_1B \subseteq \phi^{!!}B$, and that 
	$p_{B,\phi}$ coincides with $p_{B,\phi_1}=p_{B, f\circ g}=p_{B,f} \circ p_{f^{!!}B,g}$ on this subbundle.
	Recall that by Proposition \ref{pr:mein} (and Remark \ref{remark_pr:mein}), $\varphi_1^{!!}B\simeq \varphi_0^{!!}B=\id_N^{!!}B=B$. In addition, $\tilde \Pi\colon TI\times B\to \varphi^{!!}B$ is an LA-isomorphism over the identity on $I\times N$. As a consequence,  
	the map $\tilde{\Pi} \circ \mathcal{I}_1\colon B\to \varphi_1^{!!}B$ (with restricted codomain) is injective, and since $B$ and $\varphi_1^{!!}B$ have the same rank, it is an isomorphism. It hence defines an isomorphism 
	\[
	\overline{\left(\tilde{\Pi} \circ \mathcal{I}_1\right)^*}\colon \operatorname{H}_{p_{B, f\circ g}^*\nabla}\left((f\circ g)^{!!}B, (f\circ g)^!E\right)\rightarrow \operatorname{H}_\nabla(B,E)
	\]
	in cohomology.

Set 
      	\[F:= p_{B,f}\colon f^{!!}B \to B, \qquad F(v_m, b_{f(m)}):= b_{f(m)}\]   	
and
      	\[G:= p_{f^{!!}B,g} \circ \tilde{\Pi} \circ \mathcal{I}_1 \colon B \to f^{!!}B\]
      	The composition
      	$F \circ G$
      	is by definition $\tilde{\Phi} \circ \mathcal{I}_1$. Thus it is LA-homotopic to $\operatorname{Id}_B$ via the LA-homotopy $\tilde{\Phi}\colon TI \times B \to B$, over $\phi \colon I \times N \to N$. As a consequence, the induced map $\overline{G^*}$ in cohomology  is surjective and $\overline{F^*}$ is injective.
	Since $\overline{\left(\tilde{\Pi} \circ \mathcal{I}_1\right)^*}$ is an isomorphism, the surjectivity of $\overline{p_{f^{!!}B,g}^*}$ then follows from the surjectivity of $\overline{G^*}$.
	
	\bigskip
      	
      	With an analogous construction, given the homotopy $\eta \colon I \times M \to M$ between $\eta_0= \operatorname{Id}_M$ and $\eta_1= g \circ f$, apply Lemma \ref{cor:000} to get an isomorphism $\tilde{\Theta}\colon TI \times f^{!!}B \to \eta^{!!}f^{!!}B$, over the identity map, and an LA-homotopy
      	\[\tilde{\Psi}\colon TI \times f^{!!}B \xrightarrow{\tilde{\Theta}} \eta^{!!}f^{!!}B \xrightarrow{p_{f^{!!}{B},\eta}} f^{!!}B \]
      	such that
      	\[ \tilde{\Psi} \circ \tilde{\mathcal{I}}_0 = \operatorname{Id}_{f^{!!}B},\]
      	for $\tilde{\mathcal{I}}_0 \colon f^{!!}B \to TI \times f^{!!}B$ the inclusion LA-morphism at time $0$.
      	
      	At time $1$, the LA-morphism reads 
      	\begin{equation*}
      		\begin{split} \tilde{\Psi} \circ \tilde{\mathcal{I}}_1 &= p_{f^{!!}B, \eta} \circ \tilde{\Theta} \circ \tilde{\mathcal{I}}_1= p_{f^{!!}B, \eta}\an{\eta_1^{!!}f^{!!}B} \circ \tilde{\Theta} \circ \tilde{\mathcal{I}}_1=p_{f^{!!}B, g\circ f} \circ \tilde{\Theta} \circ \tilde{\mathcal{I}}_1\\&= p_{f^{!!}B,g} \circ p_{g^{!!}(f^{!!}B), f} \circ \tilde{\Theta} \circ \tilde{\mathcal{I}}_1.
      		\end{split}
      	\end{equation*}	
      	Define 
      	\[\tilde{F}:= p_{g^{!!}(f^{!!}B), f} \circ  \tilde{\Theta} \circ \tilde{\mathcal{I}}_1\colon f^{!!}B \to g^{!!}(f^{!!}B)\]
	and 
      	\[\tilde{G}= p_{f^{!!}B,g}\colon g^{!!}(f^{!!}B) \to f^{!!}B. \]
      	By construction, $\tilde{G} \circ \tilde{F}=\tilde\Psi\circ\tilde{\mathcal I}_1$ is LA-homotopic to $\operatorname{Id}_{f^{!!}B}$.

      	Then $\overline{\tilde{G}^*}$ is injective and $\overline{\tilde{F}^*}$ is surjective. Recall that by the considerations above, $\overline{\tilde{G}^*}:= \overline{p_{f^{!!}B,g}^*}$ is also surjective. Hence $\overline{G^*}$ is a module isomorphism as a composition of module isomorphisms.
	Since $\overline{(F \circ G)^*} \equiv \operatorname{Id}_{\operatorname{H}_\nabla(B,E)}$, then also $\overline{F^*}$ is an isomorphism, and similarly, $\overline{\tilde{F}^*}$ is an isomorphism.
      \end{proof}

       Since Atiyah Lie algebroids are transitive, Theorem  \ref{cor:0} has the following corollary.
      \begin{corollary}
      	Let $\pi\colon P \to N$ be a principal $G$-bundle, with $G$ a Lie group, and $f\colon M \to N$ a homotopy equivalence. For each flat $A(P)$-connection $\nabla\colon \Gamma(A(P)) \times \Gamma(E) \to \Gamma(E)$ with values on the vector bundle $E \to N$, the map
      	\[ \operatorname{H}^{\bullet}_{\nabla}(A(P),E)\xrightarrow{\overline{p_{A(P),f}^*}} \operatorname{H}^{\bullet}_{p_{A(P),f}^*\nabla}(A(f^*P),f^!E)\]
      	is a module isomorphism.
      \end{corollary}

      \begin{proof}
      	It is an immediate consequence of the well known fact that $A(f^*P) \cong f^{!!}A(P)$ (see for instance \cite{Mackenzie05}) and Theorem \ref{cor:0}.
      \end{proof}

  \subsection{A generalized Poincaré lemma}
  \label{subs:5.2}
      A crucial property of de Rham cohomology is its local triviality, also known as Poincar\'e Lemma. In general, Lie algebroid cohomology is not locally trivial. However, in the transitive case, the Lie algebroid cohomology can be more easily computed using the Chevalley-Eilenberg cohomology of the isotropy Lie algebra.
      
      \begin{lemma}[Generalized Poincaré-Lemma]
      	\label{Th:PL}
      	Let $B \to U$ be a transitive Lie algebroid over a contractible base manifold. Let $\nabla$ be a $B$-representation on $E \to U$. For  $x \in U$, denote as usual by $\mathfrak{g}_x:=\operatorname{ker}(\rho_B)\an{x}$ the \textbf{isotropy Lie algebra} of $B$ at $x$. Define $f_x\colon \{*\} \to U$, $f_x(*):=x$. Then
      	$B$ and $\mathfrak{g}_x$ are LA-homotopy equivalent for each $x$ in $U$. In particular,
      	\begin{equation}
      		\label{eq:pl}
      		\operatorname{H}^{\bullet}_{\nabla}(B, E) \xrightarrow{\overline{p_{B,f_x}^*}}  \operatorname{H}^{\bullet}_{\nabla \an x}(\mathfrak{g}_x,E\an x)\qquad \forall x \in U,
      	\end{equation}
      	is a module isomorphism over the ring isomorphism $\operatorname{H}^{\bullet}(B) \xrightarrow{\overline{p_{B,f_x}^*}} \operatorname{H}^{\bullet}(\mathfrak{g}_x)$.
      	 $\nabla \an x$ denotes the pullback of $\nabla$ via the map $f_x$. $\operatorname{H}^{\bullet}_{\nabla \an x}(\mathfrak{g}_x,E\an x)$ is the twisted Chevalley-Eilenberg cohomology with coefficients in $E\an x$ of the isotropy Lie algebra $\mathfrak{g}_x$. 
      \end{lemma} 
      \begin{proof}
      	Since $U$ is contractible, each map $f_x\colon \{*\} \to U$, $f(*):=x$, for $x \in U$, defines a homotopy equivalence between $U$ and the point $\{*\}$. Thus $f_x^{!!}B$ and $B$ have isomorphic cohomology modules, by Theorem \ref{cor:0},
	 for each choice of $x \in U$. Conclude by observing that 
      	\[f_x^{!!}B=\{(0,b) \in T\{*\}\times B\mid Tf_x(0)=\rho_B(b)\}\cong \operatorname{ker}(\rho_B(b))\an x=\mathfrak{g}_x.\]
      \end{proof}
      
      \begin{remark}
      	When the Lie algebroid $B$ is the tangent bundle $TU$ of a contractible manifold, Lemma \ref{Th:PL} reduces to the classical Poincar\'e Lemma of de Rham cohomology.
      \end{remark}
    
      In general it is not possible to reduce to the Chevalley-Eilenberg cohomology of a Lie algebra for non-transitive Lie algebroids. For instance, for Lie algebroids with zero anchor and trivial brackets the equality $\operatorname{H}^{0}(A)=C^{\infty}(U)$ holds. $C^{\infty}(U)$ is an infinite dimensional vector space, whether the Chevalley-Eilenberg cohomology of a Lie algebra is finite dimensional.\\

\subsection{Relation with the existing literature on Morita equivalence}
\label{subs:5.3}
In \cite{Cr03} Crainic proves the following result.
      	\begin{theorem*}
      		\label{th:cr}
      		Let $\pi\colon P \to M$ be a surjective submersion with homologically $n$-connected fibers, let $B\to M$ be a Lie algebroid. Then $\overline{p_{B,\pi}^*} \colon \operatorname{H}^{\bullet}(B) \to \operatorname{H}^{\bullet}(\pi^{!!}B)$ is an isomorphism in all degrees $k \le n$. The same holds with general coefficients.
      	\end{theorem*}
      	
      	This section compares Theorem \ref{cor:0} with this result for weakly connected fibers and for $B$ a transitive Lie algebroid. 
		The following special case of Crainic's result is a corollary of Theorem \ref{cor:0}. 
      	
      	\begin{corollary}
      		Let $\pi\colon P \to M$ be a surjective submersion with weakly connected fibers, let $B\to M$ be a transitive Lie algebroid and let $\nabla$ be a flat $E \to M$ connection.  $\overline{p_{B,\pi}^*} \colon \operatorname{H}_{\nabla}^{\bullet}(B,E) \to \operatorname{H}_{p_{B,\pi}^{*}\nabla}^{\bullet}({\pi}^{!!}B, \pi^{!}E)$
		 is a module isomorphism over the ring isomorphism $\overline{p_{B,\pi}^*} \colon \operatorname{H}^{\bullet}(B) \to \operatorname{H}^{\bullet}(\pi^{!!}B)$.
		      	\end{corollary}
	
      	\begin{proof}      		Fibers of surjective submersions are smooth manifolds, in particular CW-complexes. By Whitehead's theorem,
		  weakly connected (i.e.~with all homotopy group trivial) fibers are contractible fibers. It is proved in \cite{Mei02} (see Corollary $13$ combined with an observation after Lemma $6$) that any surjective submersion with contractible fibers is a fibration and a homotopy equivalence. As a consequence of Theorem \ref{cor:0}, conclude that $\overline{p_{B,\pi}^*} \colon \operatorname{H}_{\nabla}^{\bullet}(B,E) \to \operatorname{H}_{p_{B,\pi}^{*}\nabla}^{\bullet}({\pi}^{!!}B, \pi^{!}E)$
		   is a module isomorphism over the ring isomorphism $\overline{p_{B,\pi}^*}  \colon \operatorname{H}^{\bullet}(B) \to \operatorname{H}^{\bullet}(\pi^{!!}B)$.
      	\end{proof} 
      	
       On the other hand,
      	\emph{the generalized Poincar\'e Lemma (Theorem \ref{Th:PL}) is a corollary of Crainic's theorem}:   
	If $U$ is a contractible manifold, $c \colon U \to \{*\}$ is a surjective submersion with contractible fibers. Let $B\to U$ be a transitive Lie algebroid and $\mathfrak{g}$ its isotropy Lie algebra, then there is a surjective submersion 
      		\[p_{\mathfrak{g},c} \colon c^{!!}\mathfrak{g} \to \mathfrak{g}\]
      		with $c^{!!}\mathfrak{g} \cong TU \times \mathfrak{g}$, which is isomorphic to $B$ by Theorem \ref{localtriviality}. This implies that
      		$\overline{p_{\mathfrak{g},c}^*} \colon \operatorname{H}_{\nabla}^{\bullet}(\mathfrak{g},V) \to \operatorname{H}_{p_{\mathfrak{g},c}^{*}\nabla}^{\bullet}(c^{!!}\mathfrak{g}, c^{!}V)$ is a module isomorphism, using Theorem \ref{cor:0}.
		 All representations of $B$ are of the form $p_{\mathfrak{g},c}^*\nabla$, see \cite{Cr03}. 
		Then the Poincar\'e Lemma follows.

    \section{Mayer-Vietoris principle and a K\"unneth formula}
     Given the Poincaré lemma, the usual tool for computing the cohomology of more general objects is a Mayer-Vietoris principle. 
    
    \subsection{Mayer-Vietoris} \label{sec:MV}
    A Mayer-Vietoris principle seems to be folklore knowledge; it is used for instance in \cite{Cr03} and \cite{Fr19}, it is stated in \cite{Vaisman90} and \cite{Vai94} for Poisson manifolds. In this section, an explicit expression for the coboundary operator in the long exact sequence is deduced by a generalization of standard techniques to twisted Lie algebroid cohomology (see \cite{BoTu82}).\medskip

    Let $M$ be a smooth manifold and suppose that $M = U \cup V$, with $U,V$ open. Denote with $j_0\colon U \cap V  \hookrightarrow U$, $j_1\colon U \cap V  \hookrightarrow V$, $i_0\colon  U \hookrightarrow M$, $i_1\colon V \hookrightarrow M$ and $i \colon U \sqcup V \hookrightarrow M$ the inclusions. Then the two sequences
    of smooth maps 
    \[U \cap V \xrightarrow{j_k} U \sqcup V \xrightarrow{i} M \]
    for $k=0,1$, give the same composition $i\circ j_0=i\circ j_1\colon U\cap V\to M$.
    \medskip

    Let $A \to M$ be a Lie algebroid and $\nabla$ an $A$-representation on $E\to M$.
    Recall that $A\an{U}:=i^{!!}_0 A$ denotes the restriction Lie algebroid and $E\an{U}$ the restricted vector bundle. 
    The notation $\nabla\arrowvert_{U}\colon \Gamma(A\an{U})\times\Gamma(E\an{U})\to\Gamma(E\an{U})$ stands for the restriction to $U$ of the connection $\nabla$.
    \begin{theorem}\label{thm:mv}
    	Let $A \to M$ be a Lie algebroid and $\nabla$ an $A$-representation on $E\to M$.
    	Then the Mayer-Vietoris sequence 
    	\[0 \rightarrow \Omega^{\bullet}(A,E) \xrightarrow{i^*} \Omega^{\bullet}(A_{|_U}, E_{|_U}) \oplus \Omega^{\bullet}(A_{|_{V}}, E_{|_V}) \xrightarrow{j_1^*-j_0^*} \Omega^{\bullet}(A_{|_{U \cap V}}, E_{|_{U \cap V}}) \rightarrow 0, \]
    	is exact.
    	
    	Let $\{\mu_U,\mu_V\}$ be a partition of unity subordinated to the cover $\{U,V\}$. The Mayer-Vietoris sequence induces a long exact sequence in LA-cohomology with values in representations
    	\[\cdots \to \operatorname{H}^q_{\nabla\an{U \cap V}}(A \an{U \cap V},E \an{U \cap V}) \xrightarrow{\operatorname{d}_{\nabla}^*} \operatorname{H}_{\nabla}^{q+1}(A,E) \xrightarrow{i^*} \operatorname{H}_{\nabla\an{U}}^{q+1}(A\an{U}, E\an{U}) \oplus \operatorname{H}_{\nabla\an{V}}^{q+1}(A\an{V}, E\an{V})\] \[\xrightarrow{j_1^*-j_0^*} \operatorname{H}^{q+1}_{\nabla\an{U \cap V}}(A\an{U \cap V}, E\an{U \cap V} ) \xrightarrow{\operatorname{d}_{\nabla}^*} \operatorname{H}_{\nabla}^{q+2}(A,E) \to \cdots 
    	\]
    	The coboundary operator $\operatorname{d}_{\nabla}^*$ is a morphism of vector spaces explicitly defined as
    	\[\operatorname{d}_{\nabla}^*([\omega]):= 
    	\begin{cases} 
    		[-\operatorname{d}_{\nabla}(\mu_V\omega)] & \operatorname{on}\, U\\
    		[\operatorname{d}_{\nabla}(\mu_U\omega)] & \operatorname{on}\, V
    	\end{cases}
    	\]
	for $\omega\in \Omega^\bullet(A\an{U\cap V}, E\an{U\cap V})$.
    	\label{MV}
    \end{theorem}
    
    \begin{proof}
    	The injectivity of the first map and the condition $\operatorname{ker}(j^*_1-j^*_0)=\operatorname{im}(i^*)$ follows because Lie algebroid forms define a sheaf, see Remark \ref{Mv:prop3}. Each form $\omega \in \Omega^{\bullet}(A\an{U \cap V}, E\an{U \cap V})$ can be rewritten as
    	\[\omega=(\mu_U)_{|_{U \cap V}}\, \omega + (\mu_V)_{|_{U \cap V}}\, \omega. \]
    	Then a preimage of $\omega$ via $j_1^*-j_0^*$ is $\eta:=(-(\mu_V)_{|_{U}}\, \omega, (\mu_U)_{|_{V}}\, \omega ) \in \Omega^{\bullet}(A\an{U}, E\an{U}) \oplus \Omega^{\bullet}(A\an{V}, E\an{V})$. The Mayer-Vietoris sequence therefore is exact.     	The existence of the induced long exact sequence is a consequence of standard results of Homological Algebra (see e.g.~\cite{BoTu82}). The idea of the proof is sketched to deduce the explicit expression of the coboundary operator.  Let $\omega\in \Omega^{\bullet}(A\an{U \cap V}, E\an{U \cap V})$ be a closed form. Since $(j_1^*-j_0^*)$ intertwines the differentials, it is true that $(j_1^*-j_0^*)\circ \operatorname{d}_{\nabla}(\eta)= \operatorname{d}_{\nabla \an{U \cap V}}(\omega)=0$, 
	for a closed form $\omega\in \Omega^{\bullet}(A\an{U \cap V}, E\an{U \cap V})$ and $\eta:=(-(\mu_V)_{|_{U}}\, \omega, (\mu_U)_{|_{V}}\, \omega ) \in \Omega^{\bullet}(A\an{U}, E\an{U}) \oplus \Omega^{\bullet}(A\an{V}, E\an{V})$ as above. In other words, $-\operatorname{d}_{\nabla }((\mu_V)\an{U \cap V}\omega)$ and $\operatorname{d}_{\nabla}((\mu_U)\an{U \cap V} \omega)$ agree on $U \cap V$. By exactness of the Mayer-Vietoris sequence, there exists $\alpha \in \Omega^{q+1}(A,E)$ such that $i^*(\alpha)=\operatorname{d}_{\nabla}\eta$. $\alpha$ is a closed form because 
    	$i^* \circ \operatorname{d}_{\nabla}\alpha=\operatorname{d}_{\nabla} \circ i^*(\alpha)= \operatorname{d}_{\nabla} \circ \operatorname{d}_{\nabla} \eta=0$ and $i^*$ is injective. Define $\operatorname{d}_{\nabla}^*([\omega]):= [\alpha]$. It is a standard exercise to prove that $\operatorname{d}^*_{\nabla}$ is well defined. By definition of $i^*$, the coboundary operator is explicitly defined as in the theorem.
    \end{proof}

    \begin{corollary}
    	\label{findim1}
    	Let $A \to M$ be a Lie algebroid over $M$ admitting a finite good open cover (i.e.~all open subsets and their intersections are contractible) $(U_j)_{j\in J}$ such that
    	$\operatorname{H}^{k}_{\nabla\an{U_j}}(A\an{U_j}, E\an{U_j})$ is finite-dimensional for each $j\in J$. Then 
    	$\operatorname{H}_{\nabla}^{k}(A,E)$ is finite-dimensional.
    \end{corollary}
    \begin{proof}
    	See the proof of Proposition 5.3.1 in \cite{BoTu82}.
    \end{proof}
   
    More concretely, this yields an alternative basic proof for the following well known result. 
    \begin{corollary}
    	\label{findim2}
    	Let $A \to M$ be a transitive Lie algebroid over a compact $M$. Then each twisted Lie algebroid cohomology of $A$ is finite-dimensional.
    \end{corollary}
 
    \begin{proof}
    	The compact manifold $M$ admits a finite good cover $(U_j)_{j\in J}$. Moreover the generalized Poincar\'e Lemma \ref{Th:PL} implies that 	$\operatorname{H}^{k}_{\nabla\an{U_j}}(A\an{U_j}, E\an{U_j})$ is isomorphic to the Chevalley-Eilenberg cohomology of the isotropy Lie algebra with values in a representation, which is finite dimensional by definition (the vector spaces of the chain complex are finite dimensional). Conclude using Corollary \ref{findim1}.
    \end{proof}
    \begin{remark}
    	Corollary \ref{findim2} is usually proven by observing that the Chevalley-Eilenberg complex of transitive Lie algebroids is elliptic (see \cite{Kr10}). The cohomology therefore is finite dimensional if the base manifold is compact (see \cite{AtBo67}).
This arguments is used for instance in \cite{Wal23}.
    \end{remark}

    \subsection{A K\"unneth formula.}
    \label{sec:K}
    A K\"unneth formula is proved as an application of the generalised Poincaré lemma \ref{Th:PL} and the Mayer-Vietoris principle \ref{thm:mv}.\\
    
    Let $A \to M$ and $B\to N$ be Lie algebroids. Set the notation $\operatorname{Pr}_A\colon A \times B \to A$, $(a,b)\mapsto a$,  and accordingly $\operatorname{Pr}_B\colon A \times B \to B$ , $\operatorname{pr}_M\colon M \times N \to M$ and $\operatorname{pr}_N\colon M \times N \to N$ for the projection morphisms in the respective categories. 
    Let $\nabla^A$ be an $A$-representation on  the vector bundle $E\to M$ and let $\nabla^B$ be a $B$-representation on $F\to N$. The vector bundle $\operatorname{pr}_M^!E \otimes \operatorname{pr}_N^!F$ is the tensor product of the pullback vector bundles on $M\times N$. It is naturally equipped with the connection $\nabla$ defined by the Leibniz rule and 
    \begin{equation} 
    	\begin{aligned}
    			 \nabla_{(a,b)}(\operatorname{pr}_M^!e \otimes \operatorname{pr}_N^!f)&=(\operatorname{Pr}^*_A\nabla^A)_{(a,b)}(\operatorname{pr}_M^!e) \otimes \operatorname{pr}_N^!f + \operatorname{pr_M}^!e \otimes (\operatorname{Pr}^*_B\nabla^B)_{(a,b)}(\operatorname{pr}_N^!f)\\&=\operatorname{pr}_M^!(\nabla^A_{a}e) \otimes \operatorname{pr}_N^!f+ \operatorname{pr}^!_Me \otimes \operatorname{pr}_N^!(\nabla^B_{b}f)
    	\end{aligned}  
    	\label{eq:K1} 	
    \end{equation} 
    for $a\in\Gamma(A)$, $b\in\Gamma(B)$, $(a,b):= \operatorname{pr}^!_M a + \operatorname{pr}^!_N b \in \Gamma(A \times B)$, $e \in \Gamma(E)$ and $f \in \Gamma(F)$.
    
     \begin{lemma}
    	In the situation above, the connection $\nabla$ is well defined and flat.
    \end{lemma}
    \begin{proof}
    	$\nabla$ is linear in its first entry by definition. Moreover \eqref{eq:K1} 
    	is compatible with the Leibniz rule on pullbacks of functions since 
    	\begin{equation}
    		\begin{split} &\nabla_{(a,b)}(\operatorname{pr}_M^!(g \cdot e) \otimes \operatorname{pr}_N^!f)\\
		&= \operatorname{pr}^*_Mg\cdot \operatorname{pr}^!_M(\nabla_{a} e) \otimes \operatorname{pr}_N^!f + \operatorname{pr}_M^!(\rho_A(a)(g) \cdot e) \otimes \operatorname{pr}^!_Nf + \operatorname{pr}^*_Mg\cdot \operatorname{pr}_M^!e \otimes \operatorname{pr}_N^!(\nabla_{b}f)\\
			&= \operatorname{pr}^*_Mg\cdot \nabla_{(a,b)}(e \otimes f) + \rho(a,b)(\operatorname{pr}^*_Mg)\cdot e \otimes \operatorname{pr}_N^!f,
    		\end{split}
    		\label{eq:K2}
    	\end{equation}
    	for $g\in C^\infty(M)$, $(a,b) \in \Gamma(A \times B)$, $e \in \Gamma(E)$ and $f \in \Gamma(F)$. In the same manner, 
	\[\nabla_{(a,b)}(\operatorname{pr}_M^! e \otimes \operatorname{pr}_N^!(g \cdot f))
	\]
	for $g\in C^\infty(N)$, $(a,b) \in \Gamma(A \times B)$, $e \in \Gamma(E)$ and $f \in \Gamma(F)$. Hence the Leibniz rule can be used in order to extend \eqref{eq:K1} to a linear $A\times B$ connection on $\pr_M^!E\otimes \pr_N^!F$.

		It is enough to compute $R_{\nabla}$ on generators as above. For $a_1,a_2\in\Gamma(A)$, $b_1,b_2\in\Gamma(B)$, $e\in\Gamma(E)$ and $f\in\Gamma(F)$ a simple computation yields
    	\begin{equation*}
    		\begin{split}
    			&R_{\nabla}((a_1, b_1),(a_2, b_2))(\pr_M^!e \otimes \pr_N^!f)\\
			&= \operatorname{pr}^!_M(R_{\nabla^A}(a_1,a_2)e) \otimes \operatorname{pr}^!_Nf + \operatorname{pr}^!_M e \otimes \operatorname{pr}_N^!(R_{\nabla^B}(b_1, b_2))=0. \qedhere
    		\end{split}
    	\end{equation*}
    \end{proof}

    \begin{theorem}[K\"unneth isomorphism]
    	Let $A \to M$ be a Lie algebroid and let $B \to N$ be a transitive Lie algebroid over a base manifold $N$ admitting a finite good cover. Let $\nabla^A$ and $\nabla^B$ be flat connections on the vector bundles $E \to M$ and $F\to N$, respectively. Then
    	\begin{equation}
    		\operatorname{H}_{\nabla^A}^{\bullet}(A, E) \otimes \operatorname{H}_{\nabla^B}^{\bullet}(B,F)          \rightarrow \operatorname{H}_{\nabla}^{\bullet}(A \times B,\operatorname{pr}_M^!E \otimes \operatorname{pr}_N^!F) \label{eq:K0} \end{equation}
    	\[[\omega] \otimes [\eta] \mapsto [\operatorname{Pr}_A^*\omega \wedge \operatorname{Pr}_B^*\eta] \]
    		is an isomorphism of modules over the ring isomorphism 
    	\[ \operatorname{H}^{\bullet}(A) \otimes \operatorname{H}^{\bullet}(B) \to \operatorname{H}^{\bullet}(A\times B)\]
    	\[[\omega] \otimes [\eta] \mapsto [\operatorname{Pr}_A^*\omega \wedge \operatorname{Pr}_B^*\eta]. \]
    	In particular, 
    	\[\operatorname{H}^n(A \times B, \operatorname{pr}_M^!E \otimes \operatorname{pr}_N^!F) \cong \bigoplus_{p+q=n} \operatorname{H}^p(A, E) \otimes \operatorname{H}^q(B, F),\] for each $n \ge 0$.
    	\label{th:k1}
    \end{theorem}
 
   Kubarski proves in \cite{Kub02} this theorem for the Lie algebroid cohomology of the product $A \times \mathfrak{g}$, where $A$ is a transitive Lie algebroid and $\mathfrak{g}$ a Lie algebra.
   Theorem \ref{th:k1} generalizes his result to more general Lie algebroids and the twisted Lie algebroid cohomology.
In this first part of this proof, similar techniques to Kubarski's in \cite{Kub02} are used.
Theorem \ref{th:k1}    generalizes as well a more recent result by Waldron in \cite{Wal23}. The proof of Theorem \ref{th:k1} relies on the special case where $B$ is a Lie algebra. This special case is proved first by direct computations.
    \begin{lemma}
    	\label{lemma:3}
    	Let $A\to M$ be a Lie algebroid and $\mathfrak{g}$ a Lie algebra. Let $(E,\nabla^A)$ and $(F, \nabla^{\mathfrak g})$ be representations. 
	Denote with $\operatorname{Pr}_A\colon A \times \mathfrak{g} \to A$ and $\operatorname{Pr}_{\mathfrak{g}}: A \times \mathfrak{g} \to \mathfrak{g}$ the projection maps and with  $\hat{\otimes}_{\R}$ the tensor product of differential graded vector spaces. 
	
	The K\"unneth homomorphism
    	\begin{equation}
    		\mathcal{K} \colon (\Omega^{\bullet}(A, E),\operatorname{d}_{\nabla^A}) \hat{\otimes}_{\R} (\Omega^{\bullet}(\mathfrak{g}, F), \operatorname{d}_{\nabla^{\mathfrak{g}}}) \longrightarrow (\Omega^{\bullet}(A \times \mathfrak{g}, \operatorname{pr}_M^!E \otimes_{\R} \operatorname{pr}_*^!F), \operatorname{d}_{\nabla})
    		\label{eq:K3}
    	\end{equation}
    	\begin{equation*}
    		\omega \otimes \phi \longmapsto \operatorname{Pr}_A^*\omega \wedge \operatorname{Pr}_{\mathfrak{g}}^*\phi
    	\end{equation*}
    	with \[(\operatorname{Pr}_A^*\omega \wedge \operatorname{Pr}_{\mathfrak{g}}^*\phi)(a_1, \dots, a_r,v_1, \dots, v_s):= \operatorname{pr}_M^*(\omega(a_1, \dots, a_r)) \otimes \operatorname{pr}_{*}^*(\phi(v_1, \dots, v_s)) \]
    	is an isomorphism of graded differential vector spaces.  	  
    \end{lemma}
    \begin{proof}
    	For each $x \in M$ consider the map $\mathcal{K}_x$ defined by the following commutative diagram
    	\begin{equation*}
    		\small{
    			\begin{xy}
    				\xymatrix{(\bigwedge^{\bullet} A^*_x \otimes_{\R} E_x) \hat{\otimes}_{\R} (\bigwedge^{\bullet} \mathfrak{g}^* \otimes F) \ar[r] \ar[rd]^{\mathcal{K}_x} & (\bigwedge^{\bullet} A^*_x \hat{\otimes}_{\R} \bigwedge^{\bullet} \mathfrak{g}^*) \otimes_{\R} ((\operatorname{pr}_M^!E)_x \otimes_{\R} (\operatorname{pr}_{\{*\}}^!F)_x)  \ar[d]\\
    					& (\bigwedge^{\bullet} (A\times \mathfrak{g})^*_x ) \otimes_{\R} ((\operatorname{pr}_M^!E)_x \otimes_{\R} (\operatorname{pr}_{\{*\}}^!F)_x), \\
    				} 
    		\end{xy}}
    	\end{equation*}
    	i.e. $\mathcal K_x$ is explicitly defined by
    	\begin{equation*}
    		\small{
    			\begin{xy}
    				\xymatrix{(\omega_x \otimes e_x) \otimes (\eta \otimes f) \ar@{|->}[r] \ar@{|->}[dr]^{\mathcal{K}_x} & (\omega_x \otimes \eta) \otimes ((\operatorname{pr}_M^!e)_x \otimes (\operatorname{pr}^!_{\{*\}}f)_x)  \ar@{|->}[d]\\
    					& (\operatorname{Pr}^*_A)_x\omega \wedge (\operatorname{Pr}^*_{\mathfrak{g}})_x\eta \otimes ((\operatorname{pr}_M^!e)_x \otimes (\operatorname{pr}^!_{\{*\}}f)_x). \\
    				} 
    		\end{xy}}
    	\end{equation*}
 
    	The section $e \in  \Gamma(E)$ is an extension of $e_x$ and the form $\omega$ is an extension of $\omega_{x}$. The first map of the composition is an isomorphism because of the isomorphisms $(\operatorname{pr}^!_M E)_x \cong E_x$ and $(\operatorname{pr}^!_* F)_x \cong F$, defined via the canonical projection induced by the pullback restricted to a fiber. The second map is an isomorphism by standard arguments of linear algebra: it is true that $ (\bigwedge^{\bullet} A^*_x\, \hat{\otimes}_{\R}\, \bigwedge^{\bullet} \mathfrak{g}^*) \to \bigwedge^{\bullet} (A \times \mathfrak{g})^*_{|_x}$, $(\omega_x \otimes \eta) \mapsto (\operatorname{Pr}^*_A)_x\omega \wedge (\operatorname{Pr}^*_{\mathfrak{g}})_x\eta$ is an isomorphism of algebras (see for instance \cite{Greub67}, Section $5.15$),
	then apply the tensor product over $\R$ with the vector space $(\operatorname{pr}_M^!E)_x \otimes (\operatorname{pr}^!_{\{*\}}F)_x$, which is an exact functor.
	
By definition 
    	\[\mathcal{K}((\omega \otimes e) \otimes (\eta \otimes f))_x =\mathcal{K}_x((\omega_x \otimes e) \otimes (\eta_x \otimes f))\]
	for all $\omega\in\Omega^\bullet(A)$, $e\in\Gamma(E)$, $\eta\in\mathfrak g$, $f\in F$ and $x\in M$.
    Since $\mathcal{K}_x$ is an isomorphism for each $x\in M$, $\mathcal K$ is injective.

    	To show that $\mathcal{K}$ is an epimorphism, consider a section  $\Phi \in \Omega^r(A \times \mathfrak{g}, \operatorname{pr}^!_ME \otimes \operatorname{pr}^!_{\{*\}}F)$. Take  a basis $v_1, \dots, v_n$ of $\mathfrak{g}$
    	and a basis $f_1, \dots, f_k$ of $F$. Choose an open subset $U\subseteq M$ trivialising both $A$ and $E$, and smooth frames $(\alpha_1,\ldots, \alpha_r)$ for $A^*$ over $U$ and $(e_1,\ldots, e_l)$ for $E$ over $U$.
	Then $\Phi\an{U\times\{*\}}$ can be written
	\[ \sum_{\substack{1\leq i_1< \ldots<i_p\leq r\\
	1\leq j_1<\ldots<j_q\leq n\\
	p+q=r}
	}\sum_{s=1}^l\sum_{t=1}^{k}f_{\substack{i_1,\ldots, i_p\\ j_1,\ldots, j_q\\ s,t}}\Ppr_A^*(\alpha_{i_1}\wedge\ldots \wedge \alpha_{i_p})\wedge\Ppr_*^*(v_{j_1}\wedge\ldots\wedge v_{j_q})\otimes \pr_M^!e_s\otimes \pr_{*}^!f_t
	\]
	with smooth functions \[f_{\substack{i_1,\ldots, i_p\\ j_1,\ldots, j_q\\ s,t}}\in C^\infty(U\times\{*\})=C^\infty(U).\]
	It is then easy to show, using a partition of unity on $M$ subordinate to an open cover of $M$ by sets trivialising simultaneously $A$ and $E$, that $\Phi$ lies in the image of $\mathcal K$.
	In addition, $\mathcal{K}$ preserves the total degree and is a homomorphism of algebras by definition.

	\medskip
	It is left to show that $\mathcal{K}$ intertwines the differentials: take $\nabla$ to be the tensor flat connection in (\ref{eq:K1}), then the induced differential $\operatorname{d}_{\nabla}$ satisfies
    	\begin{equation}
    		\begin{split} 
    			\operatorname{d}_{\nabla} (\mathcal{K} (\omega \otimes \phi))&= \operatorname{d}_{\nabla}
    			(\operatorname{Pr}^*_A \omega \wedge \operatorname{Pr}_{\mathfrak{g}}^*\phi)\\
			&= \operatorname{Pr}_A^*(\operatorname{d}_{\nabla^A} \omega) \wedge \operatorname{Pr}^*_{\mathfrak{g}}\phi +(-1)^{\operatorname{deg}\omega} \operatorname{Pr}_A^*(\omega) \wedge \operatorname{Pr}^*_{\mathfrak{g}}(\operatorname{d}_{\nabla^{\mathfrak g}}\phi)\\
&    			= \mathcal{K}(\operatorname{d}_{\nabla^A} \omega\, \otimes\, \phi + (-1)^{\operatorname{deg}(\omega)} \omega \,\otimes\, \operatorname{d}_{\nabla^{\mathfrak g}} \phi).
    			\label{comp:1}
    		\end{split}
    	\end{equation}
    	
    	Denote now by $\operatorname{d}_{\hat{\nabla}}$ the differential on the tensor product induced by the differentials $\operatorname{d}_{\nabla^A}$ and $\operatorname{d}_{\nabla^{\mathfrak g}}$. By definition
    	\begin{equation}
    		\begin{split} 
    			\operatorname{d}_{\hat{\nabla}}(\omega \otimes \phi):= \operatorname{d}_{\nabla^A} \omega\, \otimes\, \phi + (-1)^{\operatorname{deg}(\omega)} \omega \,\otimes\, \operatorname{d}_{\nabla^{\mathfrak g}} \phi
    		\end{split}
    		\label{def:0}
    	\end{equation}
    	So conclude that
    	\[\operatorname{d}_{\nabla} \circ \mathcal{K} = \mathcal{K} \circ \operatorname{d}_{\hat{\nabla}},\]
    	as expected.
    \end{proof}
    
    Another piece of the proof of Theorem \ref{th:k1} relies on the algebraic K\"unneth isomorphism. The statement below is adapted from e.g.~\cite{Sp95}, Theorem $5.5.11$.

  \begin{lemma} 
      Let $(C,d_c)$ and $(C', d_{C'})$ be two graded differential vector spaces, and let $(C\hat{\otimes}C', d)$ be their tensor product. 
      Assume that $H^\bullet(C', d_{C'})$ is finite dimensional.
      Then the map 
    \[ \mathcal H\colon C\times C'\to C\hat{\otimes}_{\R} C', \qquad (c, c')\mapsto c\otimes c'
    \]
    induces an isomorphism of graded vector spaces in cohomology
    \[\bar{\mathcal H}\colon H^\bullet(C,d)\times H^\bullet(C', d_{C'})\to H^\bullet(C\hat{\otimes}_{\R} C', d).
    \]
    \end{lemma}

    This then implies as follows a  K\"unneth formula in the case of a Lie algebroid and a Lie algebra.
    \begin{corollary}
    	\label{co:K1}
    	The isomorphism of graded vector spaces induced in cohomology by the K\"unneth homomorphism $\mathcal{K}$ defined in (\ref{eq:K3}) is
    	\[\bar{\mathcal{K}} \colon \operatorname{H}^{\bullet}_{\hat{\nabla}}(\Omega(A,E) \,\hat{\otimes}_{\R}\, \Omega(\mathfrak{g},F)) \to \operatorname{H}_{\nabla}^{\bullet}(A \times \mathfrak{g}, \operatorname{pr}^!_M E \otimes \operatorname{pr}_{\{*\}}^!F))\]
    	\[ [\omega \otimes \phi] \mapsto [\operatorname{Pr}_A^*\omega \wedge \operatorname{Pr}^*_{\mathfrak{g}} \phi].  \]
    	Pre-composing with the inverse of the map $\bar{\mathcal{H}}$ of the algebraic K\"unneth isomorphism yields the isomorphism
    	\[\mathcal{K}_{\#} \colon \operatorname{H}^{\bullet}_{\nabla^A}(A, E) \otimes \operatorname{H}^{\bullet}_{\nabla^{\mathfrak{g }}}(\mathfrak{g},F) \to \operatorname{H}_{\nabla}^{\bullet}(A\times \mathfrak{g}, \operatorname{pr}^!_M E \otimes \operatorname{pr}_{\{*\}}^!F).\]
    	
    \end{corollary}

    The necessary background to prove the more general K\"unneth formula of Theorem \ref{th:k1} has been introduced. The proof of Theorem \ref{th:k1} now follows.

    \begin{proof}[Proof of Theorem \ref{th:k1}.] The K\"unneth isomorphism is first proved to be an isomorphism of graded vector spaces.
    	
    	The projections $\operatorname{Pr}_A\colon A \times B \to A$ and $\operatorname{Pr}_B\colon  A \times B \to B$ define the map
    	\[\Psi \colon \Omega^{\bullet}(A, E) \otimes \Omega^{\bullet}(B,F) \to \Omega^{\bullet}(A \times B, \operatorname{pr}^!_M E \otimes \operatorname{pr}^!  _N F)\]
    	\[\omega \otimes \phi \longmapsto \operatorname{Pr}_A^*\omega \wedge \operatorname{Pr}_B^*\phi,\]
    	for $\omega \in \Omega^r(A,E)$, $\phi \in \Omega^s(B,F)$. Recall that
    	\[(\operatorname{Pr}_A^*\omega \wedge \operatorname{Pr}_{B}^*\phi)(a_1, \dots a_r,b_1, \dots, b_s):= \operatorname{pr}_M^*(\omega(a_1, \dots, a_r)) \otimes \operatorname{pr}_{N}^*(\phi(b_1, \dots, b_s)) \]
	for $a_j \in \Gamma(A)$ and $b_k \in \Gamma(B)$.
    	Since $\operatorname{Pr}_A$ is a Lie algebroid morphism, $\Pr_A^*$ intertwines the differentials $\operatorname{d}_A$ and $\operatorname{d}_{A \times B}$, and analogously $\operatorname{Pr}_B^*$ intertwines the differentials $\operatorname{d}_B$ and $\operatorname{d}_{A \times B}$. 

	It is easy to check that the maps 
		\[\operatorname{H}_{\nabla^A}^{\bullet}(A,E) \otimes \operatorname{H}_{\nabla^B}^{\bullet}(B,F) \to \operatorname{H}_{\hat{\nabla}}(\Omega^{\bullet}(A, E) \hat{\otimes}_{\R} \Omega^{\bullet}(B,F)) \to  \operatorname{H}^{\bullet}_{\nabla}(A \times B,  \operatorname{pr}_M^!E \otimes \operatorname{pr}^!_NF) \]
		\[[\omega] \otimes [\phi] \mapsto [\omega \otimes \phi] \mapsto [\operatorname{Pr}_A^*\omega \wedge \operatorname{Pr}_B^*\phi ]\]
		are well defined. The differential that defines the cohomology $\operatorname{H}_{\hat{\nabla}}(\Omega^{\bullet}(A, E) \hat{\otimes}_{\R} \Omega^{\bullet}(B,F))$ is
		\begin{equation*}
			\begin{split} 
				\operatorname{d}_{\hat{\nabla}}(\omega \otimes \phi):= \operatorname{d}_{\nabla^A} \omega\, \otimes\, \phi + (-1)^{\operatorname{deg}(\omega)} \omega \,\otimes\, \operatorname{d}_{\nabla^B} \phi,
			\end{split}
		\end{equation*}
	for $\omega \in \Omega^{\bullet}(A,E)$ and $\phi \in \Omega^{\bullet}(B,F)$.

	It is left to show that the induced map in cohomology
    	\[\Psi\colon \operatorname{H}_{\nabla^A}^{\bullet}(A,E) \otimes \operatorname{H}^{\bullet}_{\nabla^B}(B,F) \longrightarrow \operatorname{H}_{\nabla}^{\bullet}(A \times B, \operatorname{pr}_M^!E \otimes \operatorname{pr}^!_NF)\]
	\[ [\omega]\otimes[\phi]\mapsto \left[\operatorname{Pr}_A^*\omega \wedge \operatorname{Pr}_B^*\phi\right]
	\]
    	is an isomorphism of graded vector spaces.
    	This is proved as follows by induction on the cardinality of a good cover, adapting Bott and Tu's classical argument in \cite{BoTu82}.\\
   
    	If $N$ is contractible, use the generalized Poincaré-Lemma \ref{Th:PL} and the special case of the K\"unneth formula proved in Corollary \ref{co:K1}: For $x \in N$ and the Lie algebra $\mathfrak{g}:=\operatorname{ker}(\rho_A)_x$, 
    	\begin{equation*}
    		\begin{split}
    			\operatorname{H}^{\bullet}(A \times B, \operatorname{pr}_M^!E \otimes \operatorname{pr}^!_NF) \overset{\operatorname{Lemma}\,\ref{Th:PL}}{\cong} \operatorname{H}^{\bullet}(A \times \mathfrak{g}, \operatorname{pr}_M^!E \otimes \operatorname{pr}^!_*(F_x))\\ \overset{\operatorname{Cor.}\,\ref{co:K1}}{\cong} \operatorname{H}^{\bullet}(A,E) \otimes \operatorname{H}^{\bullet}(\mathfrak{g},F_x) \overset{\operatorname{Lemma}\,\ref{Th:PL}}{\cong} \operatorname{H}^{\bullet}(A,E) \otimes \operatorname{H}^{\bullet}(B,F ).
    		\end{split}
    	\end{equation*}

    	In the first step, let $U$ and $V$ be open sets in $N$ and fix an integer $n$. The Mayer-Vietoris sequence is
    {	\small
    		\[ \cdots \rightarrow \operatorname{H}^p(B_{|_{U \cup V}}, F_{|_{U \cup V}}) \rightarrow \operatorname{H}^p(B_{|_U}, F_{|_U}) \,\oplus\, \operatorname{H}^p(B_{|_V}, F_{|_V}) \rightarrow \operatorname{H}^p(B_{|_{U \cap V})}, F_{|_{U \cap V}}) \rightarrow \cdots \]}
    \noindent Tensoring with the $\R$-vector space $\operatorname{H}^{n-p}(A)$ preserves exactness and yields an exact sequence\footnote{For simplicity, diagrams are drawn for the non-twisted LA-cohomology from now on in this proof.}
  {  	\small
    		\[ \cdots \rightarrow \operatorname{H}^{n-p}(A) \otimes \operatorname{H}^p(B_{|_{U \cup V}}) \rightarrow (\operatorname{H}^{n-p}(A) \otimes \operatorname{H}^p(B_{|_U})) \oplus (\operatorname{H}^{n-p}(A) \otimes \operatorname{H}^p(B_{|_V})) \rightarrow \cdots \]}

    	Summing over all integers $p$ preserves exactness
 {   	\small
    		\[ \cdots \rightarrow \bigoplus_{p=0}^n\operatorname{H}^{n-p}(A) \otimes \operatorname{H}^p(B_{|_{U \cup V}}) \rightarrow \bigoplus_{p=0}^n  (\operatorname{H}^{n-p}(A) \otimes \operatorname{H}^p(B_{|_U})) \oplus (\operatorname{H}^{n-p}(A) \otimes \operatorname{H}^p(B_{|_V})) \rightarrow %\bigoplus_{p=0}^n  \operatorname{H}^{n-p}(A) \otimes \operatorname{H}^p(B_{|_{U \cap V})} \cdots 
    		\cdots \]
		}
    	
    	Apply $\Psi$ to each term of the exact sequence. It is left to prove that the resulting diagram commutes, i.e. that the following squares commute:
    	
    	\begin{equation*}
    		{\small
    			\begin{xy}
    				\xymatrix{\bigoplus_{p=0}^n\operatorname{H}^{n-p}(A) \otimes \operatorname{H}^p(B_{|_{U \cup V}}) \ar[r] \ar[d]^{\Psi}
    					&  \bigoplus_{p=0}^n  (\operatorname{H}^{n-p}(A) \otimes \operatorname{H}^p(B_{|_U})) \oplus (\operatorname{H}^{n-p}(A) \otimes \operatorname{H}^p(B_{|_V})) \ar[d]^{\Psi} \\
    					\operatorname{H}^n(A \times B_{|_{U \cap V}}) \ar[r] & \operatorname{H}^n(A \times B_{|_U}) \oplus \operatorname{H}^n(A \times B_{|_V}), \\
    				} 
    		\end{xy}}
    	\end{equation*}
    	\bigskip
    	
    	\begin{equation*}
   { 		\small
    			\begin{xy}
    				\xymatrix{\bigoplus_{p=0}^n  (\operatorname{H}^{n-p}(A) \otimes \operatorname{H}^p(B_{|_U})) \oplus (\operatorname{H}^{n-p}(A) \otimes \operatorname{H}^p(B_{|_V})) \ar[r] \ar[d]^{\Psi} & \bigoplus_{p=0}^n  \operatorname{H}^{n-p}(A) \otimes \operatorname{H}^p(B_{|_{U \cap V}}) \ar[d]^{\Psi} \\
    					\ar[r] \operatorname{H}^n(A \times B_{|_U}) \oplus \operatorname{H}^n(A \times B_{|_V}) &\operatorname{H}^n(A \times B_{|_{U \cap V}}), \\
    				} 
    		\end{xy}}
    	\end{equation*}
    	and
    	\begin{equation*}
    		\small{
    			\begin{xy}
    				\xymatrix{\bigoplus_{p=0}^n  \operatorname{H}^{n-p}(A) \otimes \operatorname{H}^p(B_{|_{U \cap V}}) \ar[r]^{\pm \operatorname{d}^*_{\nabla}} \ar[d]^{\Psi}
    					& \bigoplus_{p=0}^n\operatorname{H}^{n-p}(A) \otimes \operatorname{H}^{p+1}(B_{|_{U \cup V}}) \ar[d]^{\Psi} \\
    					\operatorname{H}^n(A \times B_{|_{U \cap V}}) \ar[r]^{\operatorname{d}^*_{\nabla}} &\operatorname{H}^{n+1}(A \times B_{|_{U \cup V}}), \\
    				} 
    		\end{xy}}
    	\end{equation*}
    	for each $p,n \in \N$, $n \ge p$. Here, to be precise, the map  ``$\pm \operatorname{d}^*_{\nabla}$''  is given 
	by 
	\[[\omega]\otimes [\eta]\mapsto (-1)^{\deg \omega}[\omega]\otimes (\operatorname{d}^*_{\nabla}[\eta]).\]

	Recall the explicit expression for the coboundary map of the Mayer-Vietoris sequence: if $\{\rho_U, \rho_V\}$ form a partition of unity subordinate to the cover $\{U,V\}$ of $N$ then, by Theorem \ref{MV},
    	\[\operatorname{d}_{\nabla}^*([\phi]):= 
    	\begin{cases} 
    		[-\operatorname{d}_{\nabla}(\mu_V\phi)] & \operatorname{on}\, U\\
    		[\operatorname{d}_{\nabla}(\mu_U\phi)] & \operatorname{on}\, V.\\
    	\end{cases}
    	\]
The commutativity of the first two diagrams then  follows by definition. To check the commutativity of the third, notice that for $[\omega] \otimes [\phi] \in \operatorname{H}^{n-p}(A,E) \otimes \operatorname{H}^p(B_{|_{U \cap V}}, F_{|_{U \cap V}})$
    	\begin{equation*}
	\begin{split}
	\Psi(\operatorname{d}^*_{\nabla}([\omega] \otimes [\phi]))&=(-1)^{\deg(\omega)}\Psi([\omega]\otimes d_\nabla^*[\varphi])=
	(-1)^{\deg(\omega)}\Psi([\omega]\otimes[\operatorname{d}_{\nabla}(\mu_U\phi)])\\
	&=
	(-1)^{\deg(\omega)} [\operatorname{Pr}_A^* \omega \wedge \operatorname{Pr}_B^*(\operatorname{d}_{\nabla}(\mu_U\phi))]. \end{split}\end{equation*}
%    	\[\operatorname{d}^*_{\nabla}\Psi[\omega] \otimes [\phi]= \operatorname{d}^*_{\nabla}[\operatorname{Pr}^*_A \omega \wedge \operatorname{Pr}^*_B\phi]\]
    	    	Note that $\{\operatorname{pr}_N^*\mu_U, \operatorname{pr}_N^*\mu_V \}$ is a partition of unity on $M \times N$ subordinate to the cover $\{M \times U, M \times V\}$. On the intersection $M \times (U \cap V)$ of these two sets, $\operatorname{d}_\nabla^*$ is defined by 
    	\begin{equation*}
    		\begin{split}
    			\operatorname{d}^*_{\nabla}\Psi([\omega] \otimes [\phi])
			&=\operatorname{d}^*_{\nabla}\left[\operatorname{Pr}_A^*\omega \wedge \operatorname{Pr}^*_B \phi\right]=\left[\operatorname{d}_{\nabla} ((\operatorname{pr}_N^*\mu_U)\operatorname{Pr}_A^*\omega \wedge \operatorname{Pr}^*_B \phi)\right]\\
    			&=\left[\operatorname{d}_{\nabla}(\operatorname{Pr}_A^*\omega \wedge \operatorname{Pr}^*_B (\mu_U\phi))\right]
			=(-1)^{\deg(\omega)}\left[\operatorname{Pr}_A^*\omega \wedge \operatorname{d}_{\nabla}(\operatorname{Pr}^*_B (\mu_U\phi))\right]\\
%			&
%			= \overline{\operatorname{Pr}_A^*}[\omega] \wedge \overline{\operatorname{Pr}_B^*}(\operatorname{d}^*_{\nabla}[\phi])
    		\end{split}
    	\end{equation*}
where the last passage holds because $\omega$ is closed. As in \cite{BoTu82}, observe that if the theorem is true in $U$, $V$ and $U \cap V$, then it holds also in $U \cup V$. Conclude by a standard inductive argument on the cardinality of a good cover (see e.g.\cite{BoTu82}, Lemma $5.6$).\\
    
    It remains to prove that the obtained isomorphism of graded vector spaces actually is an isomorphism of modules over a ring isomorphism. With a little abuse of notation, use $\Psi$ to denote the K\"unneth isomorphism in both the twisted and non-twisted case. By the previous part of the proof,
    	\[\Psi \colon \operatorname{H}^{\bullet}(A) \otimes \operatorname{H}^{\bullet}(B) \to \operatorname{H}^{\bullet}(A \times B)\]
    	\[\omega\otimes \phi \longmapsto \operatorname{Pr}_A^*\omega \wedge \operatorname{Pr}_B^*\phi\]
    	is an isomorphism of graded vector spaces. A straightforward computation implies that it is also an isomorphism of rings.
    	Take $\omega_j \in \operatorname{H}^{k_j}(A)$ and $\phi_j \in \operatorname{H}^{l_j}(B)$, $j=1,2$. Then
    	\begin{equation*}
    		\begin{split}
    			\Psi((\omega_1 \otimes \phi_1) (\omega_2 \otimes \phi_2))&= \Psi((-1)^{\l_1 k_2}(\omega_1 \wedge \omega_2) \otimes (\eta_1 \wedge \eta_2)) \\
			&= (-1)^{\l_1 k_2} \operatorname{Pr}_A^*(\omega_1 \wedge \omega_2) \wedge \operatorname{Pr}_B^*(\phi_1 \wedge \phi_2)\\
			&= (\operatorname{Pr}_A^*\omega_1 \wedge \operatorname{Pr}_B^*\phi_1) \wedge (\operatorname{Pr}_A^*\omega_2 \wedge \operatorname{Pr}_B^*\phi_2)
		\\&	=
    			\Psi(\omega_1 \otimes \phi_1) \wedge
    			\Psi(\omega_2 \otimes \phi_2) 
    		\end{split}
    	\end{equation*}
    	The exact same computation, with entries 
    	$\omega_1\otimes \phi_1 \in \operatorname{H}^{\bullet}(A) \otimes \operatorname{H}^{\bullet}(B)$ and $\omega_2\otimes \phi_2 \in \operatorname{H}_{\nabla^A}^{\bullet}(A, E) \otimes \operatorname{H}_{\nabla^B}^{\bullet}(B,F)$ shows that the K\"unneth isomorphism $\Psi$ is a module isomorphism. 
    \end{proof}

 \section{Invariance of representations under homotopies} \label{sec:6}
 The goal of this section is to prove Theorem \ref{th:hi}. The proof uses flows of linear vector fields on vector bundles and in particular also of linear vector fields defined by Lie algebroid derivations on Lie algebroids. 

\subsection{Linear vector fields on vector bundles}
Let $q_F\colon F\to M$ be a vector bundle.  Then the tangent bundle
$TF$ has two vector bundle structures; one as the tangent bundle of
the manifold $F$, and the second as a vector bundle over $TM$. The
structure maps of $TF\to TM$ are the derivatives of the structure maps
of $F\to M$.
\begin{equation*}
\begin{xy}
\xymatrix{
TF \ar[d]_{Tq_F}\ar[r]^{p_F}& F\ar[d]^{q_F}\\
 TM\ar[r]_{p_M}& M}
\end{xy}
\end{equation*} 
The space $TF$ is a \textbf{double vector bundle}. That is, the diagram above commutes and all structure maps of the horizontal vector bundles are morphisms of vector bundles with respect to the vertical vector bundles, see \cite{Mackenzie05} for more details and references.
A section $e \in \Gamma(F)$ defines \textbf{the vertical vector field} 
$e^{\uparrow}\colon F \to TF$, i.e.~the vector field with
flow $\phi^{e^\uparrow}\colon F\times \R\to F$, $\phi_t(e'_m)=e'_m+te(m)$. 
A vector field $\chi\in \mx(F)$ is called linear, if it is a vector bundle morphism $F\to TF$ over a vector field $X\colon M\to TM$ on $M$.
The space of linear vector fields on $F$ is written $\mx^\ell(F)$. It is well-known (see e.g.~\cite{Mackenzie05}) that a
linear vector field $\chi\in\mx^l(F)$ covering $X\in\mx(M)$ corresponds
to a derivation $D\colon \Gamma(F) \to \Gamma(F)$ over $X\in
\mx(M)$. The precise correspondence is
given by the following equations
\begin{equation}\label{ableitungen}
\chi(\ell_{\varepsilon}) 
= \ell_{D^*(\varepsilon)} \,\,\,\, \operatorname{ and }  \,\,\, \chi(q_F^*f)= q_F^*(X(f))
\end{equation}
for all $\varepsilon\in\Gamma(F^*)$ and $f\in C^\infty(M)$, where
$D^*\colon\Gamma(F^*)\to\Gamma(F^*)$ is the dual derivation to $D$
\cite[\S3.4]{Mackenzie05}.  The linear vector field in $\mx^l(F)$
corresponding in this manner to a derivation $D$ of $\Gamma(F)$ is
written $\widehat D$.  Given a derivation $D$ over $X\in\mx(M)$, the
explicit formula for $\widehat D$ is
\begin{equation}\label{explicit_hat_D}
\widehat D(e_m)=T_me(X(m))+_F\left.\frac{d}{dt}\right\an{t=0}(e_m-tD(e)(m))
\end{equation}
for $e_m\in F$ and any $e\in\Gamma(F)$ such that $e(m)=e_m$.  
Let $D$, $D_1$ and $D_2$ be linear derivations of $F$ with symbols $X$, $X_1$ and $X_2\in\mx(M)$, respectively, and let $e, e_1,e_2\in\Gamma(F)$.
It is easy to see, using the equalities
in \eqref{ableitungen}, that
\begin{equation}\label{Lie_bracket_over_VB}
\begin{split}
\left[\widehat{D_1}, \widehat{D_2}\right]&=\widehat{[D_1,D_2]},\qquad 
\left[D, e^\uparrow\right]=(\nabla_Xe)^\uparrow,\qquad 
\left[e_1^\uparrow,e_2^\uparrow\right]=0.
\end{split}
\end{equation}
Since linear and vertical vector fields on $F$ generate the whole space of vector fields on $F$ as a $C^\infty(F)$-module, these equations contain the ``whole geometry of $F$''.

\medskip

The following lemma can be found in \cite{Mackenzie05}. It is proved in detail in  \cite{Jotz24}.
\begin{lemma}\label{lemma_linear_flow}
Let $F\to M$ be a smooth vector bundle and let $D\colon \Gamma(F)\to \Gamma(F)$ be a derivation with symbol $X\in\mx(M)$. Let $\xi\colon \Omega\to M$ be the flow of $X$, with, as usual, $\Omega\subseteq \R\times M$ the open flow domain containing $\{0\}\times M$. Then the flow $\Xi$ of $\widehat{D}\in\mx^l(F)$ is defined on 
\[ \widetilde\Omega:=\left\{ (t, e)\in \R\times F \mid (t,q(e))\in \Omega
\right\}
\]
and for each $t\in\R$, the map
\[ \Xi_t\colon q\inv(M_t)\to q\inv(M_{-t})
\]
is an isomorphism of vector bundles over the diffeomorphism $\xi_t\colon M_t\to M_{-t}$, where 
$M_t\subseteq M$ is the open subset
\[ M_t:=\{x\in M\mid (t,x)\in\Omega\}.
\]
\end{lemma}

\subsection{Invariance of a Lie algebroid under the flow of its adjoint vector fields}\label{invariance_ad_a}
Let $A\to M$ be a Lie algebroid and choose a section $a\in\Gamma(A)$. Then the bracket with $a$
\[ [a,\cdot]\colon \Gamma(A)\to \Gamma(A)
\]
is a derivation of the vector bundle $A$, with symbol $\rho(a)\in\mx(M)$. It is more precisely a derivation of the Lie algebroid $A$. That is, it also satisfies
\[ [a,[b,c]]=[[a,b],c]+[b,[a,c]]
\] 
for all $b,c\in\Gamma(A)$. 
More generally, this section considers a linear derivation 
\[ D\colon \Gamma(A)\to \Gamma(A)
\]
of the Lie algebroid $A$, that is, a linear derivation of the vector bundle $A\to M$ with symbol $X\in\mx(M)$, such that in addition 
\[ D[b,c]=[Db,c]+[b,Dc]
\]
for all $b,c\in\Gamma(A)$.

This section refines Lemma \ref{lemma_linear_flow} to the following proposition in the case of a linear derivation of a Lie algebroid $A\to M$.
\begin{proposition}\label{invariance_a}
Let $q\colon A\to M$ be a Lie algebroid and let $D$ be a linear derivation of \emph{the Lie algebroid} $A$ with symbol $X\in\mx(M)$. Let $\xi\colon \Omega\to M$ be the flow of $X$ and let 
\[ \Xi\colon \tilde\Omega\to A
\]
be the flow of $\widehat{D}\in\mx^l(A)$. Then for each $t\in\mathbb R$ the map
\[ \Xi_t\colon q\inv(M_t)\to q\inv(M_{-t})
\]
is an isomorphism of Lie algebroids over the diffeomorphism $\xi_t\colon M_t\to M_{-t}$.
\end{proposition}

The proof of this proposition relies on the following lemma.
\begin{lemma}\label{lemma_flat_sections}
Let $F\to M$ be a vector bundle and let $D\colon \Gamma(F)\to\Gamma(F)$ be a linear derivation of the vector bundle $F$ with symbol $X\in\mx(M)$.
Let $\Xi$ be the flow of $\widehat D$ and let $\xi$ be the flow of $X$. For $t\in \R$ and $e\in\Gamma(F)$ define 
\[ \Xi_t^\star e \in \Gamma_{M_t}(F)
\]
to be the section 
\[ (\Xi_t^\star e)_x= \Xi_{-t}(e(\xi_t(x)))
\]
for all $x\in M_t$. Then\footnote{This derivative is taken pointwise, i.e.~in fibers of $F$.}
\begin{equation*}
De=\left.\frac{d}{dt}\right\an{t=0} \Xi_t^\star e.
\end{equation*}
In particular $\Xi_t^\star e = e$ on $M_t$ for all $t$ if and only if $De=0$. 
\end{lemma}

\begin{proof}
For  $t\in\mathbb R$, $x\in M_t$ and $e'\in F_x$ compute using the linearity of $\Xi_t$
\begin{equation*}
\begin{split}
(\Xi_t^*(e^\uparrow))\an{e'}
&=  T_{\Xi_t(e')}\Xi_{-t}\left( e^\uparrow (\Xi_t(e'))\right)\\
&=   \left.\frac{d}{ds}\right\an{s=0}\Xi_{-t}\left( \Xi_t(e')+se(\xi_t(x))\right)\\
&=  \left.\frac{d}{ds}\right\an{s=0}e'+s \Xi_{-t}(e(\xi_t(x)))=(\Xi^\star_te)^\uparrow\an{e'},
\end{split}
\end{equation*}
which shows 
\[ \Xi_t^*\left(e^\uparrow\right)=\left(\Xi_t^\star(e)\right)^\uparrow
\]
on $q\inv(M_t)$.
Hence 
\begin{equation*}
\begin{split}
(De)^\uparrow \an{e'} &=\left.\left[ \widehat D, e^\uparrow\right]\right\an{e'}=\left.\frac{d}{dt}\right\an{t=0}(\Xi_t^*(e^\uparrow))\an{e'}=\left.\frac{d}{dt}\right\an{t=0}(\Xi_t^\star(e))^\uparrow\an{e'}\\
%&= \left.\frac{d}{dt}\right\an{t=0} \left.\frac{d}{ds}\right\an{s=0}\left(e'+s (\Xi_{t}^\star e)(x)\right)
%= \left.\frac{d}{ds}\right\an{s=0} \left.\frac{d}{dt}\right\an{t=0}\left(e'+s (\Xi_{t}^\star e)(x)\right),
\end{split}\end{equation*}
which equals
\begin{equation*}\begin{split}
&\left.\left(\left.\frac{d}{dt}\right\an{t=0}\Xi_t^\star e\right)^\uparrow\right\an{e'}.\end{split}
\end{equation*}

This implies as well the second part of the statement: If $\Xi_t^\star e = e$ on $M_t$ for all $t$ then clearly $De=0$. Conversely, if $De=0$ then 
\begin{equation*}
\begin{split}
0=\Xi_t^*(De)^\uparrow &=\left(\Xi_t^*\left[ \widehat D, e^\uparrow\right]\right)=\frac{d}{dt}(\Xi_t^*(e^\uparrow))=\frac{d}{dt}(\Xi_t^\star e)^\uparrow
=\left(\frac{d}{dt}\Xi_t^\star e\right)^\uparrow
\end{split}
\end{equation*}
shows that $\frac{d}{dt}\Xi_t^\star e=0$, and so that $\Xi^\star_te=\Xi^\star_0e=e$ for all $t\in\mathbb R$. \end{proof}

\begin{lemma}\label{invariance_anchor_LA_derivations}
Let $A\to M$ be a Lie algebroid and let $D$ be a derivation of $A$, as a Lie algebroid. Let $X$ be the symbol of $D$, and, as before, let $\xi\colon \Omega\to M$ 
be the flow of $X$, and let $\Xi\colon \tilde \Omega\to F$ be the flow of $\widehat{D}$.
Choose a smooth section $a\in \Gamma(A)$ and a time $t\in\mathbb R$. Then on $M_t$ the equality 
\[ \rho(\Xi^\star_t a)=\xi_t^*(\rho(a)).
\]
holds.
\end{lemma}

\begin{proof}
The anchor map $\rho\colon A\to TM$, which is a section of $A^*\otimes TM$, is send to $0$ by the derivation $\tilde D$ of $A^*\otimes TM$ defined by 
\[ \tilde D(\phi)(a)=\phi(Da)-[X,\phi(a)]
\]
for all $\phi\in\Gamma(A^*\otimes TM)$ and all $a\in\Gamma(A)$. Note that $\tilde D$ is the linear derivation over $X$  of $A^*\otimes TM$ defined by the two linear derivations $D^*$ on $A^*$ and $[X,\cdot]$ on $TM$, both with symbol $X$. In order to see that $\tilde D(\rho)=0$, take two arbitrary sections $a,b\in\Gamma(A)$ and a smooth function $f\in C^\infty(M)$.
Then \begin{equation*}
\begin{split}
\rho(Da)(f)\cdot b&=[Da,f\cdot b]-f\cdot [Da, b]\\
&= D[a, f\cdot b]-[a, D(f\cdot b)]-f\cdot D[a,b]+f\cdot [a,Db]\\
&=D\left(f\cdot [a,b]+\rho(a)(f)\cdot b\right)-\left[a, X(f)\cdot b+f\cdot Db\right]\\
&\qquad-f\cdot D[a,b]+f\cdot [a,Db]\\
&=\bcancel{X(f)\cdot [a,b]}+\cancel{f\cdot D[a,b]}+X(\rho(a)(f))\cdot b+\xcancel{\rho(a)(f)\cdot Db}\\
&\qquad  -\rho(a)(X(f))\cdot b-\bcancel{X(f)\cdot [a,b]}-\xcancel{\rho(a)(f)\cdot Db}-\cancel{f\cdot[a,Db]}\\
&\qquad -\cancel{f\cdot D[a,b]}+\cancel{f\cdot [a,Db]}\\
&=[X,\rho(a)](f)\cdot b.
\end{split}
\end{equation*}
Since $b$ and $f$ were arbitrary, this shows $\rho(Da)=[X,\rho(a)]$, and since $a$ was arbitrary, this is $\tilde D(\rho)=0$.
By the last part of Lemma \ref{lemma_flat_sections}, this implies that $\rho$ is invariant under the flow of $\widehat{\tilde D}\in\mx^l(A^*\otimes TM)$.

Using \cite{Jotz24}, this flow is defined at time $t$ by 
\[ \Xi^{\tilde D}_t=\Xi^{D^*}_t\otimes \Xi^{[X,\cdot]}_t\colon (A^*\otimes TM)\an{M_t}\to (A^*\otimes TM)\an{M_{-t}},
\]
where $\Xi^{D^*}$ is the flow of $\widehat{D^*}\in\mx^l(A^*)$ and $\Xi^{[X,\cdot]}$ is the flow of $\widehat{[X,\cdot]}\in\mx^l(TM)$.
By \cite{Jotz24}, $\Xi^{D^*}$ is defined on $\{(t,\epsilon)\mid (t,q_{A^*}(\epsilon))\in\Omega\}$ by 
\[ (t,\epsilon)\mapsto \epsilon\circ \Xi_{-t}.
\]
A standard computation shows that the flow of $\widehat{[X,\cdot]}\in\mx^l(TM)$ is defined on 
$\{(t,v)\mid (t,p_M(v))\in\Omega\}$ by 
\[ (t,v)\mapsto T\xi_t v.
\]
As a consequence, for all $t\in\mathbb R$, $p\in M_t$ and $a\in\Gamma(A)$
\[\rho\left(a\an{\xi_t(p)}\right)=\Xi^{[X,\cdot]}_t\left( \rho\left(\Xi_{-t}\left(a\an{\xi_t(p)}\right)\right)\right)=T_p\xi_t(\rho(\Xi_t^\star(a)\an{p})),
\]
which is the same as
\[(\xi_t^*(\rho(a))\an{p}=T_{\xi_t(p)}\xi_{-t}\left(\rho\left(a\an{\xi_t(p)}\right)\right)=\rho(\Xi_t^\star(a)\an{p}). \qedhere
\]
\end{proof}

\begin{proof}[Proof of Proposition \ref{invariance_a}]
It suffices to prove that for fixed $a,b\in\Gamma(A)$ and $p\in M$, the map $t\mapsto \Xi_{-t}^\star\left[\Xi_t^\star a,\Xi_t^\star b\right]\an{p}$ is constant, since it must then equal $[a,b]\an{p}$, its value at $t=0$, and the identity $\Xi^\star_t[a,b]=\left[\Xi_t^\star a,\Xi_t^\star b\right]$ then follows for all $t$.

First take a local frame $(a_1,\ldots, a_r)$ of $A$, on an open set $U\subseteq M$, and let $c_{ij}^k\in C^\infty(U)$ be the structure constants of the Lie algebroid in this frame, i.e.
\[ [a_i,a_j]=\sum_{k=1}^r c_{ij}^k a_k
\]
for $i,j=1,\ldots, r$. Define $f_{ik}\in C^\infty(\Omega\cap(\mathbb R\times U))$, for $i,k=1,\ldots, r$ such that for all $t$
\[ \Xi_t^\star a_i=\sum_{k=1}^r f_{ik}(t,\cdot)a_k
\]
on $M_t\cap U$. Then, as usual
\begin{equation*}
\begin{split}
\left[ \Xi_t^\star a_i, \Xi_t^\star a_j\right]&=\sum_{k,l=1}^r\sum_{n=1}^rf_{ik}(t,\cdot)f_{jl}(t,\cdot)c_{kl}^na_n+\sum_{k,l=1}^rf_{ik}(t,\cdot)\rho(a_k)(f_{jl}(t,\cdot))a_l\\
&\qquad -\sum_{k,l=1}^rf_{jl}(t,\cdot)\rho(a_l)(f_{ik}(t,\cdot))a_k,
\end{split}
\end{equation*}
which shows that\footnote{This equation would be immediate if the Lie bracket on sections of $A$ was defined pointwise. But since it is defined locally, the usual argument does not work in a straightforward manner since the Lie bracket cannot be 'exchanged' with the derivative without further checking.}
\[ \frac{d}{dt}\left[ \Xi_t^\star a_i, \Xi_t^\star a_j\right]=\left[\frac{d}{dt}\Xi_t^\star a_i, \Xi_t^\star a_j\right]+\left[\Xi_t^\star a_i, \frac{d}{dt}\Xi_t^\star a_j\right].
\]
Together with Lemma \ref{invariance_anchor_LA_derivations} this implies that for $f\in C^\infty (U)$ and $i,j\in\{1,\ldots, r\}$
\begin{equation*}
\begin{split}
&\frac{d}{dt}\left[ \Xi_t^\star a_i, \Xi_t^\star (fa_j)\right]=\frac{d}{dt}\left[ \Xi_t^\star a_i, \xi_t^*f\cdot\Xi_t^\star a_j\right]\\
&=\frac{d}{dt}\Bigl(\rho\left( \Xi_t^\star a_i\right)(\xi_t^*f)\cdot\Xi_t^\star a_j+ \xi_t^*f\cdot \left[ \Xi_t^\star a_i, \Xi_t^\star a_j\right]\Bigr)\\
&=\frac{d}{dt}\Bigl(\xi_t^*(\rho(a_i))(\xi_t^*f)\cdot\Xi_t^\star a_j+ \xi_t^*f\cdot \left[ \Xi_t^\star a_i, \Xi_t^\star a_j\right]\Bigr)\\
&=\frac{d}{dt}\Bigl(\xi_t^*(\rho(a_i)(f))\cdot\Xi_t^\star a_j+ \xi_t^*f\cdot \left[ \Xi_t^\star a_i, \Xi_t^\star a_j\right]\Bigr)\\
&=\xi^*_t(X(\rho(a_i)(f)))\cdot \Xi_t^\star a_j+\xi_t^*(\rho(a_i)(f))\cdot\frac{d}{dt}\Xi_t^\star a_j +\xi_t^*(X(f))\cdot \left[ \Xi_t^\star a_i, \Xi_t^\star a_j\right]\\
&\qquad +\xi_t^*(f)\cdot \left(\left[ \frac{d}{dt}\Xi_t^\star a_i, \Xi_t^\star a_j\right]+\left[\Xi_t^\star a_i, \frac{d}{dt}\Xi_t^\star a_j\right]\right).
\end{split}
\end{equation*}
The second and the last term add up to $\left[ \Xi_t^\star a_i, \xi_t^*(f)\cdot\frac{d}{dt}\Xi_t^\star a_j\right]$ and the first one equals
$\xi^*_t(\rho(a_i)(X(f)))\cdot \Xi_t^\star a_j+\xi_t^*([X,\rho(a_i)](f))\cdot \Xi_t^\star a_j $. The first one of these two terms and 
the third term $\xi_t^*(X(f))\cdot \left[ \Xi_t^\star a_i, \Xi_t^\star a_j\right]$ in the sum above add up to 
\[\left[\Xi_t^\star a_i, \xi_t^*(X(f))\cdot \Xi_t^\star a_j\right].
\] 
Hence \begin{equation*}
\begin{split}
\frac{d}{dt}\left[ \Xi_t^\star a_i, \Xi_t^\star (fa_j)\right]&=\left[ \Xi_t^\star a_i, \xi_t^*(f)\frac{d}{dt}\Xi_t^\star a_j\right]+\left[\Xi_t^\star a_i, \xi_t^*(X(f))\cdot \Xi_t^\star a_j\right]\\
&\qquad +\xi_t^*([X,\rho(a_i)](f))\cdot \Xi_t^\star a_j+\xi_t^*(f)\cdot \left[ \frac{d}{dt}\Xi_t^\star a_i, \Xi_t^\star a_j\right]\\
&=\left[ \Xi_t^\star a_i, \frac{d}{dt}\left(\xi_t^*(f)\cdot\Xi_t^\star a_j\right)\right]
+\xi_t^*([X,\rho(a_i)](f))\cdot \Xi_t^\star a_j\\
&\qquad +\xi_t^*(f)\cdot \left[ \frac{d}{dt}\Xi_t^\star a_i, \Xi_t^\star a_j\right].
\end{split}
\end{equation*}
Next note that 
\begin{equation*}
\begin{split}
\xi_t^*([X,\rho(a_i)](f))&=\xi_t^*[X,\rho(a_i)](\xi_t^*f)=\left(\frac{d}{dt}\xi_t^*(\rho(a_i))\right)(\xi_t^*f)\\
&=\left(\frac{d}{dt}\rho(\Xi_t^\star a_i)\right)(\xi_t^*f)=\rho\left(\frac{d}{dt}\Xi_t^\star a_i\right)(\xi_t^*f).
\end{split}
\end{equation*}
Hence 
\begin{equation*}
\begin{split}
\frac{d}{dt}\left[ \Xi_t^\star a_i, \Xi_t^\star (fa_j)\right]
&=\left[ \Xi_t^\star a_i, \frac{d}{dt}\left(\xi_t^*(f)\cdot\Xi_t^\star a_j\right)\right]
+\left[ \frac{d}{dt}\Xi_t^\star a_i, \xi_t^*(f)\cdot \Xi_t^\star a_j\right]\\
&=\left[ \Xi_t^\star a_i, \frac{d}{dt}\Xi_t^\star(f\cdot a_j)\right]
+\left[ \frac{d}{dt}\Xi_t^\star a_i, \Xi_t^\star(f\cdot a_j)\right].
\end{split}
\end{equation*}
This lengthy computation shows that 
\[ \frac{d}{dt}\left[\Xi_t^\star a,\Xi_t^\star b\right]=\left[\frac{d}{dt}\Xi_t^\star a,\Xi_t^\star b\right]+\left[\Xi_t^\star a,\frac{d}{dt}\Xi_t^\star b\right]
\]
for all $a,b\in\Gamma(A)$.

\medskip
So consider $a,b\in\Gamma(A)$ and compute 
\begin{equation*}
\begin{split}
\frac{d}{dt}\left(\Xi^\star_{-t}\left[\Xi_t^\star a, \Xi_t^\star b\right]\right)&=-\Xi^\star_{-t}D\left[\Xi_t^\star a, \Xi_t^\star b\right]
+\Xi^\star_{-t}\left[D\Xi_t^\star a, \Xi_t^\star b\right]+\Xi^\star_{-t}\left[\Xi_t^\star a, D\Xi_t^\star b\right]\\
&=\Xi^\star_{-t}\left(-D\left[\Xi_t^\star a, \Xi_t^\star b\right]
+\left[D\Xi_t^\star a, \Xi_t^\star b\right]+\left[\Xi_t^\star a, D\Xi_t^\star b\right]\right)=0.
\end{split}
\end{equation*}
As explained above, this shows the claim.
\end{proof}

Finally, Proposition \ref{invariance_a} has the following corollary.
\begin{corollary}\label{cor_semi_direct}
Let $A\to M$ be a Lie algebroid with anchor $\rho$ and let $F\to M$ be a smooth vector bundle. Assume that $\nabla\colon \Gamma(A)\times\Gamma(F)\to\Gamma(F)$ is a flat $A$-connection on $F$. Choose a smooth section $a\in\Gamma(A)$
and let $\xi\colon \Omega\to M$ be the flow of the vector field $\rho(a)$. Let 
\[
\Xi\colon \widetilde{\Omega}^F:=\{(\lambda,e)\in \mathbb R\times F\mid (\lambda,q_F(e))\in\Omega\}\to F\]
 be the flow of $\widehat{\nabla_a}\in\mx^l(F)$, and let 
 \[\Psi\colon \widetilde{\Omega}^A:=\{(\lambda,b)\in \mathbb R\times A\mid (\lambda,q_A(b))\in\Omega\}\to A\]
  be the flow of $\widehat{[a,\cdot]}\in\mx^l(A)$.
Then for all $\lambda\in\mathbb R$, $e\in\Gamma(F)$ and all $b\in\Gamma(A)$ 
\[ \nabla_{\Psi_\lambda^\star b}(\Xi_\lambda^\star e)=\Xi_\lambda^\star (\nabla_be).
\]
\end{corollary}

\begin{proof}
Consider the smooth map 
\[ \Pi\colon \widetilde\Omega:=\{(\lambda, b,e)\in\mathbb R\times (A\oplus F)\mid (\lambda, q_A(b))\in\Omega\}\to A\oplus F,
\]
\[ \Pi(\lambda, b,e)=\left(\Psi_\lambda(b), \Xi_\lambda(e)\right).
\]
It obviously has the flow property and satisfies 
\[ \left.\frac{d}{d\lambda}\right\arrowvert_{\lambda=0} \Pi_\lambda(b,e)=\left( \widehat{[a,\cdot]}(b), \widehat{\nabla_a}(e)\right).
\]
Hence $\Pi$ is the flow of the vector field $\widehat{D_a}$ on $A\oplus F$, where $D_a$ is the derivation of $A\oplus F$ with symbol $\rho(a)$ given by 
\[ D_a(b,e)=([a,b],\nabla_ae)=[(a,0), (b,e)]_\nabla
\]
for all $(b,e)\in\Gamma(A\oplus F)$.
Here, $[\cdot\,,\cdot]_\nabla$ is the Lie algebroid bracket on $A\oplus F$ defined by the semi-direct product of $A$ with its representation $\nabla$ on $F$:
\[ [(a_1,e_1), (a_2,e_2)]_\nabla=([a_1,a_2], \nabla_{a_1}e_2-\nabla_{a_2}e_1)
\]
and with anchor $\rho_{A\oplus F}\colon A\oplus F\to TM$, $\rho_{A\oplus F}(a,e)=\rho(a)$.
Note that the Jacobi identity for this bracket is equivalent to the Jacobi identity of the Lie algebroid $A$ and the flatness of $\nabla$.

Assume for simplicity that $\rho(a)$ is complete.
By Proposition \ref{invariance_a}, the flow $\Pi$ of $\widehat{D_a}$ is hence by isomorphisms of the Lie algebroid $A\oplus F$. As a consequence 
\begin{equation*}
\begin{split}
\left(0, \nabla_{\Psi_\lambda^\star b}(\Xi_\lambda^\star e)\right)&=\left[(0,\Psi_\lambda^\star b), (\Xi_\lambda^\star e, 0)\right]_\nabla
=\left[\Pi^\star_\lambda(b,0), \Pi^\star_\lambda(0,e)\right]_\nabla\\
&=\Pi_\lambda^\star[(b,0), (0, e)]_\nabla=\Pi_\lambda^\star(0,\nabla_be)=\left(0,\Xi_\lambda^\star(\nabla_be)\right)
\end{split}
\end{equation*}
for all $\lambda\in\mathbb R$.
\end{proof}

\subsection{Proof of Theorem \ref{th:hi}. }\label{ap:A}This section finishes the proof of Theorem \ref{th:hi}.
\medskip

Let $I\subseteq \mathbb R$ be an open interval and let $A\to M$ be a Lie algebroid.
	Let  $\nabla$ be a flat $(TI \times A)$-connection on a smooth vector bundle $F \to I\times M$. For fixed $s\in I$ consider the Lie algebroid morphism
	\[ \mathcal{I}_s\colon A\to TI\times A, \qquad a_m\mapsto (0_{s}, a_m)
	\]
	over $\iota_s\colon M\to I\times M$, $m\mapsto (s,m)$.
The connection $\nabla$ defines for each $s\in I$ the flat $A$-connection \[\mathcal{I}_s^*\nabla\colon \Gamma(A)\times \Gamma(\iota_s^!F)\to \Gamma(\iota_s^!F), 
	\]
\[ (\mathcal{I}_s^*\nabla)_{a_m}(\iota_s^!e)=\nabla_{(0_s,a_m)}e \in F\arrowvert_{(s,m)}=(\iota_s^!F)\arrowvert_m
\]
for $a_m\in A$ and $e\in\Gamma(F)$.
The proof consists in several steps.
	
	\medskip
	\subsubsection{Step 1: Construction of a gauge equivalence from $\mathcal I_t^*\nabla$ to $\mathcal I_{s}^*\nabla$ for $s,t\in I$.}\label{step1}
	The first step of the proof is the construction of a gauge equivalence $\theta_{t,s}\colon \iota_t^!F\to\iota_s^!F$ from $\mathcal I_t^*\nabla$ to $\mathcal I_{s}^*\nabla$, which induces an isomorphism in cohomology
		\[ (\theta_{t,s})_*\colon H^\bullet_{\mathcal I_t^*\nabla}(A,\iota_t^!F)\to  H^\bullet_{\mathcal I_s^*\nabla}(A,\iota_s^!F), \qquad [\omega]\mapsto [\theta_{t,s}\circ \omega],
		\]
		for all $s,t \in I$.
		The construction uses three linear flows satisfying the hypothesis of Corollary \ref{cor_semi_direct}: the flow $\Psi$ on $A$ of $\widehat{[\partial_t\times 0, \cdot]}\in\mx^l(TI\times A)$, the flow $\Xi$ of $\widehat{\nabla_{X}}\in\mx^l(F)$ on $F$ and the smooth flow $\xi$ of $\partial_t\times 0\in \mx(I\times M)$ on $I\times M$. 
		%For $\lambda \in \R$, the restriction of $\Xi_{\lambda}$ at each time $s \in (I-\lambda)\cap I$ defines isomorphisms $\theta_{s,s+\lambda}\colon \iota_s^!F\to\iota_{s+\lambda}^! F$ of vector bundles. Corollary \ref{cor_semi_direct} implies that $\theta_{s,s+\lambda}$ for $\lambda=t-s$ is the desired gauge equivalence. 
	
	\medskip
	First consider the section $X=\partial_t\times 0$ of $TI\times A\to I\times M$.
	Its anchor is $\partial_t\times 0\in \mx(I\times M)$ with the flow \[\xi\colon \Omega:=\{(\lambda,s,m)\in \mathbb R\times I\times M\mid \lambda+s\in I\}\to I\times M,\] \[(\lambda, (t,m))\mapsto (t+\lambda, m).\]
	By construction, and using Lemma \ref{lemma_flat_sections}, the flow \[\Psi\colon \Omega^A:=\{(\lambda, v_s, a_m)\in \mathbb R\times TI\times A\mid \lambda+s\in I\}\to TI\times A\] of $\widehat{[X, \cdot]}\in\mx^l(TI\times A)$
	satisfies \begin{equation}\label{inv_psi_a}
	\Psi_\lambda^\star(0, a)=(0,a)
	\end{equation}
	on $((I-\lambda)\cap I)\times M$ for all $\lambda\in\mathbb R$ and all $a\in\Gamma(A)$, 
	since $[X, (0,a)]=0$ for all $a\in\Gamma(A)$, where $(0,a)$ is the section of $TI\times A\to I\times M$ given by $(s,m)\mapsto (0_s, a(m))$. 
	Hence $\Psi$ is given by
	\[ \Psi(\lambda, r\partial_t\arrowvert_s,a_m)=(r\partial_t\arrowvert_{s+\lambda}, a_m)
	\]
	for $(\lambda, r\partial_t\arrowvert_s,a_m)\in \Omega^A$.

Consider the flow  \[\Xi\colon \tilde \Omega=\{(\lambda, e_{(s,m)})\in \mathbb R\times F\mid \lambda+s\in I\}\to F\] of $\widehat{\nabla_{X}}\in\mx^l(F)$. 
For all $s, t\in I$, set 
	\[ \theta_{t,s}:=\Xi_{s-t}\arrowvert_{\iota_t^!F}\colon \iota_t^!F\to\iota_{s}^! F.
	\]
	Then the map $\theta_{t,s}$ is an isomorphism of vector bundles over $M$, since $\Xi_{s-t}$ is an isomorphism of vector bundles over $\xi_{s-t}$.

\medskip

Choose an arbitrary time $t\in I$. For each section $e\in\Gamma_M(\iota_t^!F)$ define 
$\tilde e\in\Gamma(F)$ by 
\[ \tilde e(s,m)=\Xi(s-t, e(m))=\theta_{t,s}(e(m)).
\]
Then for each $\lambda \in\mathbb R$, $s\in (I-\lambda)\cap I$ and $m\in M$
\[ (\Xi_\lambda^\star\tilde e)(s,m)=\Xi_{-\lambda}(\tilde e(s+\lambda, m))=\Xi(s+\lambda-t-\lambda, e(m))=\Xi(s-t, e(m))=\tilde e(s,m).
\]
This shows that \begin{equation}\label{eq_tilde_e}
\tilde e=\Xi_\lambda^\star\tilde e
\end{equation}
on $(I\cap(I-\lambda))\times M$. 

In addition, for all $s\in I$ and $m\in M$, compute 
\begin{equation}\label{eq:731}
(\iota_s^!\tilde e)(m)=\tilde e(s,m)=\Xi(s-t, e(m))=\theta_{t,s}(e(m)).
\end{equation}
Hence $\iota_s^!\tilde e=\theta_{t,s}\circ e$ for all $s \in I$. In particular,
\begin{equation}\label{eq:18}
\iota_t^!\tilde e=e.
\end{equation}
\medskip

As a consequence of the considerations above, according to Corollary \ref{cor_semi_direct}, for all $a\in\Gamma(A)$, all $e\in\Gamma(\iota_t^! F)$ and all $\lambda\in \mathbb R$
	\[ \nabla_{(0,a)}\tilde e=\nabla_{\Psi_\lambda^\star (0,a)}\Xi^\star_\lambda \tilde e=\Xi^\star_\lambda(\nabla_{(0,a)}\tilde e)
	\]
	on $(I\cap(I-\lambda))\times M$.
	This implies for $(s,m)\in I\times M$
	\begin{equation*}
	\begin{split}
	(\mathcal I_s^*\nabla)_{a(m)}(\theta_{t,s}\circ e)&=(\mathcal I_s^*\nabla)_{a(m)}(\iota_s^!\tilde e)
	=\nabla_{(0_s,a(m))}\tilde e\\
	&=\Xi^\star_{t-s}(\nabla_{(0,a)}\tilde e)(s,m)=\Xi_{s-t}\left(\nabla_{(0_{t}, a(m))}\tilde e\right)\\
	&=\Xi_{s-t}\left((\mathcal I_{t}^*\nabla)_{a(m)}\iota^!_{t}\tilde e\right)=\Xi_{s-t}\left((\mathcal I_{t}^*\nabla)_{a(m)}e\right),
	\end{split}
	\end{equation*}
	for all $a\in \Gamma(A)$ and $e\in\Gamma(\iota_t^!F)$.
	This is equivalent to 
	\begin{equation}\label{eq:732}
	\theta_{s,t}\circ (\mathcal I_s^*\nabla)\circ \theta_{t,s}=\mathcal I_{t}^*\nabla.
	\end{equation}
	Hence $\theta_{t,s}$ is a gauge equivalence from $\mathcal I_t^*\nabla$ to $\mathcal I_{s}^*\nabla$.
	
%	This shows that for each $s, t\in I$ the map 
%	\[ (\theta_{t,s})_*\colon H^\bullet_{\mathcal I_t^*\nabla}(A,\iota_t^!F)\to  H^\bullet_{\mathcal I_s^*\nabla}(A,\iota_s^!F), \qquad [\omega]\mapsto [\theta_{t,s}\circ \omega]
%	\]
%	is an isomorphism.
%	
	
\bigskip

\subsubsection{Step 2: Induced flat connections on the vector bundles $I\times \iota_t^!F\to I\times M$ and gauge equivalences between them.}\label{step2}
 One key idea is the fact that $\nabla$ is gauge equivalent to flat connections \[\nabla^r\colon \Gamma(TI\times A)\times \Gamma(I\times \iota_r^!F)\to \Gamma(I\times \iota_r^! F)\] 
 for all $r\in I$. This step uses the flow $\Xi$ to construct the isomorphism $F\simeq I\times \iota_r^!F$ and the induced connection on $I\times \iota_r^!F$ for all $r\in I$.
 
 \medskip
	Consider for each $t\in I$  the  map 
	\[ \Theta^t\colon  F\to I\times \iota_t^!F, \quad e \in F_{(r,m)}\mapsto (r,\Xi(t-r,e))=(r,\theta_{r,t}(e)).
	\]
	$\Theta^t$ is an isomorphism
	of vector bundles
	over the identity on $I\times M$.  Its inverse is the map
	\[ (\Theta^t)^{-1}\colon I\times \iota_t^!F\to F, \quad (r,e)\mapsto \Xi(r-t, e).
	\]
As a consequence  for $s, t\in I$, 
\begin{equation*}
\begin{split} \Theta^s\circ (\Theta^t)\inv (r,e)&=\Theta^s(\Xi(r-t, e))=(r, \Xi(s-r, \Xi(r-t, e)))\\
&=
(r, \Xi(s-t, e))=(r,\theta_{t,s}(e))
\end{split}
\end{equation*}
for all $r\in I$ and $e\in \iota_t^!F$. Denote this map by $\Theta_{t,s}\colon I\times \iota_t^!F\to I\times \iota_s^!F$,
\[\Theta_{t,s}=\Theta^s\circ (\Theta^t)\inv \colon (r,e)\mapsto (r, \theta_{t,s}e).
\]

	For each $t\in I$ the flat connection $\nabla$ is hence gauge equivalent to the flat connection $\nabla^t\colon \Gamma(TI\times A)\times \Gamma(I\times \iota_t^!F)\to \Gamma(I\times \iota_t^! F)$ defined by 
	\[ \nabla^t:=\Theta^t\circ\nabla\circ(\Theta^t)^{-1}
	\]
	and the cohomologies defined by $\nabla$ and $\nabla^t$ are isomorphic via
	\[ (\Theta^t)_*\colon H_\nabla(TI\times A, F)\to H_{\nabla^t}(TI \times A, I\times \iota_t^!F), \qquad [\omega]\mapsto [\Theta^t\circ \omega].
	\]
	By definition for $s,t\in I$ 
	\[\nabla^t=\Theta_{s,t}\circ \nabla^s\circ \Theta_{t,s}.
	\]
	Hence the cohomologies $H^\bullet_{\nabla^t}(TI\times A, I\times \iota_t^!F)$ and $H^\bullet_{\nabla^s}(TI\times A, I\times \iota_s^!F)$ are isomorphic via
	\[(\Theta_{t,s})_*\colon H^\bullet_{\nabla^t}(TI\times A, I\times \iota_t^!F)\to H^\bullet_{\nabla^s}(TI\times A, I\times \iota_s^!F), \qquad [\omega]\mapsto [\Theta_{t,s}\circ \omega].
	\]
	
This step has hence constructed the commutative diagrams of isomorphisms
% https://q.uiver.app/#q=WzAsNCxbNCwwLCJIX3tcXG5hYmxhfShUSVxcdGltZXMgQSwgRikiXSxbMywzLCJIX3tcXG5hYmxhXnR9KFRJXFx0aW1lcyBBLCBJXFx0aW1lcyBcXGlvdGFfdF4hRikiXSxbNSwzLCJIX3tcXG5hYmxhXnN9KFRJXFx0aW1lcyBBLCBJXFx0aW1lcyBcXGlvdGFfc14hRikiXSxbMCwxMF0sWzAsMSwiKFxcVGhldGFedClfKiIsMV0sWzAsMiwiKFxcVGhldGFecylfKiIsMV0sWzEsMiwiKFxcVGhldGFfe3Qsc30pXyoiLDFdXQ==
\[\begin{tikzcd}
	&&{H_{\nabla}(TI\times A, F)} \\
	\\
	\\
	& {H_{\nabla^t}(TI\times A, I\times \iota_t^!F)} && {H_{\nabla^s}(TI\times A, I\times \iota_s^!F)} 
	\arrow["{(\Theta^t)_*}"{description}, from=1-3, to=4-2]
	\arrow["{(\Theta^s)_*}"{description}, from=1-3, to=4-4]
	\arrow["{(\Theta_{t,s})_*}"{description}, from=4-2, to=4-4]
\end{tikzcd}\]
for all $s,t\in I$.
		
	\bigskip
	
	\subsubsection{Step 3: Equality of $\nabla^t$ with $\operatorname{Pr}_A^*\mathcal{I}_t^*\nabla$ for all $t\in I$.}\label{step3}

	 For each $e \in \Gamma_M(\iota_t^!F)$, recall that the section $\tilde{e}\in\Gamma_M(F)$ defined by $\tilde e(r,m)=\Xi(r-t,e(m))$ for all $(r,m)\in I\times M$ satisfies $\tilde e=\Xi_\lambda^\star\tilde e$ for $\lambda \in (I\cap(I-\lambda))\times M$, see \eqref{eq_tilde_e}, and $\iota_t^!\tilde e=e$, see \eqref{eq:18}. Moreover, define the section 
		$\bar e=\pr_M^!e\in\Gamma_{I\times M}(I \times \iota_t^!F)$ by $\bar e(r,m)=(r, e(m))$, for all $(r,m)\in I\times M$. Then
		\[ \left((\Theta^t)^{-1}\circ \bar e\right)(r,m)=\left(\Theta^t\right)^{-1}(r,e(m))=\Xi(r-t, e(m))=\tilde e(r,m)
		\]
		and 
		\begin{equation*}
		\begin{split}
		(\operatorname{pr}_M^!\iota_t^!\tilde{e})(r,m)&=(r,m,\iota_t^! \tilde{e}(m))
		=(r,m,(m,\tilde e(t,m)))\\
		&=(r,m,(m,e(m)))=(r,e(m))= \bar{e}(r,m)
		\end{split}
		\end{equation*}
		for all $(r,m)\in I\times M$, where the fourth equation is the canonical isomorphism $\pr_M^!\iota_t^!F\simeq I\times \iota_t^!F\to I\times M$. For $e \in \Gamma(\iota_t^!F)$, $\alpha\in \R$, $(r,m) \in I \times M$ and $a_m \in A\an{m}$, compute, using Corollary \ref{cor_semi_direct} and again the notation $X:=\partial_t\times 0\in\Gamma_{I\times M}(TI\times A)$,
				\begin{equation*}
			\begin{split}
				\nabla^t_{(\alpha\cdot \partial_t \an{r},a_m)}\bar{e}&\,\,\,\,\,\,= \Theta^t( \nabla_{(\alpha \cdot \partial_t \an{r},a_m)}\tilde{e})
				%=(r,\Xi(t-r, \nabla_{(\alpha \cdot \partial_t \an{r},a_m)}\tilde{e}))
				=(r,\Xi(t-r, \alpha \cdot \nabla_{X(r,m)}\tilde{e}+ \nabla_{(0_r,a_m)}\tilde{e}))\\
&\,\,\,\,\,\,			= (r, \alpha\cdot \Xi(t-r,\nabla_{X}\tilde{e}\an{(r,m)})+ \Xi(t-r,\nabla_{(0,a_m)}\tilde{e}\an{(r,m)}))\\
				&\,\,\,\,\,\,= \left(r,\alpha\cdot\Xi_{t-r}(\nabla_X\tilde{e} \an{(r-t+t,m)})  +
				\Xi_{t-r}(\nabla_{(0,a)}\tilde{e} \an{(r-t+t,m)})\right)\\&\,\,\,\,\,\,=
				\left(r,\alpha\cdot(\Xi^*_{r-t}(\nabla_X \tilde{e}))\an{(t,m)} + (\Xi_{r-t}^*(\nabla_{(0,a)} \tilde{e})) \an{(t,m)}\right)\\&
				\stackrel{\text{Cor.}\,\ref{cor_semi_direct}}{=} 
				\left(r,\alpha\cdot(\nabla_{\Psi^*_{r-t}X}(\Xi^*_{r-t}\tilde{e}))\an{(t,m)} + (\nabla_{\Psi^*_{r-t}(0,a)}(\Xi^*_{r-t}\tilde{e}))\an{(t,m)}\right)\\
				&\stackrel{\eqref{inv_psi_a},\eqref{eq_tilde_e}}{=} 
				\left(r,\alpha \cdot \nabla_X \tilde{e} \an{(t,m)} + \nabla_{(0,a)} \tilde{e} \an{(t,m)} \right)= (r, \nabla_{(\alpha\cdot \partial_t \an{t},a_m)} \tilde{e}).
			\end{split}
	    \end{equation*}
	  On the other hand,	    \begin{equation*}
	    	\begin{split}
	    		(\operatorname{Pr}_A^*(\mathcal I_t^*\nabla))_{(\alpha \cdot \partial_t \an{r}, a_m)} \bar e&=(\operatorname{Pr}_A^*(\mathcal I_t^*\nabla))_{(\alpha \cdot \partial_t \an{r}, a_m)}(\operatorname{pr}_M^!\iota_t^!\tilde{e})=(r,(\mathcal I_t^*\nabla)_{a_m}(\iota^!_t \tilde{e}))\\
	    		&= (r,\nabla_{(0_t,a_m)} \tilde{e}).
	    	\end{split}
	    \end{equation*}
	   % Show that $\nabla_{(r \cdot \partial_t \an{t},a_m)} \tilde{e}= \nabla_{(0_t,a_m)} \tilde{e}$.
	By Lemma \ref{lemma_flat_sections}  for all $(r,m)\in I\times M$
	    \begin{equation*}
	    	\begin{split}
	    		\nabla_{(\partial_t \an{r},0_m)}\tilde{e}=\left.\frac{d}{dt}\right\an{t=0} (\Xi_t^* \tilde{e}(r,m))\overset{\eqref{eq_tilde_e}
}{=}\left.\frac{d}{dt}\right\an{t=0}\tilde{e}(r,m)=0,%=\nabla_{(0_r,0_m)}\tilde{e},
	        \end{split}
	    \end{equation*}
	    and so
	    \begin{equation*}
	    	\begin{split}
	    		\nabla_{(\alpha\cdot \partial_t \an{r},a_m)}\tilde{e}= \alpha \cdot \nabla_{(\partial_t \an{r},0_m)}\tilde{e}+ \nabla_{(0_r,a_m)}\tilde{e}=\nabla_{(0_r,a_m)}\tilde{e},
	    	\end{split}
	    \end{equation*}
	    which shows in particular 
	    \[\nabla^t_{(\alpha\cdot \partial_t \an{r},a_m)}\bar e=(r, \nabla_{(\alpha\cdot \partial_t \an{t},a_m)} \tilde{e})=(r, \nabla_{(0_t,a_m)}\tilde{e})=
	    (\operatorname{Pr}_A^*(\mathcal I_t^*\nabla))_{(\alpha \cdot \partial_t \an{r}, a_m)} \bar e.
	    \]
	Since the sections $\bar e\in\Gamma_{I\times M}(I\times\iota_t^!F)$ defined by sections $e\in\Gamma_M(\iota_t^!F)$ generate $\Gamma_{I\times M}(I\times\iota_t^!F)$ as a $C^\infty(I\times M)$-module, this proves the equality
	    \[\nabla^t=\operatorname{Pr}_A^*(\mathcal I_t^*\nabla)
	    \]
	   and therefore % https://q.uiver.app/#q=WzAsNSxbMywxLCJIX3tcXFBwcl9BXipcXG1hdGhjYWwgSV90XipcXG5hYmxhfShUSVxcdGltZXMgQSwgSVxcdGltZXMgXFxpb3RhX3ReIUYpIl0sWzUsMSwiSF97XFxQcHJfQV4qXFxtYXRoY2FsIElfc14qXFxuYWJsYX0oVElcXHRpbWVzIEEsIElcXHRpbWVzIFxcaW90YV9zXiFGKSJdLFszLDAsIkhfe1xcbmFibGFedH0oVElcXHRpbWVzIEEsIElcXHRpbWVzIFxcaW90YV90XiFGKSJdLFs1LDAsIkhfe1xcbmFibGFec30oVElcXHRpbWVzIEEsIElcXHRpbWVzIFxcaW90YV9zXiFGKSJdLFswLDJdLFsyLDMsIihcXFRoZXRhX3t0LHN9KV8qIiwxXSxbMCwxLCIoXFxUaGV0YV97dCxzfSlfKiIsMV0sWzMsMSwiIiwxLHsic3R5bGUiOnsiaGVhZCI6eyJuYW1lIjoibm9uZSJ9fX1dLFszLDEsIiIsMSx7ImxhYmVsX3Bvc2l0aW9uIjo0MCwib2Zmc2V0IjotMSwic3R5bGUiOnsiaGVhZCI6eyJuYW1lIjoibm9uZSJ9fX1dLFsyLDAsIiIsMSx7InN0eWxlIjp7ImhlYWQiOnsibmFtZSI6Im5vbmUifX19XSxbMiwwLCIiLDEseyJvZmZzZXQiOi0xLCJzdHlsZSI6eyJoZWFkIjp7Im5hbWUiOiJub25lIn19fV1d
\[\begin{tikzcd}
	& {H_{\nabla^t}(TI\times A, I\times \iota_t^!F)} && {H_{\nabla^s}(TI\times A, I\times \iota_s^!F)} \\
	& {H_{\Ppr_A^*\mathcal I_t^*\nabla}(TI\times A, I\times \iota_t^!F)} && {H_{\Ppr_A^*\mathcal I_s^*\nabla}(TI\times A, I\times \iota_s^!F)}
	\arrow["{(\Theta_{t,s})_*}"{description}, from=1-2, to=1-4]
	\arrow[no head, from=1-2, to=2-2]
	\arrow[shift left, no head, from=1-2, to=2-2]
	\arrow[no head, from=1-4, to=2-4]
	\arrow[shift left, no head, from=1-4, to=2-4]
	\arrow["{(\Theta_{t,s})_*}"{description}, from=2-2, to=2-4]
\end{tikzcd}\]
for all $s,t\in I$.
This also shows that 
 % https://q.uiver.app/#q=WzAsNSxbMiwyLCJIX3tcXG1hdGhjYWwgSV90XipcXG5hYmxhfShBLFxcaW90YV90XiFFKSJdLFsyLDAsIkhfe1xcbmFibGF9KFRJXFx0aW1lcyBBLCBGKSJdLFs0LDIsIkhfe1xcUHByX0FeKlxcbWF0aGNhbCBJX3ReKlxcbmFibGF9KFRJXFx0aW1lcyBBLCBJXFx0aW1lcyBcXGlvdGFfdF4hRikiXSxbNCwwLCJIX3tcXG5hYmxhXnR9KFRJXFx0aW1lcyBBLCBJXFx0aW1lcyBcXGlvdGFfdF4hRikiXSxbMCwzXSxbMCwyLCJcXG92ZXJsaW5le1xcUHByX0FeKn0iLDEseyJzdHlsZSI6eyJ0YWlsIjp7Im5hbWUiOiJob29rIiwic2lkZSI6InRvcCJ9LCJoZWFkIjp7Im5hbWUiOiJlcGkifX19XSxbMSwzLCIoXFxUaGV0YV50KV8qIiwxXSxbMSwwLCJcXG92ZXJsaW5le1xcbWF0aGNhbCBJX3ReKn0iLDFdLFszLDIsIiIsMSx7InN0eWxlIjp7ImhlYWQiOnsibmFtZSI6Im5vbmUifX19XSxbMywyLCIiLDEseyJvZmZzZXQiOi0xLCJzdHlsZSI6eyJoZWFkIjp7Im5hbWUiOiJub25lIn19fV1d
\[\begin{tikzcd}
	&& {H_{\nabla}(TI\times A, F)} && {H_{\nabla^t}(TI\times A, I\times \iota_t^!F)} \\
	\\
	&& {H_{\mathcal I_t^*\nabla}(A,\iota_t^!F)} && {H_{\Ppr_A^*\mathcal I_t^*\nabla}(TI\times A, I\times \iota_t^!F)}
	\arrow["{(\Theta^t)_*}"{description}, from=1-3, to=1-5]
	\arrow["{\overline{\mathcal I_t^*}}"{description}, from=1-3, to=3-3]
	\arrow[no head, from=1-5, to=3-5]
	\arrow[shift left, no head, from=1-5, to=3-5]
	\arrow["{\overline{\Ppr_A^*}}"{description}, hook, two heads, from=3-3, to=3-5]
\end{tikzcd}\]
commutes for all $t\in I$.
	    \bigskip
	    
	    \subsubsection{Step 4: Commutativity of the diagram induced from Theorem \ref{th:4}.}
	    This step checks the isomorphisms in the obtained diagram 
\begin{equation}\label{big_diag}{\footnotesize
\begin{tikzcd}
	&& {H_{\nabla}(TI\times A, F)} \\
	\\
	\\
	\\
	& {H_{\nabla^t}(TI\times A, I\times \iota_t^!F)} && {H_{\nabla^s}(TI\times A, I\times \iota_s^!F)} \\
	& {H_{\Ppr_A^*\mathcal I_t^*\nabla}(TI\times A, I\times \iota_t^!F)} && {H_{\Ppr_A^*\mathcal I_s^*\nabla}(TI\times A, I\times \iota_s^!F)} \\
	 {H_{\mathcal I_t^*\nabla}(A,\iota_t^!F)} &&&&& {H_{\mathcal I_s^*\nabla}(A,\iota_s^!F)}
	\arrow["{(\Theta^t)_*}"{description}, from=1-3, to=5-2]
	\arrow["{(\Theta^s)_*}"{description}, from=1-3, to=5-4]
	\arrow["{\overline{\mathcal I_t^*}}"{description}, curve={height=30pt}, from=1-3, to=7-1]
	\arrow["{\overline{\mathcal I_s^*}}"{description}, curve={height=-30pt}, from=1-3, to=7-6]
	\arrow["{(\Theta_{t,s})_*}"{description}, hook, two heads, from=5-2, to=5-4]
	\arrow[no head, from=5-2, to=6-2]
	\arrow[shift left, no head, from=5-2, to=6-2]
	\arrow[no head, from=5-4, to=6-4]
	\arrow[shift left, no head, from=5-4, to=6-4]
	\arrow["{(\Theta_{t,s})_*}"{description}, hook, two heads, from=6-2, to=6-4]
	\arrow["{\overline{\Ppr_A^*}}"{description}, hook, two heads, from=7-1, to=6-2]
	\arrow["{(\theta_{t,s})_*}"{description}, shift right=2, curve={height=30pt}, hook, two heads, from=7-1, to=7-6]
	\arrow["{\overline{\Ppr_A^*}}"{description}, hook, two heads, from=7-6, to=6-4]
\end{tikzcd}}\end{equation}
\noindent and its commutativity, where the bottom arrows $\overline{\Ppr_A^*}$ are the isomorphisms established in Theorem \ref{th:4}.
By Step 1 (\S\ref{step1}) the bottom arrow is an isomorphism. By Step 2 (\S\ref{step2}), the top central triangle commutes, and by Step 3 (\S\ref{step3}) the central square and the left and right sides commute.
	   	    
		Hence it remains to prove that the bottom central square 
       % https://q.uiver.app/#q=WzAsNSxbMCwzLCJIX3tcXG1hdGhjYWwgSV90XipcXG5hYmxhfShBLFxcaW90YV90XiFFKSJdLFsyLDMsIkhfe1xcbWF0aGNhbCBJX3NeKlxcbmFibGF9KEEsXFxpb3RhX3NeIUUpIl0sWzAsMSwiSF97XFxQcHJfQV4qXFxtYXRoY2FsIElfdF4qXFxuYWJsYX0oVElcXHRpbWVzIEEsIElcXHRpbWVzIFxcaW90YV90XiFFKSJdLFsyLDEsIkhfe1xcUHByX0FeKlxcbWF0aGNhbCBJX3NeKlxcbmFibGF9KFRJXFx0aW1lcyBBLCBJXFx0aW1lcyBcXGlvdGFfc14hRSkiXSxbMSwwXSxbMCwyLCJcXG92ZXJsaW5le1xcUHByX0FeKn0iLDEseyJzdHlsZSI6eyJ0YWlsIjp7Im5hbWUiOiJob29rIiwic2lkZSI6ImJvdHRvbSJ9LCJoZWFkIjp7Im5hbWUiOiJlcGkifX19XSxbMSwzLCJcXG92ZXJsaW5le1xcUHByX0FeKn0iLDEseyJzdHlsZSI6eyJ0YWlsIjp7Im5hbWUiOiJob29rIiwic2lkZSI6ImJvdHRvbSJ9LCJoZWFkIjp7Im5hbWUiOiJlcGkifX19XSxbMCwxLCIoXFx0aGV0YV97dCxzfSlfKiIsMSx7InN0eWxlIjp7InRhaWwiOnsibmFtZSI6Imhvb2siLCJzaWRlIjoidG9wIn0sImhlYWQiOnsibmFtZSI6ImVwaSJ9fX1dLFsyLDMsIihcXFRoZXRhX3t0LHN9KV8qIiwxXV0=
\[\begin{tikzcd}
	{H_{\Ppr_A^*\mathcal I_t^*\nabla}(TI\times A, I\times \iota_t^!E)} && {H_{\Ppr_A^*\mathcal I_s^*\nabla}(TI\times A, I\times \iota_s^!E)} \\
	\\
	{H_{\mathcal I_t^*\nabla}(A,\iota_t^!E)} && {H_{\mathcal I_s^*\nabla}(A,\iota_s^!E)}
	\arrow["{(\Theta_{t,s})_*}"{description}, from=1-1, to=1-3]
	\arrow["{\overline{\Ppr_A^*}}"{description}, hook', two heads, from=3-1, to=1-1]
	\arrow["{(\theta_{t,s})_*}"{description}, hook, two heads, from=3-1, to=3-3]
	\arrow["{\overline{\Ppr_A^*}}"{description}, hook', two heads, from=3-3, to=1-3]
\end{tikzcd}\]
%        commutes for all $s,t\in I$. But this is simply due to
%        \[ (\theta_{t,s}\circ \pr_M)(t,e_m)=\theta_{t,s}(e_m)=\Ppr_A(
%        \]
%        for all $s, t\in I$, $m\in M$ and $e_m\in \iota_t^!F\an{m}$
%        i.e.
%        \[
%        \]
%     %        
%%               
%%        
        This commutativity happens at the level of forms: Take $\omega\in\Omega^k(A, \iota_t^!F)$ and compute for $(r,m)\in I\times M$ and $(v_1,a_1), \ldots, (v_k,a_k)\in T_rI\times A_m$
        \begin{equation*}
        	\begin{split}
        		&(\Ppr_A^*(\theta_{t,s}\circ \omega))_{(r,m)}((v_1,a_1),\ldots, (v_k,a_k))=(r, \theta_{t,s}(\omega_m(a_1,\ldots,a_k))\\
        		&=\Theta_{t,s}(r, \omega_m(a_1,\ldots,a_k))=\Theta_{t,s}\left((\Ppr_A^*\omega)_{(r,m)}((v_1,a_1),\ldots, (v_k,a_k))\right)\\
        		&=(\Theta_{t,s}\circ(\Ppr_A^*\omega))_{(r,m)}((v_1,a_1),\ldots, (v_k,a_k)).
        	\end{split}
        \end{equation*}

%        Finally, the side diagrams commute because $\Theta^t\colon e \in F_{(r,m)}\mapsto (r,\Xi(t-r,e))$ restricts to the identity map in $F\an{t\times M}$, for each $t\in I$ ($\Xi$ is a flow, so it is the identity at time zero).
%        In detail, for each closed $\omega \in \Omega^k(TI\times A,F)$, the computation
%         \begin{equation*}
%         	\begin{split}
%         		&(\mathcal{I}_t^*\omega)_m(a_1, \dots, a_k):= (m,\omega_{(t,m)}((0_t,a_1), \dots, (0_t,a_k)))\\& =(m, \Theta^t \circ \omega_{(t,m)}((0_t,a_1), \dots, (0_t,a_k))= \mathcal{I}_t^*(\Theta^t \circ \omega)_m(a_1,\dots,a_k)
%         	\end{split}
%         \end{equation*}
%        implies
%        \begin{equation*}
%        	\begin{split}
%       	\overline{\mathcal{I}_{t,}^*}_{\nabla} [\omega]=[\mathcal{I}_t^*\omega]=
%       		[\mathcal{I}_t^*(\Theta^t \circ \omega)]= \overline{\mathcal{I}_{t}} \circ (\Theta^t)_* [\omega].
%            \end{split}
%       	\end{equation*}
%	
%	
\medskip
	
	So the commutativity of \eqref{big_diag} is proved. It immediately implies that the maps $\overline{\mathcal{I}_{r}^*}\colon \tH^{\bullet}_{\nabla}(TI \times A, F) \to \tH^{\bullet}_{\mathcal{I}_r^*\nabla}(A,\iota^!_rF)$ defined by the inclusion maps $\mathcal{I}_r\colon A \hookrightarrow TI \times A$ are isomorphisms for all $r\in I$ and that  the exterior triangle
% https://q.uiver.app/#q=WzAsNCxbMywxLCJIX3tcXG1hdGhjYWwgSV90XipcXG5hYmxhfShBLFxcaW90YV90XiFFKSJdLFs0LDAsIkhfe1xcbmFibGF9KFRJXFx0aW1lcyBBLCBGKSJdLFs1LDEsIkhfe1xcbWF0aGNhbCBJX3NeKlxcbmFibGF9KEEsXFxpb3RhX3NeIUYpIl0sWzAsNV0sWzAsMiwiKFxcdGhldGFfe3Qsc30pXyoiLDEseyJzdHlsZSI6eyJ0YWlsIjp7Im5hbWUiOiJob29rIiwic2lkZSI6InRvcCJ9LCJoZWFkIjp7Im5hbWUiOiJlcGkifX19XSxbMSwyLCJcXG92ZXJsaW5le1xcbWF0aGNhbCBJX3NeKn0iLDFdLFsxLDAsIlxcb3ZlcmxpbmV7XFxtYXRoY2FsIElfdF4qfSIsMV1d
\[\begin{tikzcd}
	& {H_{\nabla}(TI\times A, F)} \\
	 {H_{\mathcal I_t^*\nabla}(A,\iota_t^!F)} && {H_{\mathcal I_s^*\nabla}(A,\iota_s^!F)}
	\arrow["{\overline{\mathcal I_t^*}}"{description}, from=1-2, to=2-1]
	\arrow["{\overline{\mathcal I_s^*}}"{description}, from=1-2, to=2-3]
	\arrow["{(\theta_{t,s})_*}"{description}, hook, two heads, from=2-1, to=2-3]
\end{tikzcd}\]
commutes for all $s,t\in I$.

        \bigskip

\subsubsection{Step 5: Application of step 4 to the pullback under the homotopy.}
The last step of the proof is a standard argument. Consider a LA-homotopy $\Phi\colon TI\times A\to B$ over $\phi\colon I\times M\to N$ and a flat $B$-connection on $E\to N$. Apply the results above to the flat $(TI\times A)$-connection $\Phi^*\nabla$ on $\phi^!E$ and use the commutative diagrams 
%https://q.uiver.app/#q=WzAsNCxbMCwyLCJIX1xcbmFibGEoQixFKSJdLFswLDQsIkhfe1xcUGhpXipcXG5hYmxhfShUSVxcdGltZXMgQSwgXFxwaGleIUUpIl0sWzcsNCwiSF97XFxQaGlfc14qXFxuYWJsYX0oQSxcXHBoaV9zXiFFKSJdLFs1LDBdLFswLDEsIlxcb3ZlcmxpbmV7XFxQaGleKn0iLDFdLFswLDIsIlxcb3ZlcmxpbmV7XFxQaGlfc14qfSIsMl0sWzEsMiwiXFxvdmVybGluZXtcXG1hdGhjYWwgSV9zXip9IiwxXV0=
\[\begin{tikzcd}
	{H_\nabla(B,E)} \\
	\\
	{H_{\Phi^*\nabla}(TI\times A, \phi^!E)} &&&&&&& {H_{\Phi_r^*\nabla}(A,\phi_r^!E)}
	\arrow["{\overline{\Phi^*}}"{description}, from=1-1, to=3-1]
	\arrow["{\overline{\Phi_r^*}}"', from=1-1, to=3-8]
	\arrow["{\overline{\mathcal{I}_{r}^*}}"{description}, from=3-1, to=3-8]
\end{tikzcd}\]
for $r=t,s\in I$ in order to obtain the following commutative diagram and deduce from it that $(\theta_{t,s})_*\circ \overline{\Phi_t^*}=\overline{\Phi_s^*}$.
The statement of Theorem \ref{th:hi} is then obtained by setting $t=0$, $s=1$ and $\theta:=\theta_{0,1}$.
% https://q.uiver.app/#q=WzAsNSxbMywzLCJIX3tcXFBoaV90XipcXG5hYmxhfShBLFxcdmFycGhpX3ReIUUpIl0sWzQsMiwiSF97XFxQaGleKlxcbmFibGF9KFRJXFx0aW1lcyBBLFxcdmFycGhpXiFFKSJdLFs1LDMsIkhfe1xcUGhpX3NeKlxcbmFibGF9KEEsXFx2YXJwaGlfc14hRikiXSxbMCw3XSxbNCwwLCJIX1xcbmFibGEoQixFKSJdLFswLDIsIihcXHRoZXRhX3t0LHN9KV8qIiwxLHsic3R5bGUiOnsidGFpbCI6eyJuYW1lIjoiaG9vayIsInNpZGUiOiJ0b3AifSwiaGVhZCI6eyJuYW1lIjoiZXBpIn19fV0sWzEsMiwiXFxvdmVybGluZXtcXG1hdGhjYWwgSV9zXip9IiwxXSxbMSwwLCJcXG92ZXJsaW5le1xcbWF0aGNhbCBJX3ReKn0iLDFdLFs0LDEsIlxcb3ZlcmxpbmV7XFxQaGleKn0iLDFdLFs0LDIsIlxcb3ZlcmxpbmV7XFxQaGleKl9zfSIsMSx7ImN1cnZlIjotM31dLFs0LDAsIlxcb3ZlcmxpbmV7XFxQaGleKl90fSIsMSx7ImN1cnZlIjozfV1d
\[\begin{tikzcd}
	& {H_\nabla(B,E)} \\
	\\
& {H_{\Phi^*\nabla}(TI\times A,\varphi^!E)} \\
	 {H_{\Phi_t^*\nabla}(A,\varphi_t^!E)} && {H_{\Phi_s^*\nabla}(A,\varphi_s^!F)} \\
	\\
	\\
	\\
	{}
	\arrow["{\overline{\Phi^*}}"{description}, from=1-2, to=3-2]
	\arrow["{\overline{\Phi^*_t}}"{description}, curve={height=18pt}, from=1-2, to=4-1]
	\arrow["{\overline{\Phi^*_s}}"{description}, curve={height=-18pt}, from=1-2, to=4-3]
	\arrow["{\overline{\mathcal I_t^*}}"{description}, from=3-2, to=4-1]
	\arrow["{\overline{\mathcal I_s^*}}"{description}, from=3-2, to=4-3]
	\arrow["{(\theta_{t,s})_*}"{description}, hook, two heads, from=4-1, to=4-3]
\end{tikzcd}\]

%\bibliographystyle{alpha}
%\bibliography{biblio}

\def\cprime{$'$} \def\polhk#1{\setbox0=\hbox{#1}{\ooalign{\hidewidth
			\lower1.5ex\hbox{`}\hidewidth\crcr\unhbox0}}} \def\cprime{$'$}
\def\cprime{$'$}

\end{document}